\documentclass[12pt,a4paper]{amsart}
\usepackage{amssymb}
\usepackage{amsfonts}
\usepackage{amsmath}
\usepackage[mathscr]{eucal}
\usepackage[all,cmtip]{xy}
\usepackage{comment}
\usepackage{color}

\newtheorem{thm}{Theorem}[section]
\newtheorem{prop}[thm]{Proposition}
\newtheorem{lem}[thm]{Lemma}
\newtheorem{cor}[thm]{Corollary}
\numberwithin{equation}{section}

\newtheorem{ex}[thm]{Example}
\newtheorem{assumption}[thm]{Assumption}

\theoremstyle{definition}
\newtheorem{rem}[thm]{Remark}
\newtheorem{defn}[thm]{Definition}

\def\Z{{\Bbb Z}}
\def\Q{{\Bbb Q}}
\def\R{{\Bbb R}}
\def\C{{\Bbb C}}
\def\A{{\Bbb A}}

\def\ff{{\frak f}}

\def\fn{{\frak n}}

\def\fq{{\frak q}}

\def\GL{{\operatorname {GL}}}
\def\GU{{\operatorname{GU}}}
\def\SL{{\operatorname{SL}}}
\def\Sp{{\operatorname{Sp}}}
\def\SO{{\operatorname{SO}}}
\def\U{{\operatorname {U}}}

\def\PGL{{\operatorname{PGL}}}
\def\GSp{{\operatorname{GSp}}}
\def\PGSp{{\operatorname{PGSp}}}

\def\Im{{\operatorname {Im}}}
\def\tr{{\operatorname{tr}}}
\def\nr{{\operatorname{N}}}

\def\End{{\operatorname{End}}}

\def\Hom{{\rm{Hom}}}
\def\diag{{\operatorname {diag}}}
\def\sgn{{\operatorname {sgn}}}

\def\leq{\leqslant}
\def\geq{\geqslant}
\def\bsl{\backslash}
\def\le{\leq}
\def\ge{\geq}

\def\d {{{d}}}

\renewcommand{\a}{\alpha}
\renewcommand{\b}{\beta}

\newcommand{\e}{\epsilon}

\renewcommand{\l}{\lambda}

\newcommand{\s}{\sigma}

\newcommand{\ga}{{\mathfrak{a}}}

\newcommand{\gf}{{\mathfrak{f}}}
\newcommand{\gm}{{\mathfrak{m}}}
\newcommand{\gn}{{\mathfrak{n}}}

\newcommand{\gp}{{\mathfrak{p}}}
\newcommand{\gq}{{\mathfrak{q}}}

\newcommand{\gA}{{\mathfrak{A}}}

\newcommand{\Acal}{{\mathcal A}}

\newcommand{\Dcal}{{\mathcal D}}
\newcommand{\Ecal}{{\mathcal E}}
\newcommand{\Fcal}{{\mathcal F}}
\newcommand{\Gcal}{{\mathcal G}}

\newcommand{\Ical}{{\mathcal I}}

\newcommand{\Lcal}{{\mathcal L}}
\newcommand{\Mcal}{{\mathcal M}}

\newcommand{\Ocal}{{\mathcal O}}

\newcommand{\Ucal}{{\mathcal U}}

\renewcommand{\AA}{\mathbb{A}}

\newcommand{\CC}{\mathbb{C}}

\newcommand{\GG}{\mathbb{G}}

\newcommand{\NN}{\mathbb{N}}

\newcommand{\QQ}{\mathbb{Q}}
\newcommand{\RR}{\mathbb{R}}

\newcommand{\ZZ}{\mathbb{Z}}

\newcommand{\bfk}{{\mathbf k}}

\newcommand{\bfA}{{\mathbf A}}

\newcommand{\bfG}{{\mathbf G}}

\newcommand{\bfK}{{\mathbf K}}
\newcommand{\bfL}{{\mathbf L}}

\newcommand{\bfT}{{\mathbf T}}

\newcommand{\Aut}{\operatorname{Aut}}

\newcommand{\Gal}{\operatorname{Gal}}
\newcommand{\Ker}{{\operatorname{Ker}}}

\newcommand{\notdivide}{\nmid}

\newcommand{\ord}{\operatorname{ord}}

\newcommand{\id}{\operatorname{id}}

\newcommand{\Spec}{\operatorname{Spec}}

\newcommand{\fin}{{\rm fin}}

\newcommand{\Lie}{{\operatorname{Lie}}}
\newcommand{\Res}{\operatorname{Res}}
\newcommand{\AI}{\operatorname{AI}}
\newcommand{\BC}{\operatorname{BC}}
\newcommand{\HE}{\operatorname{HE}}

\def\rec{{\rm rec}}
\def\Ind{{\rm Ind}}
\def\Sym{{\rm Sym}}

\title[Integrality of Hecke eigenvalues and the growth of Hecke fields]
{Integrality of Hecke eigenvalues and the growth of Hecke fields}

\author{Kenji Sakugawa}
\address{Faculty of Education, Shinshu University, 6-Ro, Nishi-nagano, Nagano 380-8544, Japan} 
\email{sakugawa\_kenji@shinshu-u.ac.jp}

\author{Shingo Sugiyama}
\address{Faculty of Mathematics and Physics, Institute of Science and
	Engineering, Kanazawa University, Kakumamachi, Kanazawa, Ishikawa 920-1192, Japan}
\email{s-sugiyama@se.kanazawa-u.ac.jp}


\begin{document}

	\maketitle
	
	\begin{abstract}
		We prove that Hecke eigenvalues for any Hilbert and Siegel modular forms
		are algebraic integers.
		Our method does not rely on cohomologicality nor Galois representations.
		We apply the integrality of Hecke eigenvalues for Hilbert modular forms of non-parallel weight to the estimation of the growth of Hecke fields of Hilbert cusp forms with non-vanishing central $L$-values.
		As a further application,
		we give the growth of the fields of rationality of cuspidal automorphic representations of $\GL_{2d}(\AA_\QQ)$ for a prime number $d$ with non-vanishing central $L$-values.
		We also apply the integrality of Hecke eigenvalues for holomorphic Siegel cusp forms of general degree in order to give the growth of the Hecke fields of those forms.
		
	\end{abstract}
	
	\setcounter{tocdepth}{3}

	\section{Introduction}
	Hecke operators and their eigenvalues play pivotal roles
	in number theory connecting arithmetic objects with automorphic ones.
	It is well known that Hecke eigenvalues for holomorphic elliptic modular forms
	are algebraic integers.
	Shimura \cite{Shimura} gave such integrality of Hecke eigenvalues for holomorphic Hilbert modular forms
	of parallel weight.
	Integrality can be applicable and indispensable to estimate the growth of Hecke fields of elliptic modular forms
	as in Serre \cite{Serre}.
	Although such integrality was proved in special settings, 
	there is no literature on such integrality in a general setting.
	For example, the case of Hilbert modular forms of {\it non-parallel and non-paritious} weight, and of {\it vector-valued} Siegel modular forms
	of {\it level $\Gamma(N)$} (the principal congruence subgroup) have not been treated.
	
	In this paper, we give the integrality of Hecke eigenvalues for holomorphic Hilbert modular forms and holomorphic Siegel modular forms in a general setting.
	Moreover,
	we apply our integrality results to the growth of the fields of rationality of cuspidal automorphic representations of $\GL_{2d}(\AA_\QQ)$ 
	with a prime number $d$ and of $\Sp_{2n}(\AA_\QQ)$ for any $n\ge 2$.

	\subsection{Case of Hilbert modular forms}
	Let $F$ be a totally real algebraic number field of finite degree and $\Ocal_F$ the ring of integers of $F$.
	Let ${\bf k}=(k_v)_{v | \infty}$ be a family of positive integers indexed by the set of
	the archimedean places of $F$ and let $\gn$ be a non-zero ideal of $\Ocal_F$. Take a character $\chi : (\Ocal_F/\gn)^\times\rightarrow \CC^\times$.
	We write $M_{\bf k}(\gn, \chi)$ for the space of 
	Hilbert modular forms of weight $\bf k$ with respect to the congruence subgroup $\Gamma_0(\gn)$ of level $\gn$ 
	and nebentypus $\chi$ (see \cite[($2.5_c$)]{Shimura}).
	It is well known that Hecke eigenvalues of Hilbert modular forms of weight $\bf k$
	are algebraic numbers. Moreover,
	if $\bf k$ is parallel, then
	Shimura proved in \cite[Proposition 2.2]{Shimura}
	that those Hecke eigenvalues are algebraic integers
	by constructing a basis of $M_{\bf k}(\gn, \chi)$ with the aid of Siegel modular forms with integral coefficients (see \cite[(9)]{Shimura1975} and \cite[p.683]{ShimuraTheta}).
	Such integrality was used to analyze the degrees of Hecke fields of Hilbert (or elliptic) cusp forms (cf.\ \cite{Serre},
	\cite{Royer} and \cite{SugiyamaTsuzuki}).

	Let $T'(\gp)$ denote the Hecke operator on $M_{\bf k}(\gn,\chi)$ for a non-zero prime ideal $\gp$ of $\Ocal_F$ coprime to $\gn$, which is defined in \cite[(2.21)]{Shimura}.
	Let $\bfT_{F,\ZZ}^{{\bf k}, \gn, \chi}$ be the $\ZZ$-subalgebra of ${\rm End}(M_{\bf k}(\gn, \chi))$
	generated by $T'(\gp)$ over all non-zero prime ideals $\gp$ coprime to $\gn$.
	\begin{thm}\label{thm0.1}
		All eigenvalues of any Hecke operator $T\in {\bf T}_{F, \ZZ}^{{\bf k},\gn,\chi}$ on $M_{\bf k}(\gn, \chi)$ are algebraic integers.
	\end{thm}
	
	We remark that the proof of Theorem \ref{thm0.1} needs neither cohomologicality
	nor cuspidality of $f$.
	Integrality could be proved by the use of the Eichler-Shimura isomorphism and Betti cohomology
	if $\bf k$ satisfies $k_v\ge 2$ for all $v|\infty$ and all $k_v$'s have the same parity (cf.\ \cite[Definition 1.1 and Theorem 1.14]{Dimitrov2}, \cite{Hida94}).
	One of advantages of our method is to treat the non-paritious case, where the weight $\bf k$ has both even and odd components.
	As such non-paritious Hilbert modular forms are not cohomological,
	cohomological arguments do not work for integrality.
	
	Certain paritious Hilbert modular forms with properties such as ``cohomological'', ``of parallel weight one'', etc., have associated Galois representations.
We can give the integrality of Hecke eigenvalues in such paritious cases (see Remarks \ref{known for cohomological Hilbert modular} and \ref{known for parallel weight one}).
	
	We prove Theorem \ref{thm0.1} by reducing this to the existence of an integral structure where Hecke operators act.
	A candidate of the integral structure is given by Katz \cite{Katz}. However, it should be noted that our modular forms are different from Katz modular forms. Therefore, even if we use results of \cite{Katz}, it is not obvious whether an integral structure exists. Hence,
	for the existence of an integral structure,
	it is necessary to utilize the existence of a \emph{rational structure} proved by Shimura \cite{Shimura}. The Hecke-stability of this integral structure can be shown by explicit relations between Hecke operators and Fourier coefficients given in \cite{Shimura}.

	\subsection{Case of Siegel modular forms}
	
	We consider Siegel modular forms of degree $n\ge2$.
	Let $\rho$ be an irreducible algebraic representation of $\GL_n(\CC)$ with highest weight $(k_1,\ldots, k_n) \in \ZZ^n$ with $k_1\ge k_2\ge \cdots\ge k_n$.
	For any $N \in \NN$, let 
	$\Gamma(N)\subset\Sp_{2n}(\ZZ)$ be the principal congruence subgroup of level $N$,
	and let $\chi=(\chi_1,\ldots, \chi_n) : \prod_{j=1}^n (\ZZ/N\ZZ)^\times \rightarrow \CC^\times$ be a character.
	We write $M_{\rho}(\Gamma(N), \chi)$ for the space of vector-valued Siegel modular forms
	of degree $n$, type $\rho$, level $N$ and character $\chi$ (see \S \ref{Hecke operators on Siegel modular forms}).
	We may assume $k_n\ge 0$ by K\"ocher's principle.

	Let ${\bf T}_{n,\ZZ}^{\rho, N, \chi}$ be the $\ZZ$-algebra of 
	$\End(M_{\rho}(\Gamma(N), \chi))$
	generated by
	$$
	p^{\delta(k_n<n)\frac{(n-k_n)(n-k_n+1)}{2}}T(p), \qquad (\text{$p$ : prime}, \ p\notdivide N)
	$$
	and
	$$p^{\delta(k_n\le n)n(n-k_n+1)}T_{j,n-j}(p^2), \qquad (\text{$p$ : prime}, \ p\notdivide N, \ 1\le j \le n).
	$$
	Here
	$T(p)$ and $T_{j,n-j}(p^2)$ are the Hecke operators defined as in \S \ref{Hecke operators on Siegel modular forms}, and
	$\delta({\rm P})$ for a condition $\rm P$ is the generalized Kronecker delta
	so that $\delta({\rm P})=1$ if $\rm P$ is true and $0$ otherwise.
	The two factors $p^{\delta(\cdots)(\cdots)}$ are equal to $1$ when $k_n\ge n+1$.
By \cite[Theorem 1.1 and Corollary]{Evdokimov} and \cite[Theorem 3.3.23 (1)]{AndrianovBook}, ${\bf T}_{n,\ZZ}^{\rho, N, \chi}\otimes_{\ZZ} \CC$ coincides with the $\CC$-algebra consisting of the usual Hecke operators on $\End(M_\rho(\Gamma(N),\chi))$.

	\begin{thm}\label{integrality for SMF}
		All eigenvalues of any Hecke operator $T \in {\bf T}_{n, \ZZ}^{\rho, N, \chi}$ on $M_{\rho}(\Gamma(N), \chi)$
		are algebraic integers.
	\end{thm}
	Note that Theorem \ref{integrality for SMF} is still a new result even in the case of $N=1$,
	and that such integrality holds for $M_{\rho}(\Gamma_0(N))$,
	where $\Gamma_0(N)$ is the congruence subgroup of $\Sp_{2n}(\ZZ)$ defined as \cite[(1.3.28)]{AndrianovBook}. Indeed, by $M_\rho(\Gamma_0(N)) \subset M_\rho(\Gamma(N))$, the integrality of Hecke eigenvalues for $M_\rho(\Gamma_0(N))$ follows immediately from 
		Theorem \ref{integrality for SMF}.
	
	Integrality of Hecke eigenvalues for Siegel modular forms has been studied in some cases.
	Integrality was given for the case of scalar-valued Siegel modular forms of level $1$ by Kurokawa \cite{Kurokawa} for degree $n=2$, and by Mizumoto \cite{Mizumoto}, \cite{Mizumotocorrection} for general degree $n$ (see also Katsurada \cite{Katsurada}).
	The $\Sym^2\otimes\det^k$-valued case of level 1 and degree $2$ were studied by Satoh \cite{Satoh}.
	The algebraicity that Hecke eigenvalues are contained in $\overline{\QQ}$
	was given for $\Sym^l\otimes\det^k$-valued Siegel modular forms of degree $n$ and level 1 by Takei \cite{Takei}.
	Contrary to the previous works, Theorem \ref{integrality for SMF} is given for general vector-valued Siegel modular forms of general level,
	and it includes the case of singular modular forms, i.e., Siegel modular forms of small weights (cf.\ \cite[Kapitel IV, Definition 5.2 and Satz 5.3]{Freitag}, \cite[Satz A 4.1]{Freitag} and \cite[\S2.3.6]{AndrianovBook}).
	
	We also remark that, as for automorphic representations of $\Sp_{2n}$, the integrality was proved by Shin and Templier \cite[Proposition 4.1]{ShinTemplier} by using Galois representations under the hypotheses on Arthur's endoscopic classification (see Remark \ref{remark on Shin Templier}).
	
	We expect the same machinery for the existence of an integral structure for the case of automorphic forms on the adelization of a certain algebraic group of PEL-type.
	Indeed, we give the existence of integral structures for automorphic forms in such a case (see Appendix \ref{appendix}).
	Our proof for integral structures relies basically on the construction of degenerating abelian varieties by Mumford \cite{Mumford72} and an algebraic interpretation of $q$-expansions by Lan \cite{Lan12} which is a direct generalization of Faltings-Chai \cite{FaltingsChai} (see also the thesis of Taylor \cite{Taylor}).

	\subsection{Applications to the growth of the fields of rationality and the non-vanishing of central $L$-values}
	One application of the integrality of Hecke eigenvalues is Tsuzuki and the second author's work \cite{SugiyamaTsuzuki} on explicit relative trace formulas for the maximal split torus of $\GL_2$.
	They gave two
	equidistribution results of Hecke eigenvalues of Hilbert cusp forms of a fixed parallel weight with non-vanishing central $L$-values and non-vanishing central derivatives.
	They applied the two equidistributions to estimating the growth of the degrees of Hecke fields of Hilbert modular forms.

For explaining our result, we recall the notion of fields of rationality (or Hecke fields) of automorphic representations (\cite[\S3.1]{ClozelMotif}). Let $\pi=\pi_\infty \otimes \pi_\fin$ be an irreducible cuspidal automorphic representation of $\GL_n(\AA_F)$ with $n\ge 2$ for a number field $F$. The field of rationality $\QQ(\pi)$ of $\pi$ is defined to be the subfield of $\CC$ consisting of all complex numbers fixed by 
		the group $\{\sigma \in \Aut(\CC) \mid \pi_\fin^\sigma \cong \pi_\fin\}$.
			
		If $n=2$ and $\pi$ corresponds to a primitive Hilbert cusp form $f$
		of weight $\bfk=(k_v)_{v|\infty}$ and level $\gn$, $\QQ(\pi)$ is also called the Hecke field of $\pi$. We put $\lambda_{\gp}(\pi)\in \CC$ as a suitably normalized Hecke eigenvalue at any non-zero prime ideal $\gp$ not dividing $\gn$. If all $k_v$'s are even, as in \S\ref{Hilbert modular case}, we have the expression
		$$\QQ(\pi)=\QQ(\{\nr_{F/\QQ}(\gp)^{1/2}\lambda_\gp(\pi) \mid \gp\notdivide \gn\})=\QQ(f),$$
		where $\nr_{F/\QQ}$ is the absolute norm map for $F/\QQ$ and $\QQ(f)$ is the Hecke field of $f$.
	 For the detail of fields of rationality of Hilbert modular forms, see \cite[Theorem 3.10 and \S4.5.3]{Raghuram-Tanabe}.
	  
	As an application of Theorem \ref{thm0.1}, we can eliminate the parallel weight condition from the results \cite[Theorems 1.3 and 1.4]{SugiyamaTsuzuki} on the growth of fields of rationality
	(see  Corollaries \ref{nonzero L and growth of Hecke} and \ref{nonzero deriv of L and growth of Hecke}).
	Corollary \ref{nonzero L and growth of Hecke} is roughly stated as follows.
	
	\begin{cor}
		[Precisely, see Corollary \ref{nonzero L and growth of Hecke}]
		\label{main nonzero L and growth of Hecke}
		Let $F$ be a totally real number field of finite degree.
		Let $l=(l_v)_{v | \infty}$ be a family of positive even integers such that $\min_{v|\infty}l_v \ge 6$, which is not necessarily parallel.
		Then, there exists a family $\{\pi_k\}_{k\in \NN}$ of irreducible cuspidal automorphic representations of $\PGL_2(\AA_F)$
		of weight $l$
		such that
		\begin{itemize}
			\item  the conductor $\gf_{\pi_k}$ of $\pi_k$ satisfies $\lim_{k\rightarrow \infty}\nr_{F/\QQ}(\gf_{\pi_k})=\infty$,
			
			\item 
			we have the estimate
			$$[\Q(\pi_k):\Q]
			\gg \sqrt{\log\log \nr_{F/\QQ}(\gf_{\pi_k})},\qquad k \in \NN,$$
			
			\item the completed standard $L$-function $L(s, \pi_k)$ attached to $\pi_k $ satisfies $L(1/2,\pi_k)\not=0$.
			
		\end{itemize}
	\end{cor}

	By virtue of automorphic inductions (\cite{ArthurClozel}) of Hilbert cusp forms of {\it non-parallel} weight,
	Corollary \ref{main nonzero L and growth of Hecke}
	is applied to the growth of the fields of rationality of automorphic representations of $\GL_n(\AA_\QQ)$.
	Let $d$ be a prime number.
	We state the growth of the fields of rationality of
	cuspidal automorphic representations
	of $\GL_{2d}(\AA_\QQ)$ and the dimensions of cuspidal cohomologies for ${\rm GL}_{2d}$
	with non-vanishing central $L$-values as follows.
	
	\begin{cor}[Precisely, see Corollary \ref{lower bound for Hecke field}]
		\label{cor Hecke fields GL2d}
		Let $d$ be a prime number.
		Then there exists a family $\{\Pi_k\}_{k \in \NN}$ of irreducible cohomological cuspidal automorphic representations of $\GL_{2d}(\AA_\QQ)$ such that
		\begin{itemize}
			
			\item Let $\gf_{\Pi_k} \in \NN$ be the conductor of $\Pi_k$, where we identify the conductor of $\Pi_k$ as an ideal with its positive generator. Then $\{\gf_{\Pi_k}\}_{k \in \NN}$ satisfies $\lim_{k\rightarrow \infty}\gf_{\Pi_k}=\infty$,
			
			\item we have the estimate
			$$[\QQ(\Pi_k):\QQ] \gg \sqrt{\log\log {\gf_{\Pi_k}}}, \qquad k \in \NN,$$
			
			\item the completed standard $L$-function $L(s,\Pi_k)$ attached to $\Pi_k$ satisfies $L(1/2, \Pi_k)\neq 0$.
		\end{itemize}
		
Furthermore there exist	$D \in \NN$ and a local system $\tilde V_\lambda$ on
the locally symmetric space	$$S_{\bfK_1(D{\gf_{\Pi_k}}\ZZ)}^{\GL_{2d}}:=\GL_{2d}(\QQ)\bsl \GL_{2d}(\AA_{\QQ})/\RR_{>0}\,\SO(2d)\,\bfK_1(D{\gf_{\Pi_k}}\ZZ)$$
	such that, for any $q \in \ZZ$ with $d^2\le q \le d^2+d-1$, we have
		the following estimate of the dimension of the cuspidal cohomology{\rm:}
		$$\dim H_{\rm cusp}^{q}(S_{\bfK_1(D\gf_{\Pi_k}\ZZ)}^{\GL_{2d}}, \tilde V_\lambda) \ge
		C \sqrt{\log\log \gf_{\Pi_k}},$$
		where the constant $C>0$ is independent of the conductors $\{\gf_{\Pi_k}\}_{k\in \NN}$.
		Here $\bfK_1(N)$ for $N \in \NN$ is the congruence subgroup of $\GL_2(\widehat\ZZ)$ with level $N$ give in \S4.
	\end{cor}

	We cannot produce such a family $\{\Pi_k\}_{k\in\NN}$ if we use the parallel weight version \cite[Theorem 3]{SugiyamaTsuzuki}
	of Corollary \ref{main nonzero L and growth of Hecke}.
	Note that
	Clozel \cite{Clozel87} gave the non-vanishing of cuspidal cohomology for $\GL_{2n}$, but not the non-vanishing of central $L$-values, using automorphic inductions from $\GL_1$.
	Also remark that Corollary \ref{cor Hecke fields GL2d} is not covered by Shin, Templier \cite[Proposition 6.10]{ShinTemplier} since $\GL_{2d}(\RR)$ does not have discrete series.
	Besides, they did not treat the non-vanishing of central $L$-values. 
	For upper bounds of dimensions of cuspidal cohomologies
	toward the Sarnak-Xue conjecture, refer to \cite{CalegariEmerton}
	and \cite{MarshallShin}.

	By the base change lifting from $\GL_2(\AA_F)$ to $\GL_2(\AA_E)$ where $E/F$ is a ramified quadratic extension of number fields,
	Corollary \ref{main nonzero L and growth of Hecke} leads us to the following result on the growth of the fields of rationality of automorphic representatons of
	$\GL_2(\AA_E)$.
	\begin{cor}\label{growth for BC}
		Let $F$ be a totally real number field of finite degree and $E$ a ramified quadratic extension.
		Let $\{\mathfrak{n}_k\}_{k\in \NN}$ be a family of non-zero ideals of $\Ocal_F$ such that $\lim_{k\rightarrow \infty}\nr_{F/\QQ}(\gn_k)=\infty$.
		Then, there exists a family $\{\Pi_k'\}_{k\in \NN}$ of irreducible
		cohomological cuspidal automorphic representations
		of $\PGL_2(\AA_E)$ such that $\nr_{E/\QQ}(\gf_{\Pi_k'}) = \nr_{F/\QQ}(\gn_k)^2$,
		$$[\QQ(\Pi_k'):\QQ]\gg \sqrt{\log\log \nr_{E/\QQ}(\gf_{\Pi_k'})}, \quad k \in \NN$$
		and		$$L(1/2,\Pi_k')\neq 0, \quad k \in \NN,$$
		where the implied constant is independent of $\{\gn_k\}_{k\in \NN}$.

Let $r_2$ be the number of complex places of $E$. Put $b_2^E=[E:\QQ]-r_2$ and $\tilde t_2^E=2[E:\QQ]-1$, which are the same as in \cite[Proposition 2.15]{Raghuram}.
		Then there exists a local system $\tilde V_{\lambda'}$ on
		$$S^{\Res_{E/\QQ} \GL_{2, E}}_{\bfK_1(\gf_{\Pi_k'})}:=\GL_{2}(E)\bsl \GL_{2}(\AA_{E})/ K_{E_\infty}^{+}\,\bfK_1(\gf_{\Pi_k'})$$
such that, for any $q \in \ZZ$ with $b_2^{E}\le q \le \tilde t_2^{E}$
		we have
		$$\dim H_{\rm cusp}^{q}(S^{\Res_{E/\QQ}\GL_{2,E}}_{\bfK_1(\gf_{\Pi_k'})}, \tilde V_{\lambda'}) \ge C \sqrt{\log\log \nr_{E/\QQ}(\gf_{\Pi_k'})},$$
		where the constant $C>0$ is independent of $\{\gn_k\}_{k\in \NN}$.
		Here $K_{E_\infty}^+$ is the product of the standard maximal compact subgroup of $\SL_2(E\otimes_\QQ \RR)$ and the connected component of the center of $\GL_2(E\otimes_\QQ \RR)$, and $\bfK_1(\ga)$ for a non-zero ideal $\ga$ of $\Ocal_E$ is the congruence subgroup of $\GL_2(\widehat\Ocal_E)$ with level $\ga$ give in \S4.
	\end{cor}
	When $E$ has complex places, 
	Shin and Templier's result \cite[Proposition 6.10]{ShinTemplier} does not cover the estimate
	$[\QQ(\Pi_k'):\QQ]\gg \sqrt{\log\log \nr_{E/\QQ}(\gf_{\Pi_k'})}$
	since $\GL_2(\CC)$ does not have discrete series.
	Besides, they did not treat the non-vanishing of central $L$-values.

	\subsection{Growth of Hecke fields of Siegel modular forms of degree 2}
	
	We give an application of the integrality for Siegel modular forms of degree $2$
	to the growth of their Hecke fields. Such an application was
	first stated by Kim, Wakatsuki and Yamauchi \cite{KimWY} and its continuation \cite{KimWY2}.
	In the proofs of \cite[Corollaries 1.4 and 1.5]{KimWY},
	the integrality of Hecke eigenvalues for Siegel modular forms of degree $2$
	should be used in an essential way.
	Nevertheless, the deduction of the integrality seems to be not appropriately addressed in their paper.
	In the proof in \cite[\S7.3]{KimWY}, they referred to Taylor's thesis \cite[Lemma 2.1]{Taylor}.
	However, \cite[Lemma 2.1]{Taylor} is a result on the integrality of Fourier coefficients but not
	of Hecke eigenvalues.
	Furthermore, since \cite[Theorems 1.1, 1.2 and 1.3]{KimWY2}\footnote{\cite[Theorem 1.3]{KimWY2} is stated by using $\inf\{\cdots\}$.
		This should be corrected as $\sup\{\cdots\}$.} rely on arguments in \cite[\S7.3]{KimWY} and \cite[\S6.3]{ShinTemplier},
	the	integrality of Hecke eigenvalues is needed to complete the proofs of \cite[Theorems 1.1, 1.2 and 1.3]{KimWY2}.
	Our Theorem \ref{integrality for SMF} is the required integrality of Hecke eigenvalues.

	Let $\rho$ be an irreducible algebraic representation
	of $\GL_2(\CC)$ with highest weight $(k_1, k_2)$.
	For $N\in \NN$, let $M_\rho(\Gamma(N))_\chi=\oplus_{\chi_1}M_\rho(\Gamma(N),(\chi_1,\chi))$ be the space of Siegel modular forms of type $\rho$ and level $N$ with central character $\chi: (\ZZ/N\ZZ)^\times \rightarrow \CC^\times$.
	The subspace of cusp forms in $M_\rho(\Gamma(N))_\chi$ is denoted by
	$S_\rho(\Gamma(N))_\chi$.
	Take a basis ${\rm HE}_\rho(N,\chi)$ of $S_\rho(\Gamma(N))_\chi$
	consisting of cuspidal Hecke eigenforms for all Hecke operators at all primes $p\notdivide N$
	so that each $f \in {\rm HE}_\rho(N,\chi)$ generates an irreducible cuspidal automorphic representation of $\GSp_{4}(\AA_\QQ)$.
	 Let ${\rm HE}_\rho(N,\chi)^{\rm tm}$ be  the set of all $f \in {\rm HE}_\rho(N,\chi)$ such that the corresponding automorphic representations are tempered at all primes not dividing $N$.
	
For any $f \in {\rm HE}_\rho(N,\chi)$, let $\QQ(f)$ be the Hecke field of $f$ by removing data at ramified places. Namely, $\QQ(f)$ is the subfield of $\CC$ generated by all Hecke eigenvalues of $f$ for $T(p)$, $T_{1, 1}(p^2)$ and $T_{2, 0}(p^2)$ at all primes $p\notdivide N$. The Hecke field $\QQ(f)$ is a finite extension of $\QQ$ (see Corollary \ref{Hecke field is finite}).

By using Theorem \ref{integrality for SMF} on the integrality of Hecke eigenvalues for $S_\rho(\Gamma(N))_\chi$,
	we complete the proof of \cite[Corollaries 1.4 and 1.5]{KimWY}
	and obtain the following estimate on the degrees of Hecke fields.
	
	\begin{cor}\label{Hecke fields for degree 2}
		Fix $\rho$ and suppose $k_1\ge k_2\ge 3$.
		For a fixed prime number $p$,
		we take a family $\{(N_j, \chi_j)\}_{j\in \NN}$
		of $N_j \in \NN$ and Dirichlet characters $\chi_j$ modulo $N_j$
		satisfying $\gcd(N_j, 11!p)=1$ for all $j\in \NN$ and $\displaystyle \lim_{j\rightarrow \infty}N_j=\infty$.
		Then we have
		$$\max_{\substack{f \in {\rm HE}_{\rho}(N_j,\chi_j)^{\rm tm}}}[\QQ(f):\QQ]\gg_{\rho, p} \sqrt{\log \log N_j}, \quad j \in \NN.$$
		This result is still true when $f$ is restricted to be non-endoscopic.
	\end{cor}

	\subsection{Growth of Hecke fields of Siegel modular forms of general degree}
	\label{Growth of Hecke fields of Siegel modular forms of general degree}
	
	We give an application of the integrality for Siegel modular forms of degree $n \ge 2$ combining with the equidistribution 
	result \cite[Theorem 1.2]{KimWY3}.
	For an irreducible algebraic representation $\rho$ of $\GL_n(\CC)$ with highest weight $(k_1,\ldots,k_n)$
	and for $N\in \NN$,
	let $S_\rho(\Gamma(N))_{\bf1}$ be the space consisting of all Siegel cusp forms
	in $\oplus_{\chi_1,\ldots,\chi_{n-1}}M_\rho(\Gamma(N), (\chi_1,\ldots,\chi_{n-1},\bf1))$, where $\chi_1,\ldots,\chi_{n-1}$
	run over the set of Dirichlet characters modulo $N$ and $\bf1$ denotes the principal character modulo $N$.
	
	Let ${\rm HE}_\rho(N)$ be a basis of $S_\rho(\Gamma(N))_{\bf1}$ consisting of Hecke eigenforms for 
	Hecke operators at all primes $p\notdivide N$
	such that each $f\in {\rm HE}_\rho(N)$ generates an irreducible cuspidal automorphic representation of $\Sp_{2n}(\AA_\QQ)$ (see \cite[\S2]{KimWY3}).
	For $f \in {\rm HE}_\rho(N)$, let $\QQ(f)$ be the Hecke field of $f$ defined
	as the subfield of $\CC$ generated by the Hecke eigenvalues of $f$ for $T(p)$ and $T_{j,n-j}(p^2)$ ($1\le j \le n$) at all primes $p\notdivide N$.
	The Hecke field $\QQ(f)$ is a finite extension of $\QQ$ (see Corollary \ref{Hecke field is finite}).

	By using Theorem \ref{integrality for SMF} on the integrality of Hecke eigenvalues for $S_\rho(\Gamma(N))_{\bf1}$,
	we give the following.
	\begin{cor}\label{Hecke fields for degree n}
		Fix $\rho$ such that $k_1\ge k_2 \ge \cdots \ge k_n > n+1$.
		For a fixed prime number $p$,
		let $\{N_j\}_{j \in \NN}$ be a family of positive integers such that $p\notdivide N_j$ for all $j \in \NN$ and $\displaystyle \lim_{j \rightarrow \infty}N_j=\infty$.	
		Then we have
		\[\max_{ f \in {\rm HE}_{\rho}(N_j)}[\QQ(f):\QQ] \gg_{n, \rho, p} \sqrt{\log\log N_j}, \qquad j \in \NN.\]
	\end{cor}
	
	Shin and Templier \cite[Proposition 6.10]{ShinTemplier} gave an estimate similar to the result above when $(k_1, \ldots, k_n)$ is regular,	under the hypotheses on Arthur's endoscopic classification (see Remark \ref{remark on Shin Templier}).
	They considered a family of irreducible cuspidal automorphic representations
	of $\Sp_{2n}(\A_\Q)$ whose archimedean component is in the fixed $L$-packet containing the holomorphic discrete series representation
	of lowest weight $(k_1, \ldots, k_n)$.
	Thus the advantages of our result are: (i) we give results without endoscopic classifications, and
	(ii) we restrict the archimedean 
	components of automorphic representations to a fixed holomorphic discrete series representation.

	\subsection{Organization of this paper}
	In \S \ref{abst},
	we introduce an abstract machinery for proving the integrality of eigenvalues of Hecke operators,
	and apply this to the Hilbert modular and the Siegel modular cases (Theorems \ref{thm0.1} and \ref{integrality for SMF}).
	In \S \ref{Application of integrality to the growth of Hecke fields}, the proof of Corollaries \ref{nonzero L and growth of Hecke} and \ref{nonzero deriv of L and growth of Hecke} (and hence Corollary \ref{main nonzero L and growth of Hecke}) is explained, where
	we use the integrality of Hecke eigenvalues of Hilbert cusp forms of non-parallel weight.
	The proof of Corollaries \ref{Hecke fields for degree 2} and \ref{Hecke fields for degree n} for Siegel modular forms is also explained in \S \ref{Proof of Corollaries and (Siegel modular case)}.
	In \S \ref{Cuspidal cohomology}, we prove Corollary \ref{lower bound for Hecke field} (and hence Corollary \ref{cor Hecke fields GL2d})
	and Corollary \ref{growth for BC}.
In Appendix \ref{appendix}, we prove Theorem \ref{integral model}, 
which states the existence of integral structures of the spaces of automorphic forms of PEL-type under
Assumption \ref{totdeg}, with the aid of Lan's comparison 
of algebraic and analytic $q$-expansions \cite{Lan12}.

	\subsection{Notation}\label{Notation}
	Let $\NN$ be the set of positive integers and $\overline{\QQ}$ the algebraic closure of $\QQ$ realized as a subfield of $\CC$. The ring of algebraic integers in $\overline{\QQ}$ is denoted by $\overline{\ZZ}$.
	
	For a condition $\rm P$, let $\delta({\rm P})$ denote the generalized Kronecker delta
	so that $\delta({\rm P})=1$ if $\rm P$ is true and
	$\delta({\rm P})=0$ otherwise.
	
	For a unital ring $R$, let ${\rm M}_n(R)$ denote the matrix algebra of $n\times n$
	matrices whose components are contained in $R$.
	We write $1_n$ and $0_n$ for the unit matrix and the zero matrix in ${\rm M}_{n}$.
	The
	general symplectic group $\GSp_{2n}(R)$ is defined as
	$$\GSp_{2n}(R)=\{ g \in \GL_{2n}(R) \mid {}^{t}g\left[\begin{smallmatrix}
		0_n & 1_n \\ -1_n & 0_n
	\end{smallmatrix}\right] g = \nu(g)\left[\begin{smallmatrix}
		0_n & 1_n \\ -1_n & 0_n
	\end{smallmatrix}\right] \},$$
	where $\nu: \GSp_{2n}(R)\rightarrow \GG_m(R)=R^\times$ is the similitude character.
	Then, $\Sp_{2n}(R) =\Ker\, \nu$ is the symplectic group of rank $n$.

	The symbol $\ll$ denotes Vinogradov's notation. If we clarify the dependence of the implied constant on parameters $a,b,c,\ldots$, we attach such parameters with the symbol $\ll$ such as $\ll_{a,b,c,\ldots}$.

	
	\section{Integrality of eigenvalues}\label{abst}
	
	In this section, we give a proof of the integrality of Hecke eigenvalues for Hilbert and Siegel modular forms.
	
	\subsection{Abstract setting}
	In this subsection, we introduce an abstract machinery
	for integrality.
	Let $R$ be a Noetherian integral domain with the fractional field $K$ and
	$L$ an algebraically closed extension field of $K$.
	Let $M$ be a finite dimensional $L$-vector space
	and let $T$ be an $L$-linear endomorphism of $M$.
	Further, we assume that there exists the following
	$L$-linear injective homomorphism
	\begin{eqnarray*}
		e:M\hookrightarrow L^{\mathbb N}
	\end{eqnarray*}
	from $M$ to the uncountably infinite dimensional $L$-vector space $L^\NN$.
	The $L$-linear homomorphism $e$ will be a Fourier expansion map later.
	We define an $R$-submodule $M_0$ of $M$ by
	\begin{eqnarray*}
		M_0:=e^{-1}(R^{\mathbb N}).
	\end{eqnarray*}
	We assume the following two conditions:
	\begin{itemize}
		\item[(i)] The $R$-module $M_0$ contains an $L$-basis of $M$.
		\item[(ii)] The operator $T$ preserves $M_0$.		
	\end{itemize}
	
	\begin{thm}\label{machine of int}
		In the setting as above, each eigenvalue $\lambda \in L$
		of $T$ is integral over $R$.
	\end{thm}
	In the rest of this subsection, we give a proof of this theorem.
	For each positive integer $N$,
	we regard $L^N$ as a quotient of $L^{\mathbb N}$
	by the natural projection
	\begin{eqnarray*}
		\mathrm{pr}_N: L^{\mathbb N}\twoheadrightarrow L^N;(x_1,\dots, x_N,x_{N+1},\dots)\mapsto (x_1,\dots, x_N).
	\end{eqnarray*}
	\begin{lem}	\label{lem1}
		For a sufficiently large positive integer $N$,
		the composite
		\[
		\mathrm{pr}_N\circ e:M\rightarrow L^N
		\]
		is injective.
	\end{lem}
	\begin{proof}Let $M_N$ be the kernel of $\mathrm{pr}_N\circ e$.
		Then, by definition, we have the descending filtration $M\supset M_1\supset\cdots$
		of $L$-vector spaces.
		Since $M$ is finite dimensional, this sequence stabilizes for sufficiently large integer $N'$.
		If $M_{N'}\neq 0$, then $e$ is not injective. Hence, we have $M_{N'}=0$.
	\end{proof}
	
	\begin{lem}\label{lem2}The $R$-module $M_0$ is finitely generated over $R$. Moreover, the natural homomorphism
		\[
		M_0\otimes_RL\to M
		\]
		is an isomorphism.
	\end{lem}
	\begin{proof}
		By the lemma above, we have the following commutative diagram:
		\[
		\xymatrix{
			M_0\ar@{^{(}->}[d]\ar@{^{(}->}[rrr]^{\mathrm{pr}_N\circ e}& & &R^N\ar@{^{(}->}[d]\\
		M\ar@{^{(}->}[rrr]^{\mathrm{pr}_N\circ e}& & &L^N.
		}
		\]
		Then, $M_0$ is finitely generated over $R$ as $R$ is a Noetherian ring.
		
		We show the second assertion.
		Since $L$ is flat over $R$, the homomorphism
		\[
		M_0\otimes_RL\to R^N\otimes_RL\xrightarrow{\sim}L^N
		\]
		is injective. This homomorphism is factorized as
		\[
		M_0\otimes_RL\to M\xrightarrow{\mathrm{pr}_N\circ e}L^N,
		\] 
		the first homomorphism is injective. This is also surjective by the assumption (i).
		We are done.
	\end{proof}

		\begin{lem}\label{eigen is in R}
			Let $R$ be an integral domain with the fractional field $K$.
			Let $V$ be a finite dimensional $K$-vector space
			and $V_0$ a finitely generated $R$-submodule of $V$ which contains a $K$-basis of $V$.
			Let $T$ be a $K$-linear endomorphism of $V$.
			We regard $T$ as the natural extension of $T$ to $V\otimes_K\overline{K}$,
				where $\overline K$ is an algebraic closure of $K$.

			If $T(V_0)\subset V_0$, then any eigenvalue $\l \in \overline{K}$ of $T$ is integral over $R$.
			\label{lem1.5}
		\end{lem}
	\begin{proof}
		Let $m_1,\dots,m_n\in V_0$ be generators of the $R$-module $V_0$.
		By the assumption that $T(V_0)\subset V_0$, there exists $(a_{ij})_{1\le i,j\le n}\in {\rm M}_n(R)$ such that
		\[
		T(m_i)=\sum_{j=1}^n a_{ij}m_j
		\]
		for all $i \in \{1,\ldots, n\}$. Let us define a monic polynomial $F(X)\in R[X]$ by
		\[
		F(X)=\det\left(X1_n-\begin{pmatrix}
			a_{ij}
		\end{pmatrix}_{1\leq i,j\leq n}\right).
		\]
		Then, by \cite[Proposition 2.4]{Atiyah}, we have $F(T)=0$ in the ring $\End_R(V_0)$.
			Since $V_0$ contains a $K$-basis of $V$, the identity $F(T)=0$ also holds in $\End_K(V)$
		and hence in $\End_{\overline K}(V\otimes_K\overline K)$.
		
		Let $v \in V\otimes_K \overline{K}$ be an eigenvector of $T$ with the eigenvalue $\lambda \in \overline{K}$. Then we have
		\[
		F(\lambda)v=F(T)v=0
		\]
		in $V\otimes_K\overline K$. Since $v$ is non-zero, we obtain $F(\lambda)=0$.
		This completes the proof.
		\end{proof}
	\begin{proof}[Proof of Theorem \ref{machine of int}]
		By the assumption (ii), $T$ induces a $K$-linear map $M_0 \otimes_{R} K\rightarrow M_0 \otimes_{R} K$.
		According to Lemma \ref{lem2}, we have $M_0\otimes_RK\otimes_KL\xrightarrow{\sim}M$.
		Therefore, any eigenvalue of $T$ is contained in $\overline K$.
		Moreover, $M_0$ is finitely generated over $R$ by Lemma \ref{lem2}.
		Therefore by applying Lemma \ref{eigen is in R} as $V=M_0\otimes_{R} K$ and $V_0=M_0$,
		we have the conclusion.
	\end{proof}
	\subsection{Proof of Theorem \ref{thm0.1}}\label{proofofthm1}
	We prove the integrality of Hecke eigenvalues for Hilbert modular forms, which was proved by Shimura for the parallel weight modular forms (cf.\ \cite[Proposition 2.2]{Shimura}).
	
	Let $F$ be a totally real number field of finite degree,
	$\Ocal_F$ the integer ring of $F$,
	$\AA_F$ the adele ring of $F$, and $\AA_{F,\fin}$ the ring of finite adeles of $F$, respectively.
	Let $\Dcal_{F/\QQ}$ be the global different of $F/\QQ$ and $D_F$ the absolute value of the discriminant of $F/\QQ$.
	For a place $v$ of $F$, we write $v | \infty$ if $v$ is archimedean.
	For a finite place $v$, $\Ocal_{F,v}$ denotes the integer ring of the completion $F_v$ of $F$ at $v$.
	For a non-zero ideal $\ga$ of $\Ocal_F$, set $\ga_v=\ga\Ocal_{F,v}$. The absolute norm of $\ga$ is denoted by $\nr_{F/\QQ}(\ga)$.
	
	Let $\mathfrak n$ be a non-zero ideal of $\mathcal O_F$ and $\mathbf k=(k_v)_{v|\infty}$ a family of positive integers.
	For a Hilbert modular form $f\in M_{\bf k}(\gn, \chi)$,
	the adelic lift $\tilde f$ of $f$ defined as \cite[(2.7)]{Shimura}
	has a Fourier expansion. This is of the form
	\begin{align}\label{adelicFourierExpansion}
		\tilde f(\begin{smallmatrix}
		y & x \\ 0 & 1
	\end{smallmatrix})=& \sum_{0\ll \xi \in F}c(\xi y \Ocal_F, f)(\xi y_\infty)^{{\bf k}/2}\exp(-2\pi \,\tr(\xi y_\infty)) \psi_{F}(\xi x) \\
	&+\delta(\text{${\bf k}$ : parallel})\, c_0(y\Ocal_F, f)|y|_\AA^{k_0/2}\notag
	\end{align}
	for $x \in \AA_F$ and $y \in \AA_F^\times$ such that all components of $y_\infty$ are positive
	(cf.\ \cite[(2.18)]{Shimura})\footnote{The symbol $0\ll \xi \in F$ in the Fourier expansion means that $\xi \in F$ is totally positive. It is different from Vinogradov's notation.}.
	Here $\psi_{F}$ is the additive character of $F\bsl \AA_F$ given by $\psi_\QQ \circ \tr_{F/\QQ}$ with $\psi_\QQ=\otimes_v \psi_{\QQ, v}$ being the additive character of $\QQ\bsl \AA_\QQ$ determined by $\psi_{\QQ, \infty}(x)=e^{2\pi \sqrt{-1}x}$ for all $x \in \RR$.
	Furthermore,
	the symbol $\delta(\text{${\bf k}$ : parallel})$ equals $1$ or $0$ when $\bf k$ is parallel or non-parallel, respectively,
	$|\cdot |_\AA$ is the idele norm of $\AA_F^\times$, and we set $k_0:=\max_{v | \infty} k_v$.
	For a subring $R$ of $\mathbb C$, $M_{\bf k}(\gn, \chi;R)$
	denote the set of Hilbert modular forms of weight $\mathbf k$, level $\mathfrak n$ and nebentypus $\chi$
	whose Fourier coefficients $c(\xi y\mathcal O_F,f)\xi^{\bfk/2}$ are all in $R$.
	
	This $R$-module can be interpreted as Fourier expansions of Hilbert modular forms.
	Let $\GL_2(F)^{+}$ denote the group of all elements of $\GL_2(F)$ with totally positive determinants.
	Let $\Gamma_1(\mathfrak b,\mathfrak c)$ be the congruence subgroup of $\GL_2(F)^{+}$ defined by
	\[
	\Gamma_1(\mathfrak b,\mathfrak c)=\left\{\begin{pmatrix}
		a&b\\
		c&d
	\end{pmatrix}\in \GL_2(F) \ \bigg|\  a\in 1+\mathfrak c,\ b\in \mathfrak b^{-1},\ c\in \mathfrak b\mathfrak c,\ d\in \mathcal O_F,\ ad-bc\in \mathcal O_F^\times \right\}
	\]
	with a non-zero fractional ideal $\mathfrak b$ and a non-zero integral ideal $\mathfrak c$ in $F$ (\cite[(1.13)]{Shimura}).
Let $\mathfrak c_1,\dots,\mathfrak c_h$ be integral ideals of $F$ representing the narrow
	ideal class group of $F$.
	Then, by the definition of Hilbert modular forms by Shimura \cite{Shimura}, we have a decomposition
	\begin{equation}
		\label{incl1}
		\bigoplus_{\chi\colon (\mathcal O_F/\mathfrak n)^\times \to \mathbb C^\times}M_{\mathbf k}(\mathfrak n,\chi)=\prod_{i=1}^hM_{\mathbf k}(\Gamma_1(\mathfrak c_i\mathcal D_{F/\mathbb Q},\mathfrak n)).
	\end{equation}
	Here, $\mathfrak{H}$ is the Poincar\'e upper half-plane and
	$M_{\mathbf k}(\Gamma_1(\mathfrak c_i\mathcal D_{F/\mathbb Q},\mathfrak n))$ denotes the space of $\mathbb C$-valued holomorphic functions
	on $\mathfrak{H}^{[F:\QQ]}$
	satisfying the invariance
	\begin{equation}
		\label{modularity}
		f||_{\mathbf k}\gamma = f,\quad \forall \gamma\in \Gamma_1(\mathfrak c_i\mathcal D_{F/\mathbb Q},\mathfrak n)
		\end{equation}
	and the holomorphy at the cusps (see \cite[pp.638--639]{Shimura} for details).
	Recall that, for each function $f\colon \mathfrak H^{[F:\QQ]}\to\CC$ and each $\alpha=[\begin{smallmatrix}
			*&*\\
			c&d
	\end{smallmatrix}]\in (\GL_2(\RR)^{+})^{[F:\QQ]}$, the function $f||_\mathbf k\alpha\colon \mathfrak H^{[F:\QQ]}\to \CC$ is defined by the identity
\[
f||_{\mathbf k}\alpha(\tau)=\prod_{i=1}^{[F:\QQ]}\det(\alpha_i)^{-k_i/2}\prod_{i=1}^{[F:\QQ]}(c_i\tau_i+d_i)^{-k_i}f(\alpha\tau),
\]where the subscription $i$ implies the $i$th component. Therefore, each Hilbert modular form has the Fourier expansion by using the function
	$\exp(2\pi \sqrt{-1}\mathrm{Tr}_{F/\mathbb Q}(\tau))$, where $\mathrm{Tr}_{F/\mathbb Q}\colon \mathfrak{H}^{[F:\QQ]}\to \mathfrak H$ is the function defined by $\mathrm{Tr}_{F/\mathbb Q}((\tau_1,\dots,\tau_{[F:\QQ]}))=\sum_{j=1}^{[F:\QQ]}\tau_j$.
	Let us call this Fourier expansion {\it the Fourier expansion of $f$ in the classical sense}.
	According to \cite[(2.16) and (2.17)]{Shimura}, the Fourier expansion in the classical sense ``coincides'' with the Fourier expansion (\ref{adelicFourierExpansion}) above. Therefore, $M_{\bf k}(\gn, \chi;R)$ coincides with the subspace of $M_{\bf k}(\gn, \chi)$ consisting of Hilbert modular forms whose Fourier coefficients \emph{in the classical sense} are contained in $R$.
	Similarly, for any congruence subgroup $\Gamma$ of $\GL_2(F)^+$,
	we define $M_{\mathbf k}(\Gamma)$ to be the space of holomorphic function on $\mathfrak H^{[F:\QQ]}$ satisfying (\ref{modularity}) and define its subspace $M_{\mathbf k}(\Gamma;R)$  by using their Fourier expansions in the classical sense.
	
	We show the existence of an integral structure of the space of Hilbert modular forms as below.
	\begin{thm}\label{integral model HMF}
		There exists an algebraic number field $E$ of finite degree contained in $\mathbb C$ such that $\Ocal_E$ contains all values of $\chi$ and
		\[
		M_{\bf k}(\gn, \chi)=M_{\bf k}(\gn, \chi;\mathcal O_E)\otimes_{\mathcal O_E}\mathbb C.
		\]
	\end{thm}
	We note that Shimura already proved the existence of $E$ satisfying
	\[
	M_{\bf k}(\gn, \chi)=M_{\bf k}(\gn, \chi;E)\otimes_{E}\mathbb C
	\]
	in \cite[Proposition 1.3, (1.18) and (2.5$_{\rm c}$)]{Shimura}.
	Therefore to prove Theorem \ref{integral model HMF}, it suffices to show the boundedness
	of the denominators of Fourier coefficients.
	To prove the boundedness of denominators of Fourier coefficients of any elements of $M_{\bf k}(\gn, \chi; \overline{\QQ})$,
	it suffices to show that the boundedness for all elements of $\bigoplus_{\chi}M_{\bf k}(\gn, \chi; \overline{\QQ})$.
	
A key ingredient in the proof of Theorem \ref{integral model HMF} is a Katz modular form, which is a modular form on a PEL-type Shimura variety. We briefly recall its \emph{analytic} definition. For simplicity, we assume $F\neq \mathbb Q$. Let $\mathfrak c$ be a non-zero fractional ideal in $F$ as before.
A \emph{$\mathfrak c$-polarized lattice of $F\otimes_{\mathbb Q}\mathbb C$} is a pair $(\mathcal L,\langle\ ,\rangle)$, where
\begin{itemize}
\item $\mathcal L$ is an $\mathcal O_F$-submodule of $F\otimes_{\mathbb Q}\mathbb C$ such that $\mathcal L\otimes_{\mathbb Z}\mathbb R=F\otimes_{\mathbb Q}\mathbb C$, and
\item $\langle\ ,\ \rangle\colon \mathcal L\wedge_{\mathcal O_F}\mathcal L\xrightarrow{\sim} \mathcal D_{F/\mathbb Q}^{-1}\mathfrak c^{-1}$ is an isomorphism as $\mathcal O_F$-modules
\end{itemize}
(see \cite[(1.4.3), (1.4.4)]{Katz}).
A \emph{$\Gamma_{00}(\mathfrak n)$-level structure on a $\mathfrak c$-polarized lattice $(\mathcal L,\langle\ ,\ \rangle)$} is an injection $i\colon \mathfrak n^{-1}\mathcal D_{F/\mathbb Q}^{-1}/\mathcal D_{F/\mathbb Q}^{-1}\hookrightarrow \mathfrak n^{-1}\mathcal L/\mathcal L$ as $\mathcal O_F$-modules.
	We denote by $\mathrm{PL}(\mathfrak c;\mathfrak n)$ the set of triples $(\mathcal L,\langle\ ,\ \rangle,i)$ of $\mathfrak c$-polarized lattices of $F\otimes_{\mathbb Q}\mathbb C$ with $\Gamma_{00}(\mathfrak n)$-level structures.

Following Katz \cite[1.1]{Katz76}, we define the set $\mathbf{GL}^+_F$ by
\[
\mathbf{GL}^+_F=\{(\omega_1,\omega_2)\in ((F\otimes_{\mathbb Q}\mathbb{C})^{\times})^{2}\ |\ \omega_1\omega_2^{-1}\in \mathfrak H_F\}.
\]
Here we set
$$\mathfrak{H}_F=\{(z_{\iota})_{\iota} \in \CC^{\Hom(F, \RR)}\cong F\otimes_\QQ \CC\mid \Im(z_{\iota})>0 \ (\forall\iota\in\Hom(F, \RR))\}.$$
The natural action of $\mathrm{GL}_2(F\otimes_{\mathbb Q}\mathbb C)$ on $(F\otimes_{\mathbb Q}\mathbb C)^2$ from the right induces a natural \emph{left} action of $\mathrm{SL}_2(F\otimes_{\mathbb Q}\mathbb R)$ on $\mathbf {GL}^+_F$ defined by
\[
\SL_2(F\otimes_\QQ \RR)\times \mathbf{GL}_F^{+}\ni(\gamma, (\omega_1,\omega_2))\longmapsto (\omega_1,\omega_2)^{t\! }\gamma\in \mathbf{GL}_F^{+}.
\]
Note that our order is the reverse of that of Katz.

For each element $(\omega_1,\omega_2)$, we define a polarized $\mathfrak c$-lattice $(\mathcal L(\omega_1,\omega_2),\langle\ ,\ \rangle_{\omega_1,\omega_2})$ by
\[
\mathcal L(\omega_1,\omega_2):=2\pi\sqrt{-1}(\mathcal O_F\omega_1+\mathfrak c^{-1}\mathcal D_{F/\mathbb Q}^{-1}\omega_2)
\]
and by
\[
\langle 2\pi\sqrt{-1}\omega_1 ,2\pi\sqrt{-1}\omega_2\rangle_{\omega_1,\omega_2}=1.
\]
We choose an isomorphism
	\[
	\epsilon\colon \mathcal O_F/\mathfrak n\xrightarrow{\sim}\mathfrak c^{-1}\mathfrak n/\mathfrak n
	\]
	as $\mathcal O_F$-modules. Then we obtain a map $L_{\epsilon,\mathfrak c}\colon \mathbf{GL}^+_F\to \mathrm{PL}(\mathfrak c;\mathfrak n)$ defined by
\begin{equation}\label{naturalbij}
L_{\epsilon,\frak c}(\omega_1,\omega_2):=(\mathcal L(\omega_1,\omega_2),\ \langle\ ,\ \rangle_{\omega_1,\omega_2},i_{\omega_1,\omega_2}),
\end{equation}
where the $\Gamma_{00}(\mathfrak n)$-level structure $i_{\omega_1,\omega_2}$ is defined as the composite
\[
i_{\omega_1,\omega_2}\colon \mathfrak n^{-1}\mathcal D_{F/\mathbb Q}^{-1}/\mathcal D_{F/\mathbb Q}^{-1}\xrightarrow{\epsilon^{-1}}\mathfrak n^{-1}\mathcal D_{F/\mathbb Q}^{-1}\mathfrak c^{-1}/\mathcal D_{F/\mathbb Q}\mathfrak c^{-1}\xrightarrow{\times 2\pi\sqrt{-1}\omega_2} \mathfrak n^{-1}\mathcal L/\mathcal L
\]
of homomorphisms as $\mathcal O_F$-modules.
We set
\[
\Gamma_{00}(\mathfrak c\mathcal D_{F/\mathbb Q},\mathfrak n):=\Gamma_1(\mathfrak c\mathcal D_{F/\mathbb Q},\mathfrak n)\cap \SL_2(F),
\]
which coincides with the congruence subgroup $\Gamma_{00}(\mathfrak n;\mathfrak c,\mathcal O_F)$
defined in \cite[p.262]{Deligne-Ribet}.
\begin{lem}\label{naturalbijlem}
The map $
L_{\epsilon,\frak c}$ induces a bijection
\[
\overline L_{\epsilon,\frak c}\colon \Gamma_{00}(\mathfrak c\mathcal D_{F/\mathbb Q},\mathfrak n)\backslash\mathbf{GL}^+_F\xrightarrow{\sim} \mathrm{PL}(\mathfrak c;\mathfrak n)
\]
as sets.
\end{lem}
\begin{proof}
First, we show the surjectivity of the map $L_{\epsilon,\mathfrak c}$.
Let $(\mathcal L,\langle\ ,\ \rangle,i)$ be a $\mathfrak c$-polarized lattice with a $\Gamma_{00}(\mathfrak n)$-level structure.
Then the polarization $\langle\ ,\ \rangle$ induces a natural isomorphism
\[
\mathcal L\wedge_{\mathcal O_F}\mathfrak c\mathcal D_{F/\mathbb Q}\mathcal L\xrightarrow{\sim}\mathcal O_F
\]
as $\mathcal O_F$-modules. By abuse of notation, we denote this isomorphism by the same notation $\langle\ ,\ \rangle$. Pick $2\pi\sqrt{-1}u\in \mathcal L$ and $2\pi\sqrt{-1}v\in \mathfrak c\mathcal D_{F/\mathbb Q}\mathcal L$ such that
\[
\langle 2\pi\sqrt{-1}u,2\pi\sqrt{-1}v\rangle=1,\quad uv^{-1}\in \mathfrak H_F
\]
and the composite
\[
\mathfrak n^{-1}\mathcal D_{F/\QQ}^{-1}\mathfrak c^{-1}/\mathcal D_{F/\QQ}^{-1}\mathfrak c^{-1}\xrightarrow{\epsilon^{-1}}\mathfrak n^{-1}\mathcal D_{F/\QQ}^{-1}/\mathcal D_{F/\QQ}^{-1}\xrightarrow{i}\mathfrak n^{-1}\mathcal L/\mathcal L
\]
coincides with the map $a\mapsto 2\pi\sqrt{-1}a v$.
Then we have $\mathcal L=2\pi\sqrt{-1}(\mathcal O_Fu+\mathfrak c^{-1}\mathcal D^{-1}_{F/\mathbb Q}v)$.
Indeed, the inclusion $\mathcal L\supset2\pi\sqrt{-1}(\mathcal O_Fu+\mathfrak c^{-1}\mathcal D^{-1}_{F/\mathbb Q}v)$ is obvious. For each $l=2\pi\sqrt{-1}(au+bv)\in \mathcal L (a,b\in F)$, we have 
\[
a=\langle l,2\pi\sqrt{-1}v\rangle\in\mathcal O_F\quad b=-\langle l,2\pi\sqrt{-1}u\rangle\in  \mathfrak c^{-1}\mathcal D_{F/\mathbb Q}^{-1}
\]
because $2\pi\sqrt{-1}v\in \mathfrak c\mathcal D_{F/\QQ}\mathcal L$ and $2\pi\sqrt{-1}u\in \mathcal L$.
This implies the converse inclusion relation $\mathcal L\subset \mathcal O_Fu+\mathfrak c^{-1}\mathcal D_{F/\mathbb Q}^{-1}v$.
Therefore the map in the lemma is surjective.

Next we show the injectivity of the map $\overline L_{\epsilon,\frak c}$.
	Suppose $L_{\epsilon,\frak c}(\omega_1,\omega_2)=L_{\epsilon,\frak c}(\omega_1',\omega_2')$ for $(\omega_1,\omega_2),(\omega_1',\omega_2')\in \mathbf{GL}_F^+$.
	Since $\mathcal L(\omega_1,\omega_2)=\mathcal L(\omega_1',\omega_2')$,
	there exists an element $\gamma=[\begin{smallmatrix}
		\alpha&\beta\\
		\gamma&\delta
	\end{smallmatrix}]$ of $\GL_2(F)$ such that
	\[
	(\omega_1',\omega_2')=(\omega_1,\omega_2) ^{t} \!\gamma.
	\]
	Since the identities
	\[1=\langle 2\pi\omega_1',2\pi\omega_2'\rangle_{\omega_1',\omega_2'}=\langle 2\pi\omega_1',2\pi\omega_2'\rangle_{\omega_1,\omega_2}=\det(\gamma)
	\]
	hold, $\gamma$ is contained in $\SL_2(F)$. Moreover, by the identities
	\begin{equation*}
		\begin{split}
			\omega_1'&=\alpha \omega_1+\beta\omega_2\in \mathcal L(\omega_1,\omega_2),\\
			\omega_2'&=\gamma \omega_1+\delta\omega_2\in \mathfrak c\mathcal D_{F/\QQ}\mathcal L(\omega_1,\omega_2),
		\end{split}
	\end{equation*}
	we see \[
	\alpha,\delta\in \mathcal O,\quad \gamma\in \mathfrak c\mathcal D_{F/\QQ},\quad \beta\in \mathfrak c^{-1}\mathcal D_{F/\QQ}^{-1}.
	\]
	Finally, since $\gamma$ stabilizes $i_{\omega_1,\omega_2}$, we conclude that $\gamma$ is contained in $\Gamma_{00}(\mathfrak c\mathcal D_{F/\mathbb Q},\mathfrak n)$. This implies that $\overline L_{\epsilon,\frak c}$ is well-defined and injective.
\end{proof}
Since the action of the discrete group $\Gamma_{00}(\mathfrak c\mathcal D_{F/\mathbb Q},\mathfrak n)$ on $\mathbf{GL}_F^+$ is free and properly discontinuous, the quotient set
$\Gamma_{00}(\mathfrak c\mathcal D_{F/\mathbb Q},\mathfrak n)\backslash\mathbf{GL}^+_F$ has a natural structure of a complex manifold under which the map
$\Gamma_{00}(\mathfrak c\mathcal D_{F/\mathbb Q},\mathfrak n)\backslash\mathbf{GL}^+_F$ is a covering map between complex manifolds.
Then, by using a natural bijection in Lemma \ref{naturalbijlem}, we equip $\mathrm{PL}(\mathfrak c;\mathfrak n)$ with a structure as a complex manifold.

For each $(\mathcal L,\langle\ ,\ \rangle,i)\in \mathrm{PL}(\mathfrak c;\mathfrak n)$ and $\alpha \in F\otimes_{\mathbb Q}\CC$, put
		\[
		\alpha*(\mathcal L, \langle\ ,\ \rangle, i):=(\alpha^{-1}\mathcal L,\alpha\overline{\alpha} \langle\ ,\ \rangle,\alpha^{-1} i),
		\]
		where $\overline \alpha$ denotes the complex conjugate of $\alpha$. Then it is easily checked that the identity
		\begin{equation}\label{actionofGm}
		L_{\epsilon,\mathfrak c}(\alpha\omega_1,\alpha\omega_2)=\alpha^{-1}*L_{\epsilon,\mathfrak c}(\omega_1,\omega_2)
		\end{equation}
		holds.
Let $\chi_{\mathbf k}\colon  \mathrm{Res}_{F/\mathbb Q}(\mathbb G_{m,F})\times_{\Spec(\mathbb Q)}\Spec(\mathbb C)\cong \mathbb G_{m,\mathbb C}^d\to \mathbb G_{m,\mathbb C}$ be the algebraic character defined by
		\[
		\chi_{\mathbf k}((z_v)_{v|\infty})=\prod_{v|\infty}z_v^{k_v}.
		\]
An analytic \emph{$\mathfrak c$-Hilbert modular form of weight $\mathbf k$ and level $\Gamma_{00}(\mathfrak n)$ in the sense of Katz} (\cite[(1.6.1)]{Katz}) is a holomorphic function
\[
\mathcal F\colon \mathrm{PL}(\mathfrak c;\mathfrak n)\to \mathbb C
\]
satisfying
\[
\mathcal F(\alpha*(\mathcal L, \langle\ ,\ \rangle, i))=\chi_{\mathbf k}(\alpha)\mathcal F((\mathcal L,\langle\ ,\ \rangle,i))
\]
for all $(\mathcal L,\langle\ ,\ \rangle,i)\in \mathrm{PL}(\mathfrak c;\mathfrak n)$ and all $\alpha\in (F\otimes_{\mathbb Q}\mathbb C)^\times$.
We denote by $M(\mathfrak c,\Gamma_{00}(\mathfrak n),\chi_{\mathbf k};\mathbb C)^{\mathrm{an}}$ the space of analytic $\mathfrak c$-Hilbert modular forms of weight $\mathbf k$ and level $\Gamma_{00}(\mathfrak n)$ in the sense of Katz.

Next, we briefly recall the algebraic definition of Katz modular forms. See \cite[Section 1]{Katz} for details.
		Let $R$ be a commutative ring. A Hilbert-Blumenthal abelian variety (HBAV) over $R$ is an abelian scheme $X$ over $R$ endowed with an $\mathcal O_F$-action such that $\mathrm{Lie}(X/R)$ is a projective $R\otimes_{\ZZ}\mathcal O_F$-module of rank one.
	Let $\mathcal M(\mathfrak c,\Gamma_{00}(\mathfrak n))$ be the moduli stack of $\mathfrak c$-polarized HBAV's with $\Gamma_{00}(\mathfrak n)$-level structures.
	See \cite[(1.0)]{Katz} for the definition of a $\mathfrak c$-polarization and a $\Gamma_{00}(\mathfrak n)$-level structure on a HBAV. This moduli stack is a smooth algebraic space over $\ZZ$ when $\mathfrak n $ is divisible by an integer greater than $3$.
	Let $$\pi\colon X_{\mathrm{univ}}\to \mathcal M(\mathfrak c,\Gamma_{00}(\mathfrak n))$$ be the universal HBAV and set $$\omega:=\pi_*\Omega^1_{X_{\mathrm{univ}}/\mathcal M(\mathfrak c,\Gamma_{00}(\mathfrak n))}.$$
	By the definition of HBAV's, the group scheme $\mathrm{Res}_{\mathcal O_F/\ZZ}(\mathbb{G}_{m,\mathcal O_F})$ acts on the locally free coherent sheaf $\omega$.
	Let $\chi_{\mathbf k}\colon \mathrm{Res}_{\mathcal O_F/\ZZ}(\mathbb{G}_{m,\mathcal O_F})\times_{\mathrm{Spec}(\ZZ)}\mathrm{Spec}(\CC)\to \mathbb{G}_{m,\CC}$ be the algebraic character defined as before.
	Then $\chi_{\mathbf k}$ defines an algebraic character on $\mathrm{Res}_{\mathcal O_F/\ZZ}(\mathbb{G}_{m,\mathcal O_F})\times_{\mathrm{Spec}(\ZZ)}\mathrm{Spec}(R)$ for any commutative ring $R$ containing $F$. We fix such $R$ and define $\omega_R(\chi_{\mathbf k})$ to be the $\chi_{\mathbf k}$-isotropic component of $\bigoplus_{n\geq 0}\omega^{\otimes n}\otimes_{\ZZ}R$, which is an invertible sheaf on $\mathcal M(\mathfrak c,\Gamma_{00}(\mathfrak n))\times_{\mathrm{Spec}(\ZZ)}\mathrm{Spec}(R)$. The space of Katz modular forms defined over $R$ is then given by
	\[
	M(\mathfrak c,\Gamma_{00}(\mathfrak n),\chi_{\mathbf k};R):=H^0(\mathcal M(\mathfrak c,\Gamma_{00}(\mathfrak n))\times_{\mathrm{Spec}(\ZZ)}\mathrm{Spec}(R),\omega_R(\chi_{\mathbf k})).
	\]
	\begin{prop}[{\cite[(1.4.6), (1.6.3)]{Katz}}]
		There exists a natural isomorphism\[
		(F\otimes_{\QQ}\CC)^\times \backslash\mathrm{PL}(\mathfrak c;\mathfrak n)\xrightarrow{\cong}\mathcal M(\mathfrak c,\Gamma_{00}(\mathfrak n))(\CC)
		\]
		as complex manifolds such that the isomorphism class of HBAV's corresponding to the triple $(\mathcal L,\langle\ ,\ \rangle,i)$ is represented by $F\otimes_{\QQ}\CC/\mathcal L$. Moreover, this isomorphism induces an isomorphism
		\[
		M(\mathfrak c,\Gamma_{00}(\mathfrak n),\chi_{\mathbf k};\CC)\xrightarrow{\cong}M(\mathfrak c,\Gamma_{00}(\mathfrak n),\chi_{\mathbf k};\CC)^{\mathrm{an}}
		\]
		as $\CC$-vector spaces.
		\label{algebraicvsanalytic}
\end{prop}
Here we use the assumption $F\neq \QQ$.
Using the natural isomorphism above, we identify the algebraic Katz modular forms over $\CC$ with the analytic Katz modular forms.

	The complex manifold $\mathfrak H_F$ can be regarded as a complex submanifold of $\mathbf{GL}_F^+$ via the closed immersion
	\[
	\iota\colon \mathfrak H_F\hookrightarrow \mathbf{GL}_F^+;\quad \tau\mapsto (\tau,1).
	\]
	Then, by pullback via $\iota$, each analytic Katz modular form can be regarded as a holomorphic function on $\mathfrak H_F$.
Let $M_{\mathbf k}(\Gamma_{00}(\mathfrak c\mathcal D_{F/\mathbb Q},\mathfrak n))$ denote the space of holomorphic functions on $\mathfrak H_F$ satisfying the equation $f||_{\mathbf k}\gamma=f$ for all $\gamma\in \Gamma_{00}(\mathfrak c\mathcal D_{F/\mathbb Q},\mathfrak n)$.
\begin{lem}[{cf.\ \cite[Proposition (5.7)]{Deligne-Ribet}}]The correspondence $\mathcal F\mapsto \iota^*\mathcal F$ defines a natural isomorphism
	\[
	M(\mathfrak c,\Gamma_{00}(\mathfrak{n}),\chi_{\mathbf k};\mathbb C)^{\mathrm{an}}\xrightarrow{\sim}M_{\mathbf k}(\Gamma_{00}(\mathfrak c\mathcal D_{F/\mathbb Q},\mathfrak n)).
	\]
	\label{KatzMFandclassicalMF}
	\end{lem}
	\begin{proof}
	First, we show that $ \iota^*\mathcal F$
	 for each $\mathcal F\in M(\mathfrak c,\Gamma_{00}(\mathfrak{n}),\chi_{\mathbf k};\mathbb C)^{\mathrm{an}}$
	is invariant under the slash operator $||_{\mathbf k}\gamma$ for any $\gamma \in  \Gamma_{00}(\mathfrak c\mathcal D_{F/\mathbb Q},\mathfrak{n})$.
	For each $\gamma \in \Gamma_{00}(\mathfrak c\mathcal D_{F/\mathbb Q},\mathfrak{n})$,
	we have
	\begin{align*}
	\iota^*\mathcal F(\gamma\tau)&=\mathcal F(L_{\epsilon,\mathfrak c}(\gamma\tau,1))=\mathcal F(L_{\epsilon,\mathfrak c}((c\tau+d)^{-1}(\tau,1)^{t\! }\gamma))\\
	&=\mathcal F((c\tau+d)*L_{\epsilon,\mathfrak c}(\tau,1))=\chi_{\mathbf k}(c\tau+d)\iota^*\mathcal F(\tau).
	\end{align*}
	Therefore, the correspondence $\mathcal F\mapsto \iota^*\mathcal F$ defines the homomorphism \[
	M(\mathfrak c,\Gamma_{00}(\mathfrak{n}),\chi_{\mathbf k};\mathbb C)^{\mathrm{an}}\to M_{\mathbf k}(\Gamma_{00}(\mathfrak c\mathcal D_{F/\mathbb Q},\mathfrak n)).
	\]
	Conversely, for a given $f\in M_{\mathbf k}(\Gamma_{00}(\mathfrak c\mathcal D_{F/\mathbb Q},\mathfrak n))$, define $\mathcal F_f\colon \mathrm{PL}(\mathfrak c;\mathfrak{n})\to \CC$ by
	\[
	\mathcal F_f((\mathcal L,\langle\ ,\ \rangle,i)):=\chi_{\mathbf k}(\omega_2)f(\omega_1/\omega_2),
	\]
	where $(\omega_1,\omega_2)$ is an element of $ \mathbf{GL}_F^+$ satisfying $L_{\epsilon,\mathfrak c}(\omega_1,\omega_2)=(\mathcal L,\langle\ ,\ \rangle,i)$.
	Then, by Lemma \ref{naturalbijlem}, $\mathcal F_f$ is well-defined.
	It is clear that $\mathcal F_f$ is holomorphic and satisfies the identity (\ref{actionofGm}).
	By construction, the correspondences $f\mapsto \mathcal F_f$ and $\mathcal F\mapsto \iota^*\mathcal F$ are inverse to each other.
	This completes the proof of the lemma.
		\end{proof}
For each $\mathcal F\in M(\mathfrak c,\Gamma_{00}(\mathfrak{n}),\chi_{\mathbf k};\mathbb C)^{\mathrm{an}}$, the function $\iota^*\mathcal F$ admits a Fourier expansion
\[
\iota^*\mathcal F(\tau)=\sum_{\alpha\in \mathfrak c}b_\alpha(\mathcal F)\exp(2\pi\sqrt{-1}\mathrm{Tr}_{F/\QQ}(\alpha\tau)).
\]
We call $b_\alpha(\mathcal F)$ the Fourier coefficients of $\mathcal F$. The complex number $b_\alpha(\mathcal F)$ coincides with $b(\mathcal F,\alpha,\mathfrak c,\mathcal O_F,\epsilon)$ in \cite[(1.7.5)]{Katz}.
For each \emph{algebraic} Katz modular form $\Fcal$ in the space $M(\mathfrak c,\Gamma_{00}(\mathfrak{n}),\chi_{\mathbf k};\mathbb C)$, we define the complex number $b_\alpha(\mathcal F)$ by using the identification in Proposition \ref{algebraicvsanalytic}.
The following proposition is the key ingredient of our result.
\begin{prop}[{cf. \cite[(1.7.6), (1.2.11)]{Katz}, \cite[(6.15)]{Rapoport}, \cite[Theorem 5.4]{Lan12}}]\label{integralityKM}
For each $\mathcal F\in M(\mathfrak c,\Gamma_{00}(\mathfrak{n}),\chi_{\mathbf k};\mathbb C)$, the $\ZZ$-submodule
\[
\sum_{\alpha\in \mathfrak c}\ZZ b_\alpha(\mathcal F)
\]
of $\CC$ is a finitely generated $\ZZ$-module. In particular, if all Fourier coefficients $b_\alpha(\mathcal F)$ of $\mathcal F$ are contained in $\overline{\QQ}$, then there exists a positive integer $M$ such that all $Mb_\alpha(\mathcal F)$ are integral over $\ZZ$.
\end{prop}
\begin{proof}
Set $d=[F:\QQ]$.
	Let  $\ZZ[\![\mathfrak c;S]\!]$ be the ring of formal power series with coefficients in $\ZZ$ defined by
	a basis $l_1,\dots,l_d$ of $\Hom_{\QQ-{\rm linear}}(F,\QQ)$ such that $l_i(x)>0,\ i=1,\dots,d$ for any totally positive element $x\in F$ (\cite[(1.1)]{Katz}).
	Katz defined an algebraic Fourier expansion map
	\[
	M(\mathfrak c,\Gamma_{00}(\mathfrak{n}),\chi_{\mathbf k};\mathbb C)\to \CC[\![\mathfrak c;S]\!];\quad \mathcal F\mapsto \mathcal F(q)
	\]
	by evaluating a certain Tate object $\mathrm{Tate}_{\mathfrak c,\mathcal O_F}(q)$ (\cite[from (1.1.9) to (1.20)]{Katz}) which is a $\mathfrak c$-polarized HBAV with $\Gamma_{00}(\mathfrak n)$-level structure defined over the fraction field of $\ZZ[\![\mathfrak c;S]\!]$.
	Since this Tate object is defined over the fraction field of $\ZZ[\![\mathfrak c;S]\!]$,
	the image of the algebraic Fourier expansion map is contained in $\CC\otimes_{\ZZ}\ZZ[\![\mathfrak c;S]\!]$ (\cite[(1.2.11)]{Katz}). Write
	\[
	\mathcal F(q)=\sum_{\alpha\in \mathfrak c,\ l_i(\alpha)\geq 0}a_\alpha(\mathcal F)q^\alpha\in \CC\otimes_{\ZZ}\ZZ[\![\mathfrak c;S]\!].	\]
	The key point of the proof is that the identity
	\begin{equation}
		\label{comparisonofFC}
		a_\alpha(\mathcal F)=b_\alpha(\mathcal F)
	\end{equation}
	holds for all $\alpha\in\mathfrak c$ (\cite[(1.7.6)]{Katz}). This identity is proved by comparing the classical and the algebraic toroidal compactifications (\cite[pp.330--332]{Rapoport}).
	Such identities are proved for holomorphic automorphic forms on general PEL-type Shimura varieties in \cite[Theorem 5.4]{Lan12}.
	Thus, it suffices to show that the module
	\[
	\sum_{\alpha\in \mathfrak c,\ l_i(\alpha)\geq 0}\ZZ a_{\alpha}(\mathcal F)
	\]
	is finitely generated. Since $\mathcal F(q)\in \CC\otimes_{\ZZ}\ZZ[\![\mathfrak c;S]\!]$, there exist finitely many complex numbers $a_1,\dots,a_m$ and elements $f_1,\dots f_m$ of $\ZZ[\![\mathfrak c;S]\!]$ such that the identity
	\[
	\mathcal F(q)=\sum_{i=1}^ma_if_i
	\]
	holds. This implies that all $a_{\alpha}(\mathcal F)$ are contained in $\sum_{i=1}^m\ZZ a_i$.
	Since $\ZZ$ is Noetherian, the submodule $\sum_{\alpha}\ZZ a_{\alpha}(\mathcal F)$ of $\sum_{i=1}^m\ZZ a_i$ is also finitely generated. This completes the proof of the proposition.
\end{proof}
	\begin{lem}
		Let $\mathfrak c_1,\dots,\mathfrak c_h$ be integral ideals of $F$ representing the narrow
		ideal class group of $F$ and let $M(\mathfrak c_i,\chi_{\mathbf k},\Gamma_{00}(\mathfrak n);\mathbb C)$
		denote the space of $\mathfrak c_i$-Hilbert modular forms of weight $\chi_{\mathbf k}$ on $\Gamma_{00}(\mathfrak n)$ defined over $\mathbb C$ in the sense of Katz.
		 Then, there exists a natural inclusion
		\begin{equation}\label{inclu}
			\bigoplus_{\chi\colon (\mathcal O_F/\mathfrak n)^\times \to \mathbb C^\times} M_{\mathbf k}(\mathfrak n,\chi)\hookrightarrow \prod_{i=1}^{h}M(\mathfrak c_i,\chi_{\mathbf k},\Gamma_{00}(\mathfrak n);\mathbb C),
		\end{equation}
		which preserves the Fourier expansions.
		\label{HMtoKM}
	\end{lem}
		\begin{proof}
		Since the identity
		\[
		\Gamma_{00}(\mathfrak c_i\mathcal D_{F/\mathbb Q},\mathfrak n)=\Gamma_1(\mathfrak c_i\mathcal D_{F/\mathbb Q},\mathfrak n)\cap \SL_2(F)
				\]
		holds,
		we have a natural inclusion
		\[
		M_{\mathbf k}(\Gamma_1(\mathfrak c_i\mathcal D_{F/\mathbb Q},\mathfrak n))\subset M_{\mathbf k}(\Gamma_{00}(\mathfrak c_i\mathcal D_{F/\mathbb Q},\mathfrak n))
		\]
		which preserves the Fourier expansions.
		On the other hand, according to Proposition \ref{algebraicvsanalytic} and Lemma \ref{KatzMFandclassicalMF} (cf.\ \cite[Proposition (5.7)]{Deligne-Ribet}),
		we have the natual isomorphism
		\[
		M_{\mathbf k}(\Gamma_{00}(\mathfrak c_i\mathcal D_{F/\mathbb Q},\mathfrak n))\cong M(\mathfrak c_i,\chi_{\mathbf k},\Gamma_{00}(\mathfrak n);\mathbb C).
		\]
	Combining the above with \eqref{incl1}, we have the lemma.
	\end{proof}
	\begin{proof}[Proof of Theorem $\ref{integral model HMF}$]
	Let $f$ be an element of $M_{\mathbf k}(\mathfrak n,\chi;E)$, where $E$ is a finite number field.
	Recall that it suffices to show the boundedness of Fourier coefficients of $f$.
	Therefore, we may assume that the level $\mathfrak n$ of $f$ is divisible by an integer greater than $3$.
	By Lemma \ref{HMtoKM}, it suffices to show the boundedness of Fourier coefficients for
	each element of $M(\mathfrak c_i,\chi_{\mathbf k},\Gamma_{00}(\mathfrak n);\mathbb C)$,
	and it was proved in Proposition \ref{integralityKM}.
		This completes the proof of Theorem \ref{integral model HMF}. 
	\end{proof}
	
\begin{proof}[Proof of Theorem \ref{thm0.1}]According to Theorem \ref{integral model HMF},
	there exists an algebraic number field $E$ of finite degree contained in $\mathbb C$ such that
	\[
	M_{\bf k}(\gn, \chi)=M_{\bf k}(\gn, \chi;\mathbb C)=M_{\bf k}(\gn, \chi;\mathcal O_E)\otimes_{\mathcal O_E}\mathbb C.
	\]
	Put $C(\gm,f):=\nr(\gm)^{k_0/2}c(\gm,f)$.
	Let $T'(\gp)$ denote the Hecke operator on $M_{\bf k}(\Gamma_1(\gn))$ for a non-zero prime ideal $\gp$ of $\Ocal_F$ coprime to $\gn$, which is defined in \cite[(2.21)]{Shimura}.
	In this case, explicit relations of Fourier coefficients are given by
	\begin{align}\label{relation of FC}C(\gm, T'(\gp)f) = \sum_{\substack{\ga \subset\Ocal_F\\
				\gm+\gp\subset\ga}}\chi(\ga)\nr_{F/\QQ}(\ga)^{k_0-1}C(\ga^{-2}\gm\gp, f)
	\end{align}
	(see \cite[(2.23)]{Shimura}).
	Therefore, 
	each operator $T'(\mathfrak p)$ preserves the $\Ocal_E$-module $M_{\bf k}(\gn, \chi;\mathcal O_E)$.
	Thus, we can apply Theorem \ref{machine of int} to the case where $L=\CC$, $R=\Ocal_E$,
	$M=M_{\bf k}(\gn, \chi)$,
	$e$ is the Fourier expansion map and $T=T'(\gp)$.
	This completes the proof of Theorem \ref{thm0.1}.
\end{proof}
	
We review known results on the integrality of Hecke eigenvalues for Hilbert modular forms with Galois representations in the following two remarks.
	
	\begin{rem}\label{known for cohomological Hilbert modular}
	
		When all components of the weight ${\bf k}$ are greater than or equal to $2$ and 
		all $k_v$'s are even, 
		Hilbert modular eigenforms $f$ of weight $\bf k$ are cohomological,
		and cohomologicality yields Galois representations attached to them.
		Thus the integrality of Hecke eigenvalues can be deduced as in \cite[Proposition 4.1]{ShinTemplier},
		where Saito's integrality \cite[Proposition 3.5]{Saito} was essentially used.
		
	\end{rem}

\begin{rem}\label{known for parallel weight one}
		As for weight one modular forms, Rogawski and Tunnel \cite{RogawskiTunnel} gave an Artin representation
		associated to a primitive Hilbert cusp form of parallel weight one so that the standard or adjoint $L$-functions match.
		By \cite[Theorem 3.1]{RogawskiTunnel},
		the Satake parameter at any finite place $v$ of $F$ not dividing the level
		consists of roots of unity. Thus Hecke eigenvalues are algebraic integers for primitive Hilbert cusp forms of parallel weight one.
	\end{rem}

	\subsection{Proof of Theorem \ref{integrality for SMF}}
	\label{Hecke operators on Siegel modular forms}
	In this section, we prove Theorem \ref{integrality for SMF}.
Unlike the situation in \S\ref{proofofthm1}, the existence of an integral structure of the Siegel modular forms was already proved essentially by Taylor \cite[Lemma 2.1]{Taylor}.
	In the case of Siegel modular forms, the difficulty is to prove Theorem \ref{preserve SMF} that states Hecke operators preserve the integral structure. (Such preservation holds for Hilbert modular forms by \eqref{relation of FC}.)
For proving Theorem \ref{preserve SMF}, we
have only to prove Propositions
\ref{lem 1 preserve SMF} for $T(p^\delta)$ and \ref{lem 2 preserve SMF} for $T_{j,n-j}(p^2)$ as below.
As for Proposition
\ref{lem 1 preserve SMF} for $T(p^\delta)$, we find an explicit complete system of representatives for the set $\Sp_{2n}(\ZZ)\bsl \bar S_{p^\delta}^{(n)}$ in order to give an explicit relation among Fourier coefficients of Siegel modular forms $f$ and those of $T(p^\delta)f$.
After the calculation of Fourier coefficients of $T(p^\delta)f$,
the preservation of the integral structure by $T(p^\delta)$ is reduced to the $p$-divisibility of exponential sums $G_{a,b}(T)$ like Gauss sums indexed by certain integral matrices.
The case of Proposition
\ref{lem 2 preserve SMF} for $T_{j,n-j}(p^2)$ are similarly treated.
We assume $n\ge 2$ throughout this section.
	\begin{rem}
		Taylor's proof is based on Mumford's construction of degenerating abelian varieties over certain complete rings \cite{Mumford72} and Faltings' comparison of $q$-expansions which is a direct consequence of the construction of toroidal compactifications of Siegel modular varieties {\rm(}see Faltings-Chai's book \cite[Chapter V, Section 1]{FaltingsChai}{\rm)}. Note that a direct generalization to general PEL-type Shimura varieties was given by Lan \cite{Lan12}.
		We can show the existence of integral structures on the spaces of automorphic forms on certain PEL-type Shimura
		varieties by using Lan's comparison of $q$-expansions {\rm(}see Appendix \ref{appendix}{\rm)}.
	\end{rem}
	
	Let $\mathfrak{H}_n$ be the Siegel upper half-space of degree $n$ and
	let $\GSp_{2n}(\RR)^+$ be the group consisting of all $g \in \GSp_{2n}(\RR)$ with $\nu(g)>0$ (for the definition of $\GSp_{2n}$, see \S\ref{Notation}).
	For $N \in \NN$, we call
	$$\Gamma(N):= \Ker (\Sp_{2n}(\ZZ) \rightarrow \Sp_{2n}(\ZZ/N\ZZ))$$
	the principal congruence subgroup of level $N$.
	
	For an irreducible algebraic representation $(\rho, W_\rho)$ of $\GL_{n}(\CC)$ and a $W_\rho$-valued function $f: \mathfrak{H}_n\rightarrow W_\rho$, we set
	$f|_\rho\gamma (Z):=\rho(\nu(\gamma)^{-1/2}(CZ+D))^{-1}f(\gamma Z)$
	for $\gamma=[\begin{smallmatrix}A & B \\ C & D\end{smallmatrix}] \in \GSp_{2n}(\RR)^+$.
	We call $f$ a Siegel modular form of type $\rho$ and level $N$
	if $f$ is a holomorphic function and satisfies $f|_{\rho}[\gamma](Z)=f(Z)$ for all $\gamma \in \Gamma(N)$ and $Z\in \mathfrak{H}_n$.
	The symbol $M_\rho(\Gamma(N))$ stands for the space of Siegel modular forms of type $\rho$ and level $N$.
	By highest weight theory, $\rho$ is characterized by its highest weight $(k_1, \ldots, k_n) \in \ZZ^n$ with $k_1\ge \cdots \ge k_n$.
	Remark that
	\begin{align}\label{rho center}\rho(a 1_n)=a^{\sum_{j=1}^{n} k_j} \qquad (a>0)
	\end{align}
	holds.
	By K\"ocher's principle, the condition $k_n<0$ implies $M_\rho(\Gamma(N)) = 0$ (see \cite{Freitag1979}).
	Thus we may assume $k_n\ge 0$ from now on.
	
	The following lemma is used to prove Theorem \ref{preserve SMF}.
	\begin{lem}\label{comp of weight}
		Let $(\rho, W_\rho)$ be an irreducible algebraic representation
		of $\GL_n(\CC)$ whose highest weight is $(k_1,\ldots,k_n) \in \ZZ^n$ with $k_1\ge \cdots\ge k_n$.
		Let $\mu=(\mu_1,\ldots,\mu_n)$ be any weight of $W_\rho$.
		Then, we have $\mu_j\ge k_n$ for all $j=1,\ldots,n$.
	\end{lem}
	\begin{proof}
		The representation $\rho_0:=\rho\otimes \det^{-k_n}$ has the highest weight
		$(k_1-k_n,\ldots,k_{n-1}-k_n, 0)$, and $\rho_0$
		is realized as a subrepresentation of $\Sym^{k_1-k_n}(\CC^n)\otimes{\Sym^{k_2-k_n}(\wedge^2\CC^n)\otimes}\cdots\otimes\Sym^{k_{n-1}-k_{n}}(\wedge^{n-1}\CC^n)$, where $\CC^n$ is the standard representation of $\GL_n(\CC)$ (see \cite[\S15.5]{FultonHarris}).
		By this realization, any weight $\mu'=(\mu_1',\ldots,\mu_n')$ of
		$\rho_0$ satisfies $\mu'_j\ge 0$ for all $j=1,\ldots,n$.
		We obtain the assertion by twisting $\rho_0$ by $\det^{k_n}$.
	\end{proof}
	
By the realization of $\rho$ in the proof of Lemma \ref{comp of weight}, we may assume that $W_\rho$ is a subrepresentation of $\{\Sym^{k_1-k_n }(\CC^n)\otimes\Sym^{k_2-k_n }(\wedge^2\CC^n)\otimes\cdots\otimes\Sym^{k_{n-1}-k_{n}}(\wedge^{n-1}\CC^n)\}\otimes \det^{k_n}$.
		Set $$W_{\rho,0}=W_\rho \cap [\{\Sym^{k_1-k_n }(\ZZ^n)\otimes\Sym^{k_2-k_n }(\wedge^2\ZZ^n)\otimes\cdots\otimes\Sym^{k_{n-1}-k_{n}}(\wedge^{n-1}\ZZ^n)\} \otimes {\rm det}_{\ZZ}(\ZZ^n)^{\otimes k_n}].$$
		Then $W_{\rho,0}$ is a $\ZZ$-lattice such that $W_{\rho, 0}\otimes_\ZZ \CC=W_\rho$ and $\GL_n(\ZZ)$ acts on $W_{\rho,0}$ by $\rho$.
	Moreover, $W_{\rho,0}$ has a $\ZZ$-basis consisting of weight vectors of $\rho$. For any $\ZZ$-algebra $R$, the group $\GL_n(R)$ acts on the
	$R$-module $W_{\rho,0}(R):=W_{\rho,0}\otimes_\ZZ R$.

	Let us introduce Hecke operators on Siegel modular forms (cf.\ \cite{Evdokimov}).
	For $m \in \NN$ coprime to $N$, set
	\begin{align}\label{Sbar}\bar{S}_m^{(n)}(N)=\{ g \in {\rm M}_{2n}(\ZZ)\cap \GSp_{2n}(\QQ) \mid g \equiv \left[\begin{smallmatrix}
			1_n& \\ & \nu(g)1_n \end{smallmatrix}\right] \pmod N, \ \nu(g)=m \ \}.\end{align}
	Let $\bfL(N)_\CC$ be the abstract Hecke algebra over $\CC$
	generated by $\Gamma(N)g\Gamma(N)$ for all $g \in \bar S_m^{(n)}(N)$ with $m \in \NN$ coprime to $N$ (see \cite[p.433]{Evdokimov}).
	We define the action of $\Gamma(N)g\Gamma(N)$ for $g \in \bar S_m^{(n)}(N)$ by
	{\allowdisplaybreaks\begin{align*}
			t_{\rho,N}^n(\Gamma(N) g \Gamma(N))f(Z) := &
			\, \nu(g)^{\frac{\sum_{j=1}^{n}k_j}{2}-\frac{n(n+1)}{2}}\sum_{\gamma= \left[\begin{smallmatrix}A&B\\C&D \end{smallmatrix}\right]\in \Gamma(N)\bsl \Gamma(N) g \Gamma(N)} f|_\rho\gamma(Z) \\
			= & \, \nu(g)^{\{\sum_{j=1}^{n}k_j\}-\frac{n(n+1)}{2}}
			\sum_{\gamma = \left[\begin{smallmatrix}A&B\\C&D \end{smallmatrix}\right]
				\in \Gamma(N)\bsl \Gamma(N) g \Gamma(N)}
			\rho(CZ+D)^{-1}f(\gamma Z),
		\end{align*}
	}where we use $f|_\rho\gamma (Z)=\rho(\nu(\gamma)^{-1/2}(CZ+D))^{-1}f(\gamma Z)$ and \eqref{rho center} for the second equation.
	Then, $t_{\rho, N}^n : {\bf L}(N)_\CC \rightarrow \End(M_\rho(\Gamma(N)))$
	is a $\CC$-algebra homomorphism.
	For any prime number $p$ coprime to $N$, we also define the Hecke operators $T(p)$, $T_{j,n-j}(p^2)$ $(j=1,\ldots,n)$ and $T(p^\delta)$ ($\delta \in \NN$) on $M_\rho(\Gamma(N))$ by
$$T(p):=t_{\rho,N}^n(\Gamma(N)\left[\begin{smallmatrix} 1_n & \\ & p1_n \end{smallmatrix}\right]\Gamma(N)),$$
	$$T_{j,n-j}(p^2):=t_{\rho,N}^{n}\left(\Gamma(N) g_j(p)\left[\begin{smallmatrix} p1_j & & & \\
		& 1_{n-j} & & \\
		& & p 1_j & \\
		&  & & p^21_{n-j} \end{smallmatrix}\right]\Gamma(N)\right),$$
	$$T(p^{\delta}):=\sum_{g \in \Gamma(N)\bsl \bar S_{p^\delta}^{(n)}(N)/\Gamma(N)}t_{\rho,N}^{n}(\Gamma(N)g\Gamma(N)).$$
Here $g_j(p) \in \Sp_{2n}(\ZZ)$ is a fixed element so that
$$g_j(p) \equiv \left[\begin{smallmatrix} p^{-1}1_j & & & \\
		& 1_{n-j} & & \\
		& &  p1_j & \\
		&  & & 1_{n-j} \end{smallmatrix}\right] \pmod N.$$
			We can take such a $g_j(p)$ by the surjectivity of $\Sp_{2n}(\ZZ) \rightarrow \Sp_{2n}(\ZZ/N\ZZ)$ (cf.\ \cite[Lemma 3.3.2]{AndrianovBook}).
We note that the abstract Hecke algebra $\bfL(N)_\CC$ is generated by
$T(p)$ and $T_{j,n-j}(p^2)$ for all prime numbers $p$ coprime to $N$ and all $j\in \{1,\ldots, n\}$ as a $\CC$-algebra (see \cite[Theorem 1.1 and Corollary]{Evdokimov} and \cite[Theorem 3.3.23 (1)]{AndrianovBook}).
Also remark that the third operator $T(p^\delta)$ for $\delta=1$ coincides with the first operator $T(p)$ by \cite[Lemma 1.1, 1)]{Evdokimov} and Lemma \ref{comp rep for S(N)} below.

	For $m \in \NN$ such that $\gcd(m,N)=1$ and for $j \in \NN$ with $0\le j\le n$,
	let $g_j(m) \in \Sp_{2n}(\ZZ)$ be an element such that
\begin{align}\label{gj(m)}
	g_j(m) \equiv \left[\begin{smallmatrix} m^{-1}1_j & & & \\
	& 1_{n-j} & & \\
	& &  m1_j & \\
	&  & & 1_{n-j} \end{smallmatrix}\right] \pmod N.
\end{align}
as in \cite[p.437]{Evdokimov}.
	For an $n$-tuple $\chi=(\chi_1,\ldots,\chi_n)$ of Dirichlet characters modulo $N$,
	let $M_{\rho}(\Gamma(N),\chi)$ denote the subspace consisting of all $f \in M_\rho(\Gamma(N))$
	such that $f|_{\rho} g_j(m) = \chi_j(m)f$ for all $m \in \NN$ with $\gcd(m, N)=1$ and
	all $j \in \{1,\ldots,n\}$.
	Since the diagonal subgroup of $\Sp_{2n}(\ZZ/N\ZZ)$
is a finite abelian group and	acts on $M_{\rho}(\Gamma(N))$,
	we have a Hecke equivalent decomposition \begin{align}\label{chi decomp}M_{\rho}(\Gamma(N))=\bigoplus_{\chi=(\chi_1,\ldots,\chi_n)}M_{\rho}(\Gamma(N),\chi),
	\end{align}
	where $\chi_j$ for each $j\in\{1,\ldots,n\}$ runs over all Dirichlet characters modulo $N$ (see \cite[p.437]{Evdokimov}).
Hence $t_{\rho, N}^{n}$ induces a $\CC$-algebra morphism $t_{\rho, N, \chi}^n : {\bf L}(N)_\CC \rightarrow \End(M_\rho(\Gamma(N),\chi))$.
	We remark that the $n$th component $\chi_n$ of $\chi$ corresponds to the central character of the irreducible automorphic representations of $\GSp_{2n}(\AA_\QQ)$ generated by the adelic lift of a modular form in $M_\rho(\Gamma(N),\chi)$. One can check this fact with the aid of \cite[\S2.3]{KimWY} for $n=2$ and \cite[\S 3]{Saha} for the scalar valued case, although we omit the detail.
	
	Fix $N$ and $\rho$ throughout this subsection.
	Any $f \in M_{\rho}(\Gamma(N))$ has the Fourier expansion
	\begin{align}\label{Fourier exp of Siegel}f(Z)=\sum_{S \in \bfA_n}a(S, f)\exp(2\pi\sqrt{-1} \tfrac{1}{N}\tr(SZ)),
		\end{align}
	where $\bfA_n$ is the set of all symmetric positive semi-definite half-integral $n\times n$ matrices.
	For a subring $\Ocal$ of $\CC$, we consider the $\Ocal$-submodule $M_{\rho}(\Gamma(N); \Ocal)$ of $M_{\rho}(\Gamma(N))$ consisting of $f \in M_{\rho}(\Gamma(N))$ whose Fourier coefficients
	$a(T,f)$ at all $T \in \bfA_n$ are contained in $W_{\rho,0}(\Ocal)$.
	Then we have the existence of an integral structure of
	 $M_{\rho}(\Gamma(N), \chi)$ for general $\chi$ as follows.

	\begin{lem}\label{existence of int model of chi}
				Let $\Ocal$ be the integer ring of a number field of finite degree
		which contains all values of Dirichlet characters modulo $N$.
		Let $\chi = (\chi_1,\ldots,\chi_n)$ be a character of the diagonal subgroup of $\Sp_{2n}(\ZZ/N\ZZ)$.
		We have
		$$M_{\rho}(\Gamma(N), \chi; \Ocal) \otimes_{\Ocal} \CC=M_{\rho}(\Gamma(N), \chi),$$
		where we set $M_{\rho}(\Gamma(N), \chi; \Ocal) := M_{\rho}(\Gamma(N), \chi)\cap M_\rho(\Gamma(N);\Ocal)$.
	\end{lem}
	\begin{proof}
Set $\vec{n}:= (k_n,\ldots,k_1)$.
In a part of the proof of $M_{\rho}(\Gamma(N); \ZZ)\otimes_{\ZZ} \CC= M_{\rho}(\Gamma(N))$ in \cite[Lemma 2.1]{Taylor}, 
the image of $\bigotimes^{\vec{n}}\omega_{\Acal/\Mcal}(\Mcal)$ in $M_\rho(\Gamma(N))$ spans $M_\rho(\Gamma(N))$ and each element has a Fourier expansion with coefficients in $W_{\rho,0}(\ZZ)$.\footnote{The $R$-module $\bigotimes^{\vec{n}}R^g$ for a commutative unital ring $R$ in \cite[\S2.2]{Taylor} is different from our $W_{\rho,0}(R)$ even if $\vec{n}$ is replaced with $(k_n,\ldots,k_1)$. This causes no problems as they are isomorphic to each other. We also note that the Young symmetrizer $c$ in \cite[p.24]{Taylor} is not correct. The fourth property of $c$ should be
``$\sigma c=(-)^{\sigma}c$ (where $(-)^{\sigma}$ denotes the sign of sigma)
if $\sigma$ preserves sets of the form
$\{\sum_{k=1}^{i_j-1}n_k+j, \sum_{k=1}^{i_j}n_k+j, \sum_{k=1}^{i_j+1}n_k+j,\ldots,
\sum_{k=1}^{g-1}n_k+j \}$ where $i_j$ is the least index such that
$n_{i_j}\ge j$ (where $n_0=0$).''}
In our setting, the finite abelian group $((\ZZ/N\ZZ)^\times)^n$ acts on $\bigotimes^{\vec{n}}\omega_{\Acal/\Mcal}(\Mcal)$ as in the case of $M_\rho(\Gamma(N))$. Hence
$\bigotimes^{\vec{n}}\omega_{\Acal/\Mcal}(\Mcal)$ is decomposed into the sum of the $\chi$-isotypic components as an $E$-module similarly to \eqref{chi decomp}, where $E$ is the fractional field of $\Ocal$.
Since the map from $\bigotimes^{\vec{n}}\omega_{\Acal/\Mcal}(\Mcal)$ to $M_\rho(\Gamma(N))$ is compatible with the decomposition into the $\chi$-isotypic components, we obtain that $M(\Gamma(N), \chi; \Ocal)$ contains a basis of the $\CC$-vector space $M(\Gamma(N), \chi)$. This completes the proof.
	\end{proof}
	
	Let ${\bf T}_{n,\ZZ}^{\rho, N, \chi}$ be the $\ZZ$-algebra of 
	$\End(M_{\rho}(\Gamma(N), \chi))$
	generated by
	$
	p^{\delta(k_n<n)\frac{(n-k_n)(n-k_n+1)}{2}}T(p)$
	and $p^{\delta(k_n\le n)n(n-k_n+1)}T_{j,n-j}(p^2)$
	for all prime numbers $p$ with $p\notdivide N$ and  for all $1\le j \le n$.

	\begin{thm}\label{preserve SMF}	Let $\Ocal$ and $\chi$ be as in Lemma \ref{existence of int model of chi}.
		Then, any element of $\bfT_{n,\ZZ}^{\rho, N, \chi}$ preserves $M_\rho(\Gamma(N),\chi; \Ocal)$.
	\end{thm}
	This theorem follows from Propositions \ref{lem 1 preserve SMF} and
	\ref{lem 2 preserve SMF} below.
	\begin{prop}\label{lem 1 preserve SMF}
		Let $\Ocal$ and $\chi$ be
as in Lemma \ref{existence of int model of chi}.
		For any prime  $p\notdivide N$ and any $\delta \in \NN$,
				 the Hecke operator $(p^\delta)^{\delta(k_n<n)\frac{(n-k_n)(n-k_n+1)}{2}}T(p^\delta)$
		preserves $M_\rho(\Gamma(N),\chi; \Ocal)$.
	\end{prop}
	\begin{prop}\label{lem 2 preserve SMF}
		Let $\Ocal$ and $\chi$ be as in Lemma \ref{existence of int model of chi}.
			For any prime $p\notdivide N$ and
		any $j\in \{1, \ldots, n\}$,
		the Hecke operator
		$p^{\delta(k_n\le n)n(n-k_n+1)}T_{j,n-j}(p^2)$
	preserves $M_\rho(\Gamma(N),\chi; \Ocal)$.
	\end{prop}

	For our purpose, we need only Proposition \ref{lem 1 preserve SMF} for $\delta=1$.
	However, we keep Proposition \ref{lem 1 preserve SMF} being the statement for all $\delta \in \NN$.
	
	For proving Proposition \ref{lem 1 preserve SMF}, we shall explicitly give a complete system of representatives for $\Gamma(N)\bsl \bar S_{p^\delta}^{(n)}(N)$.
	Set $\Gamma=\Sp_{2n}(\ZZ)$.
	For a non-negative integer $\delta\ge0$ and a prime power $p^\delta$ coprime to $N$, we put
	$$S_{p^{\delta}}^{(n)}=\{ g \in {\rm M}_{2n}(\ZZ)\cap \GSp_{2n}(\QQ) \mid \nu(g)=p^\delta \}.$$ 
	As in \cite[p.438]{Evdokimov}
	and \cite[(2.1)]{Zharkovskaya}, a complete system of representatives for $\Gamma \bsl S_{p^{\delta}}^{(n)}$ is
	given as
	\begin{align*}& V(p^\delta)=\{ [\begin{smallmatrix}
			A & B \\ 0_n & D
		\end{smallmatrix}] \in {\rm M}_{2n}(\ZZ)\ | \ D \in \diag(p^{\a_0}, p^{\a_0+\a_1}, \ldots,p^{\a_0+\cdots+\a_{n-1}})R(p^{\a_0},\ldots,p^{\a_{n-1}}),\\
		& \qquad \a_0+\cdots+\a_{n-1}\le \delta, \a_0\ge 0, \ldots,\a_{n-1}\ge 0,
		A =p^{\delta}({}^tD^{-1}), {}^{t}BD ={}^tDB, B\, (\text{mod } D)\}.
	\end{align*}
	Here $B\, (\text{mod } D)$ means that $B$ runs over a complete system of representatives for the set ${\rm M}_n(\ZZ)/{\sim_D}$, where $\sim_D$ is the equivalence relation defined by writing $B\sim_D B'$ for $B-B'\in 
	{\rm M}_n(\ZZ)D$.
Furthermore, the set $R(p^{\a_0}, \ldots, p^{\a_{n-1}})$ is a complete system of representatives
	for $$(\SL_n(\ZZ)\cap \diag(p^{\beta_1},\ldots, p^{\beta_n})^{-1}\SL_n(\ZZ)\diag(p^{\beta_1},\ldots,p^{\beta_n})) \bsl \SL_n(\ZZ),$$
	where we put $\beta_j=\sum_{l=0}^{j-1}\a_l$ for each $j$.

For any $M\in \NN$, set
	$$\Gamma_{\SL_n}(M)=\Ker(\SL_n(\ZZ)\rightarrow \SL_n(\ZZ/M\ZZ)).$$
By $\gcd(p,N)=1$, we have $\Gamma_{\SL_n}(p^{\b_n})\Gamma_{\SL_n}(N)=\SL_n(\ZZ)$ (cf.\ \cite[(1.8)]{Evdokimov}).
Combining this with
	$$\Gamma_{\SL_n}(p^{\beta_n})\subset \SL_n(\ZZ)\cap \diag(p^{\beta_1},\ldots, p^{\beta_n})^{-1}\SL_n(\ZZ)\diag(p^{\beta_1},\ldots,p^{\beta_n}),$$
we can fix $R(p^{\a_0}, \ldots, p^{\a_{n-1}})$ so that $R(p^{\a_0}, \ldots, p^{\a_{n-1}}) \subset \Gamma_{\SL_n}(N)$.
 Recall that the set $\bar S_{m}^{(n)}(N)$ for $m \in \NN$ with $\gcd(m,N)=1$ is defined in \eqref{Sbar}.
	We have the following lemma in the same way as \cite[Proposition 3.1]{Evdokimov} for the case $n=2$.
	\begin{lem}\label{comp rep for S(N)}
		Let $p$ be a prime number coprime to $N$.
		Let $\delta\ge 0$ be a non-negative integer.
		Then, the following is a complete system of representatives for $\Gamma(N)\bsl \bar S_{p^\delta}^{(n)}(N)${\rm:}
		\begin{align*} V_N(p^\delta)=&\{g_1(p^{\a_1})\cdots g_n(p^{\a_n}) [\begin{smallmatrix}
				A & N B \\ 0_n & D
			\end{smallmatrix}] \mid A=p^\delta({}^{t}D^{-1}), \\
			& D \in
			\diag(p^{\a_0}, p^{\a_0+\a_1}, \ldots, p^{\a_0+\cdots +\a_{n-1}})
			R(p^{\a_1}, \ldots, p^{\a_{n-1}}), \\
			&  {}^tBD={}^{t}DB, B\, \text{\rm mod }D,\quad \a_0,\ldots,\a_{n-1}\ge 0,\
			\a_0+\cdots+\a_{n-1}\le \delta
			\},
		\end{align*}
		where $\a_n:=\delta-(\a_0+\cdots+\a_{n-1})$ and each $g_j(p^{\alpha_j})$ is an element of $\Sp_{2n}(\ZZ)$ defined by \eqref{gj(m)}.
	\end{lem}

	\begin{proof}[Proof of Proposition \ref{lem 1 preserve SMF}]
		An essential idea of the proof is the same as in \cite{Mizumoto}.
		Let $f$	be an element of $M_{\rho}(\Gamma(N), \chi; \Ocal)$.
		Let $p$ be a prime number coprime to $N$.
		Let us consider the action of $T(p^\delta)$ for $\delta \in \NN$ (for Theorem \ref{preserve SMF}, the case $\delta=1$ is enough). The function
		$$T(p^{\delta})f (Z) =\sum_{g \in \Gamma(N)\bsl \bar S_{p^\delta}^{(n)}(N)/\Gamma(N)}t_{\rho,N}^{n}(\Gamma(N)g\Gamma(N))f(Z)$$
		has the following Fourier expansion
		\begin{align*}
			& (p^{\delta})^{\{\sum_{j=1}^{n}k_j\}-\frac{n(n+1)}{2}}\sum_{\a_0,\ldots,\a_{n-1}, B, D}\bigg\{\prod_{j=1}^{n}\chi_j(p^{\a_j})\bigg\}\rho(D)^{-1}f((p^\delta({}^{t}D^{-1})Z+NB)D^{-1}) \\
			= & (p^{\delta})^{\{\sum_{j=1}^{n}k_j\}-\frac{n(n+1)}{2}}\sum_{T \in \bfA_n}\sum_{\a_0,\ldots,\a_{n-1}, B, D}\bigg\{\prod_{j=1}^{n}\chi_j(p^{\a_j})\bigg\}\\
			&\times\rho(D)^{-1}a(T,f)\exp\left(2\pi\sqrt{-1}\frac{1}{N}\tr(D^{-1}Tp^{\delta}({}^{t}D^{-1})Z)\right) \exp(2\pi\sqrt{-1}\tr(TBD^{-1}))
		\end{align*}
		by virtue of \eqref{Fourier exp of Siegel} and the decomposition
		$$\Gamma(N)\bsl \bar S_{p^\delta}^{(n)}(N) =
		\coprod_{\substack{\a_0,\ldots,\a_{n-1}, B, D}} 
		\Gamma(N)g_1(p^{\a_1})\cdots g_n(p^{\a_n})[\begin{smallmatrix}
			p^{\delta}({}^tD^{-1}) & N B \\ 0 & D
		\end{smallmatrix}],$$
		where $\a_0,\ldots,\a_{n-1}$, $B$ and $D$ vary as in Lemma \ref{comp rep for S(N)}.
		For any $T \in \bfA_n$, the Fourier coefficient $a(T, T(p^\delta)f)$
		is equal to
		\begin{align*}& \sum_{\a_0,\ldots, \a_{n-1}}\bigg\{\prod_{j=1}^{n}\chi_j(p^{\a_j})\bigg\} \\
			& \times (p^{\delta})^{\{\sum_{j=1}^{n}k_j\}-\frac{n(n+1)}{2}}\sum_{B, D}\rho(D)^{-1}
			\exp(2\pi\sqrt{-1}\tr(p^{-\delta}T[{}^tD]BD^{-1}))
			a(p^{-\delta}T[{}^tD],f).
		\end{align*}
		Here $T[{}^tD]=DT\,{}^{t}D$, and we regard $a(S,f)=0$ unless $S\in \bfA_n$.
		For $S \in \bfA_n$, set
		$$\Gcal(S,D) = \sum_{B}\exp(2\pi\sqrt{-1}\tr(SBD^{-1})).$$
		This is evaluated as follows.
		We can take $U,V \in \GL_n(\ZZ)$ such that
		$UDV=\diag(d_1,\ldots, d_n)$,
		where $d_j \in \ZZ$ satisfies $d_j|d_{j+1}$ for all $j \in \{1,\ldots, n-1\}$.
		Put $[\frac{1}{2}(1+\delta_{\mu\nu})s_{\mu\nu}]_{1\le \mu,\nu\le n} =S[^{t}U]$.
		Then, as in the proof of \cite[Theorem 1.1]{Mizumoto}
		(cf.\ \cite[Kapitel IV, Hilfssatz 5.8]{Freitag}), we have
		$$\Gcal(S,D)=\begin{cases}
			\prod_{j=1}^{n}d_j^{n-j+1} \qquad (\text{if } d_\nu|s_{\mu\nu}, 1\le\mu\le\nu\le n),\\
			0 \qquad \qquad \qquad\quad (\text{otherwise}).
		\end{cases}$$
		For proving $(p^\delta)^{\delta(k_n<n)\frac{(n-k_n)(n-k_n+1)}{2}}a(T, T(p^\delta)f) \in W_{\rho,0}(\Ocal)=W_{\rho, 0}\otimes_\ZZ\Ocal$, it is sufficient to show the claim
		$$(p^\delta)^{\delta(k_n<n)\frac{(n-k_n)(n-k_n+1)}{2}}\times (p^\delta)^{\{\sum_{j=1}^n k_j\}-\frac{n(n+1)}{2}}\rho(D)^{-1}\Gcal(S,D) \in \End_\ZZ(W_{\rho,0})$$
		for all $S \in \bfA_n$ such that $\Gcal(S, D)\neq 0$.
		This claim is proved as follows.
		By $V^{-1}(p^\delta D^{-1})U^{-1}=\diag(a_1,\ldots,a_n)$, where $a_j \in \ZZ$ satisfies
		$a_jd_j=\nu(g)=p^{\delta}$, we have
		\begin{align*}
			(p^\delta)^{\{\sum_{j=1}^n{k_j}\}-\frac{n(n+1)}{2}}\rho(D^{-1})\Gcal(S,D)
			=\, & \rho(p^\delta D^{-1}) \prod_{j=1}^{n}(p^{\delta})^{-(n-j+1)}d_j^{n-j+1}
			\\
			=\, & \rho(V)\bigg\{\rho(\diag(a_1,\ldots a_n)) \prod_{j=1}^{n}a_j^{-n+j-1}\bigg\}\rho(U). \notag
		\end{align*}
		Since $\GL_n(\ZZ)$ acts on $W_{\rho, 0}$, both $\rho(U)$ and $\rho(V)$ are elements of $\End_{\ZZ}(W_{\rho,0})$.
Let $v$ be any weight vector of $W_{\rho,0}$ with weight $(\mu_1,\ldots, \mu_n)$.
	      Then we have
	      $$w:= \rho(\diag(a_1,\ldots,a_n))\bigg\{\prod_{j=1}^{n}a_j^{-n+j-1}\bigg\}v
	      =\bigg\{\prod_{j=1}^{n}a_j^{\mu_j+(-n+j-1)}\bigg\} v.$$
Since $W_{\rho,0}$ has a $\ZZ$-basis consisting of weight vectors as in the definition, it is enough to consider such a $w$.
If $k_n \ge n$, then we have $\mu_j+(-n+j-1) \ge k_n-n\ge 0$
for any $j\in\{1,\ldots,n\}$ by Lemma \ref{comp of weight}. This leads us to $w \in W_{\rho,0}$.
If $k_n<n$,
we set $t=n-k_n$. Then we note $1\le t \le n$.
For any $j \in \{1,\ldots, t\}$, we have $(p^\delta)^{t-j+1}a_j^{\mu_j+(-n+j-1)} \in \ZZ$
by $t-j+1+\mu_j+(-n+j-1) =\mu_j-k_n \ge 0$ with the aid of Lemma \ref{comp of weight}.
Furthermore, we have $a_j^{\mu_j+(-n+j-1)} \in \ZZ$ for any $j\ge t+1$ by $\mu_j+(-n+j-1)\ge \mu_j-k_n\ge0$.
Hence $\{\prod_{j=1}^{t}(p^\delta)^{t-j+1} \}w = (p^\delta)^{\frac{t(t+1)}{2}}w \in W_{\rho,0}$ holds.

As a result, the claim follows.
\end{proof}

	\begin{proof}[Proof of Proposition \ref{lem 2 preserve SMF}]
		We use the same symbols as in the proof of Proposition \ref{lem 1 preserve SMF}.
		The assertion is proved in the same way as \cite{Mizumotocorrection} as below.
		Let $f$	be an element of $M_{\rho}(\Gamma(N), \chi; \Ocal)$ and $p$ a prime number coprime to $N$.
Let us fix $j\in \{1,\ldots,n\}$. For our purpose, it is enough to show
		$$p^{\delta(k_n\le n)n(n-k_n+1)}a(T,T_{j,n-j}(p^2)f) \in W_{\rho,0}(\Ocal)$$
		for all $T\in \bfA_n$. We divide the proof into two cases ($N=1$ and $N\ge 2$).
		
		First, let us consider the case $N=1$ (i.e., $\Gamma(N)=\Gamma=\Sp_{2n}(\ZZ)$). In the same way as \cite{Mizumotocorrection}, the set
		$$\Gamma\bsl\Gamma \diag(p1_j, 1_{n-j}, p 1_j, p^{2}1_{n-j}) \Gamma
		$$
		has a complete system of representatives
		\begin{align}\label{complete rep for Tj}\left\{ \left[\begin{matrix}
			p^2D_{a,b}^{-1} & B \\ 0_n & D_{a,b}
		\end{matrix}\right]\left[\begin{matrix}
			{}^{t}u^{-1} & 0_n \\ 0_n & u
		\end{matrix}\right] \, \bigg| \, a,b \in \ZZ_{\ge0},\, a+b\le n, \, a\ge j, \, B \in X_{j, a,b},\,
		u\in Y_{a,b}\right\},
		\end{align}
		where we set
		$D_{a,b}=\diag(
		1_{n-a-b}, p1_a, p^2 1_b)$,
		$X_{j, a,b}$ is a complete system of representatives for
		$$\left\{ B \in {\rm M}_n(\ZZ) \ \middle| \ \left[\begin{matrix}
			p^2D_{a,b}^{-1} & B \\ 0_n & D_{a,b}
		\end{matrix}\right] \in \Gamma \diag(p1_j, 1_{n-j}, p 1_j, p^{2}1_{n-j}) \Gamma \right\}\bigg/ \text{mod}\ {D}_{a,b},$$
		and $Y_{a,b}$ is a complete system of representatives for $(\SL_n(\ZZ)\cap D_{a,b}^{-1}\SL_n(\ZZ)D_{a,b})\bsl\SL_n(\ZZ)$, respectively.
		By \cite[(3.3.67)]{AndrianovBook} and the mod $p$ rank condition on $B$, we may take $X_{j,a,b}$ as
		\begin{align*}\Bigg\{B=\left[\begin{matrix}
			0 & 0 & 0 \\
			0 & B_{22} & p\,{}^{t}B_{32} \\
			0 & B_{32} & B_{33}
		\end{matrix}\right] \, \Bigg| \, & B_{22} \in {\rm S}_a(\ZZ)/\text{mod }p, \ B_{32}\in {\rm M}_{b,a}(\ZZ)/\text{mod } p, \\ & B_{33} \in {\rm S}_b(\ZZ)/\text{mod }p^2, \ r_p(B_{22})=a-j \Bigg\},
	\end{align*}
		where ${\rm S}_m(\ZZ)$ for $m \in \NN$ denotes the set of all symmetric matrices in ${\rm M}_m(\ZZ)$, and ${\rm M}_{b,a}(\ZZ)$ denotes the set of all $b\times a$ integral matrices, respectively.
	Furthermore, $r_p(B_{22})$ denotes the rank of $B_{22}$ over $\ZZ/p\ZZ$.
	We note that the condition $r_p(B_{22})=a-j$ occurs from the requirement
	$$b+r_p(B_{22})+(n-a-b) =r_p\left(\left[\begin{matrix}
		p^2D_{a,b}^{-1} & B \\ 0_n & D_{a,b}
	\end{matrix}\right]\right)=r_p(\diag(p1_j, 1_{n-j}, p 1_j, p^{2}1_{n-j}) )=n-j.$$
		The condition $a\ge j$ in \eqref{complete rep for Tj} is required by $a-j=r_p(B_{22}) \ge 0$.
		
		For any $T \in \bfA_n$ and $a,b\in\ZZ_{\ge0}$ such that $a+b\le n$ and $a\ge j$, set
		$$G_{a,b}(T)=\sum_{B \in X_{j, a,b}}\exp(2\pi\sqrt{-1}\tr(TBD_{a,b}^{-1})).$$
		As for $G_{a,b}(T)$, we know $G_{a,b}(T) \in \ZZ$ and $\ord_p(G_{a,b}(T)) \ge b(a+b+1)$ by \cite{Mizumotocorrection}
		(cf.\ the argument in \cite[p.158]{AndrianovBook} and the rationality result \cite[p.244, Lemma 4.2.25]{AndrianovBook}).
		
In the same way as the Fourier expansion of $T(p^\delta)f$, $a(T,T_{j,n-j}(p^2)f)$ is equal to
\begin{align}\label{Fourier exp of Tj}\sum_{\substack{a\ge j ,b\ge0 \\ a+b\le n}}\sum_{u \in Y_{a,b}}(p^2)^{\{\sum_{l=1}^{n}k_j\}-\frac{n(n+1)}{2}}\rho(D_{a,b}u)^{-1}G_{a,b}(p^{-2}T[{}^{t}u\,{}^{t}D_{a,b}])a(p^{-2}T[{}^{t}u\,{}^{t}D_{a,b}], f),
	\end{align}
where we regard $a(S,f)=0$ unless $S \in \bfA_n$. By $\rho(u) \in \End_\ZZ(W_{\rho,0})$ and $\ord_p(G_{a,b}(S)) \ge b(a+b+1)$ for all $S \in \bfA_n$, it is enough to show
$$p^{\delta(k_n\le n)n(n-k_n+1)}\times p^{\{\sum_{l=1}^{n}2k_l\}-n(n+1)+b(a+b+1)} \rho(D_{a,b})^{-1}
\in \End_\ZZ(W_{\rho,0})$$
for all $a, b \in \ZZ_{\ge 0}$ with $a+b\le n$ and $a\ge1$ for our purpose.
Let $v$ be any weight vector of $W_{\rho,0}$ with weight $(\mu_1,\ldots, \mu_n)$.
Then the equality $$p^{\sum_{l=1}^{n}2k_l}\rho(D_{a,b})^{-1}v
=\rho(p^2D_{a,b}^{-1})v
=p^{\sum_{l=1}^{n-a-b}2\mu_l+\sum_{l=n-a-b+1}^{n-b}\mu_l}v$$
holds. With the aid of Lemma \ref{comp of weight}, we have
$\sum_{l=1}^{n-a-b}2\mu_l+\sum_{l=n-a-b+1}^{n-b}\mu_l \ge (2n-a-2b)k_n$.
Therefore, it is sufficient to show the claim
$$b(a+b+1) - n(n+1) + (2n-a-2b)k_n+\delta(k_n\le n)n(n-k_n+1) \ge 0.$$
We put $F_b(a):=b(a+b+1) - n(n+1) + (2n-a-2b)k_n,$
where $a,b \in \ZZ_{\ge 0}$ vary such that $a+b\le n$ and $a\ge1$.

When $k_n>n$,
For any fixed $b$, $F_b(a)$ is minimal when $a=n-b$ because of $b-k_n<b-n<0$.
		Hence, by noting $F_b(a)\ge F_b(n-b) = (n-b)(k_n-n-1)\ge0$, the claim follows.
		
		When $k_n\le n$, first we consider the case $b\ge k_n$. In that case we give
		\begin{align*}
		F_b(a)\ge & \, F_b(1) = b(b+2)-n(n+1)+(2n-1-2b)k_n, \qquad 1\le a \le n-b.
	\end{align*}
		As $F_b(1)$ on $k_n\le b \le n$ is minimal at $b=k_n$, we obtain $F_b(a)\ge -(n-k_n)(n-k_n+1)$.
		If $b < k_n$, we have $F_b(a)\ge F_b(n-b) = -(n-b)(n-k_n+1)\ge -n(n-k_n+1)$.
		Hence we have the claim.
		
		Second, let us consider the case $N\ge 2$.
		In a similar fashion to the construction of $V_N(p^\delta)$ in Lemma \ref{comp rep for S(N)},
		with the aid of \eqref{complete rep for Tj} for the case $N=1$ and \cite[Lemma 1.1, 3)]{Evdokimov},
		we find that the set
		$$\Gamma(N)\bsl\Gamma(N) g_j(p)\diag(p1_j, 1_{n-j}, p 1_j, p^{2}1_{n-j}) \Gamma(N)
		$$
		has a complete system of representatives 
		\begin{align*}\bigg\{ g_{n-a-b}(p) g_{n-b}(p)\left[\begin{matrix}
				p^2D_{a,b}^{-1} & NB \\ 0_n & D_{a,b}
			\end{matrix}\right]\left[\begin{matrix}
				{}^{t}u^{-1} & 0_n \\ 0_n & u
			\end{matrix}\right] \ \bigg| \ & a,b \in \ZZ_{\ge0},\ a+b\le n,\ a\ge j, \\ & B \in X_{j, a,b},\ u\in Y_{a,b} \bigg\}.
		\end{align*}
		Here
			$D_{a,b}=\diag(
		1_{n-a-b}, p1_a, p^2 1_b)$ and
		$X_{j, a,b}$ and $Y_{a,b}$ are the sets given in the case of $N=1$.
		We may assume $Y_{a,b}\subset \Gamma_{\SL_n}(N)$ since $p$ is coprime to $N$.		
		In the same way as the Fourier coefficients of $T_{j,n-j}(p^2)f$ at $T$ for $N=1$,
			$a(T,T_{j,n-j}(p^2)f)$ is equal to \eqref{Fourier exp of Tj} modified by adding the factor $\chi_{n-a-b}(p)\chi_{n-b}(p) \in \Ocal$ to the term at each $(a,b)$.
		Hence the argument in the proof for $N=1$ works even when $N\ge2$. This completes the proof of $p^{\delta(k_n\le n)n(n-k_n+1)}a(T,T_{j,n-j}(p^2)f)\in W_{\rho,0}(\Ocal)$.
	\end{proof}
	
	\begin{proof}[Proof of Theorem \ref{integrality for SMF}]
		By Lemma \ref{existence of int model of chi} and Theorem \ref{preserve SMF}, 
		we have only to apply Theorem \ref{machine of int} to $L=\CC$, $R=\Ocal$ and $M=M_\rho(\Gamma(N), \chi)$.
	\end{proof}
	
	In particular, we have the following.
	
	\begin{cor}\label{cor kn n+1}
		Suppose $k_n\ge n+1$.
		For any prime numbers $p$ with $p\notdivide N$and any $j\in\{1,\ldots,n\}$,		
		all Hecke eigenvalues of $T(p)$ and of $T_{j,n-j}(p^2)$
		on $M_\rho(\Gamma(N),\chi)$
		are algebraic integers.
	\end{cor}

	\begin{rem}When $n=2$, Collorary \ref{cor kn n+1} was given via Galois representations in \cite[Theorem I]{Weissauer}.
\end{rem}

	\begin{rem}
Hecke operators on $M_\rho(\Gamma_0(N))$ are defined in the same way as in the case of $M_\rho(\Gamma(N))$.
Then, the same assertion as in Theorem \ref{integrality for SMF} holds for $M_\rho(\Gamma_0(N))$.
		Indeed, by $M_\rho(\Gamma_0(N)) \subset M_\rho(\Gamma(N))$, the integrality of Hecke eigenvalues for $M_\rho(\Gamma_0(N))$ follows immediately from Theorem \ref{integrality for SMF}.
	\end{rem}

	\begin{rem}\label{remark on Shin Templier}
				
		We compare our method with the one by Shin and Templier \cite{ShinTemplier} as follows.
		Shin and Templier \cite[Proposition 4.1]{ShinTemplier} proved the integrality
		of Satake parameters of some irreducible cohomological conjugate-self-dual cuspidal automorphic representations
		of $\GL_n(\AA_F)$ for some number field $F$.
		They used compatible systems of Galois representations
		and Saito's integrality result \cite{Saito} of Frobenius actions arising in arithmetic geometry.
		It was generalized to some groups $G$ such as symplectic, quasi-split special orthogonal, and quasi-split unitary groups in the same paper \cite[Corollary 4.16]{ShinTemplier} by deducing such general cases to the $\GL_n$-case \cite[Proposition 4.1]{ShinTemplier}
		via endoscopic liftings to $\GL_n$.
		By virtue of the integrality, the local result \cite[Lemma 5.1]{ShinTemplier} was proved,
		and hence two integrality results
		\cite[Proposition 4.1]{ShinTemplier} and \cite[Lemma 5.1]{ShinTemplier} mentioned above lead us to the finiteness of representations with bounded degree of fields of rationality (\cite[Proposition 5.5]{ShinTemplier} and \cite[Corollary 5.7]{ShinTemplier}).
		Finally, \cite[Corollary 5.7]{ShinTemplier} was essentially used for proving
		their main result \cite[Theorem 6.1]{ShinTemplier}
		on the growth of fields of rationality.
		Furthermore, three results \cite[Corollary 6.4]{ShinTemplier}, \cite[Theorem 6.6]{ShinTemplier} and
		\cite[Proposition 6.10]{ShinTemplier} on fields of rationality depend on the local result \cite[Lemma 5.1]{ShinTemplier}.
		In Arthur's endoscopic classification, we need hypotheses such as
		the stabilization of twisted trace formulas\footnote{This stabilization was given by 
			Moeglin and Waldspurger in \cite{MoeglinWaldspurger1} and \cite{MoeglinWaldspurger2}.}
		and the weighted fundamental lemma
		{\rm(}cf.\ \cite[Hypothesis 4.8]{ShinTemplier} and \cite[footnote 10 in p.2031]{ShinTemplier}{\rm)}.
		Although Shin and Templier's result covers a larger class of algebraic groups,
		their technique depends on these hypotheses on endoscopic classifications and the global Langlands correspondence for $\GL_n$. Their technique seems so difficult and integrality was not much explicated in \cite[Proposition 4.1]{ShinTemplier}.
		On the other hand,
		our proof of integrality is elementary and independent of their technique.
		
	\end{rem}

		Let $f \in M_\rho(\Gamma(N),\chi)$ be a simultaneous eigenvector
	for all Hecke operators $T(p)$ and $T_{j,n-j}(p^2)$ ($1\le j\le n$) at all primes $p\notdivide N$. Let $\lambda(p)$ and $\lambda_j(p^2)$ for each $j\in \{1,\ldots,n\}$ be the eigenvalues of $f$ for $T(p)$ and $T_{j,n-j}(p^2)$, respectively.
	Then the Hecke field of $f$ is defined as
	$$\QQ(f):=\QQ(\{\lambda(p), \lambda_j(p^2) \mid 1\le j\le n,\ p\notdivide N\}).$$
	The Hecke field $\QQ(f)$ is expected to be a finite extension of $\QQ$ as it is in many cases such as $k_n\ge n+1$, $N=1$, etc.
In \cite[p.13]{KimWY}, they concluded $[\QQ(f):\QQ]<\infty$ by using the claim that Hecke operators preserve a $\QQ$-structure of $M_\rho(\Gamma(N))$.
However, there seems to be no literature on the claim.
We can complete the proof of the finiteness of $[\QQ(f):\QQ]$ by our results as follows.

	\begin{cor}\label{Hecke field is finite}
		Let $f \in M_\rho(\Gamma(N),\chi)$ be a simultaneous eigenvector
		for all Hecke operators $T(p)$, $T_{j,n-j}(p^2)$ $(1\le j\le n)$ at all primes $p\notdivide N$.
Then $\QQ(f)$ is a finite extension of $\QQ$.
	\end{cor}
	\begin{proof}
Let $K$ be the cyclotomic field obtained from $\QQ$ by adjoining all values of Dirichlet characters $\chi_j$ for all $j\in\{1,\ldots,n\}$.
Then $\Ocal_K$ satisfies the condition in Lemma \ref{existence of int model of chi}.
Then $M_\rho(\Gamma(N),\chi;K)$ is a $K$-structure of $M_\rho(\Gamma(N), \chi)$ by Lemma \ref{existence of int model of chi}, 
and $T(p)$ and $T_{j,n-j}(p^2)$ ($1\le j\le n$)
	preserve $M_\rho(\Gamma(N),\chi;K)$ by Theorem \ref{preserve SMF}.
	Let $H$ be the subspace of the $K$-vector space $\End_K(M_\rho(\Gamma(N),\chi;K))$ generated by $T(p)$ and $T_{j,n-j}(p^2)$ at all $j\in \{1,\ldots,n\}$ and all primes $p\notdivide N$.
	By $\dim_K M_\rho(\Gamma(N),\chi;K)<\infty$, $H$ has a $K$-basis $T_1,\ldots, T_d$ each of which is of the form $T(p)$ or $T_{j,n-j}(p^2)$ for some $p$ and $j$. Then
	$\Q(f)= K(S)$ holds, where $S$ is the set of all eigenvalues of $T_1,\ldots, T_d$. As any element of $S$ are algebraic over $K$, we are done.
	\end{proof}
If $k_n\ge n+1$ and $f$ is cuspidal, this corollary follows from the rationality of cuspidal cohomology (cf.\ \cite[Proposition 2.15]{ShinTemplier}).
	
	\section{Applications of integrality to the growth of Hecke fields}
	\label{Application of integrality to the growth of Hecke fields}
	In this section, we prove 
	Corollaries \ref{nonzero L and growth of Hecke} and \ref{nonzero deriv of L and growth of Hecke} as  applications of Theorem \ref{thm0.1} (Hilbert modular case), and also prove Corollaries \ref{Hecke fields for degree 2} and \ref{Hecke fields for degree n} as applications
	of Theorem \ref{integrality for SMF} (Siegel modular case).
	Corollary \ref{main nonzero L and growth of Hecke} follows from Corollary \ref{nonzero L and growth of Hecke} immediately.
	
	\subsection{Hilbert modular case}
	\label{Hilbert modular case}
	
	We basically follow the notation of \cite{SugiyamaTsuzuki}.
	Let $F$ be a totally real number field of finite degree,
	$\Ocal_F$ the integer ring of $F$,
	$\AA_F$ the adele ring of $F$, and $\AA_{F,\fin}$ the ring of finite adeles of $F$, respectively.
	Let $D_F$ be the absolute value of the discriminant of $F/\QQ$.
	For a place $v$ of $F$, we write $v | \infty$ if $v$ is archimedean.
	For a finite place $v$, $\Ocal_{F,v}$ denotes the integer ring of the completion $F_v$ of $F$ at $v$.
	For a non-zero ideal $\ga$ of $\Ocal_F$, let $S(\ga)$ be the set of all finite places of $F$ dividing $\ga$ and
	$\nr_{F/\QQ}(\ga)$ the absolute norm of $\ga$, respectively.
	
	For a non-zero ideal $\gn$ of $\Ocal_F$ of $F$ and a family $l=(l_v)_{v|\infty}$, of positive even integers,
	let $\Pi_{\rm cus}^{*}(l, \gn)$ be the set of all irreducible cuspidal automorphic representations of $\PGL_{2}(\AA_F)$ of weight $l$ and conductor $\gf_{\pi}=\gn$.
Set $\l_v(\pi) = \a_v+\a_v^{-1}$ for $\pi \in \Pi_{\rm cus}^{*}(l, \gn)$ for any finite place $v$ not dividing $\gn$, where $(\a_v, \a_v^{-1})$ is the Satake parameter of $\pi$ at $v$.
	
	The field of rationality of $\pi \in \Pi_{\rm cus}^{*}(l, \gn)$ is denoted by $\QQ(\pi)$.
This field is defined as follows. Let $\pi_\fin$ be the finite part of $\pi$.
Then $\QQ(\pi)$ is defined as the set of all elements of $\CC$ fixed by 
the group $\{\sigma \in \Aut(\CC) \mid \pi_\fin^\sigma \cong \pi_\fin\}$
(for the definition of $\pi_\fin^{\s}$, see \S\ref{Preliminaries on cuspidal cohomologies}).
We note that $\QQ(\pi)$ coincides with the Hecke field $\QQ(f)$ of the primitive Hilbert cusp form $f$ corresponding to $\pi$.
Indeed, $\nr_{F/\QQ}(\gp)^{(\max_{v|\infty}l_v-1)/2}\lambda_{\gp}(\pi) \in \CC$ for any non-zero prime ideal $\gp\notdivide \gn$ is equal to the Hecke eigenvalue of $f$ with respect to the Hecke operator $T'(\gp)$. Hence,
by the strong multiplicity one theorem for $\GL_2$, we have the expression
$$\QQ(\pi)
=\QQ(\{\nr_{F/\QQ}(\gp)^{1/2}\lambda_\gp(\pi) \mid \gp\notdivide \gn\})
=\QQ(\{\nr_{F/\QQ}(\gp)^{(\max_{v|\infty}l_v-1)/2}\lambda_\gp(\pi) \mid \gp\notdivide \gn\})=\QQ(f)$$
	 (see \cite[Theorem 3.10 and \S4.5.3]{Raghuram-Tanabe}).
	 Since all $l_v$'s are even, $\QQ(\pi)$ is a finite totally real number field (\cite[Proposition 2.8]{Shimura}).
If there exist distinct archimedean places $v_1$ and $v_2$ of $F$ such that $l_{v_1}$ is even and $l_{v_2}$ is odd, then $\pi \otimes |\cdot|_\AA^{\max_{v |\infty}l_v/2}$ is not C-algebraic (see \cite[Theorem 1.4 (2)]{Raghuram-Tanabe}). In that case, $\QQ(\pi)$ is a subfield of $\overline{\QQ}$ but not expected to be of finite degree over $\QQ$ by \cite[Conjecture 3.1.6]{BuzzardGee}.

	We write $\Sigma_\fin$ for the set of all finite places of $F$.
	Let $T$ be a finite subset of $\Sigma_\fin$ and $\eta$ a quadratic character of $F^\times \backslash \AA_{F}^{\times}$ such that the conductor $\gf_{\eta}$ of $\eta$ is coprime to $T$.
	The set $\Ical_{T,\eta}^{+}$ (resp.\ $\Ical_{T,\eta}^{-}$)
	denotes the set of all non-zero ideals of $\Ocal_F$ satisfying the following:
	\begin{itemize}
		\item $\gn$ is coprime to $T\cup S(\gf_{\eta})$,
		\item $\tilde{\eta}(\gp_v)=-1$ for all $v \in S(\gn)$,
		\item $\tilde{\eta}(\gn)\prod_{v |\infty}\eta_{v}(-1)=1$ \qquad (resp.\ $\tilde{\eta}(\gn)\prod_{v |\infty}\eta_{v}(-1)=-1$),
	\end{itemize}
	where $\gp_v$ is the prime ideal of $\Ocal_F$ corresponding to $v$ and $\tilde{\eta}$ is the quadratic character of the narrow ideal class group modulo $\gf_\eta$ corresponding to $\eta$. Here are our corollaries of Theorem \ref{thm0.1}.
	
	\begin{cor} \label{nonzero L and growth of Hecke}
		Let $l=(l_v)_{v | \infty}$ be a family of positive even integers such that $\min_{v|\infty}l_v \ge 6$, which is not necessarily parallel.
		Let $\eta$ be a quadratic character of 
		$F^\times \backslash \AA_{F}^{\times}$. Let $S$ be a finite subset of $\Sigma_\fin-S(\ff_\eta)$ and ${\bf J}=\{J_v\}_{v\in S}$ a family of closed subintervals of $(-2,2)$. Given a non-zero prime ideal $\fq$ coprime to $S\cup S(\ff_\eta)$,
		there exist constants $C_{\fq, l}>0$ and
		$N_{\gq, S, l, \eta, {\bf J}}>0$ 
		satisfying the following{\rm :}
		For any $\fn\in \Ical_{S\cup S(\fq),\eta}^{+}$ with $\nr_{F/\QQ}(\fn)>N_{\fq, S, l, \eta, {\bf J}}$,
		there exists $\pi \in \Pi_{\rm{cus}}^*(l,\fn)$ such that
		\begin{itemize}
			\item[(i)] $L(1/2,\pi)\not=0$ and $L(1/2,\pi \otimes \eta)\not=0$, 
			\item[(ii)] $[\Q(\pi):\Q]\geq C_{\fq,l} \,\sqrt{\log\log \nr_{F/\QQ}(\fn)}$,
			\item[(iii)] $\l_{v}(\pi) \in J_v$ for all $v\in S$.
		\end{itemize}
	\end{cor}

	\begin{cor} \label{nonzero deriv of L and growth of Hecke}
		Let $l=(l_v)_{v | \infty}$, $\eta$, $S$, $\bf J$ and $\gq$ be the same as in Corollary \ref{nonzero L and growth of Hecke}. Assume that
		for any non-zero ideal $\gn$ coprime to $\ff_\eta$,
		\begin{align}
			\tfrac{\d}{\d s}|_{s=1/2}(L(s,\pi)L(s,\pi\otimes \eta))\geq 0 \quad \text{for all $\pi \in \Pi_{\rm{cus}}^*(l,\gn)$.}
			\label{DLnonneg}
		\end{align}
		Then, for any $M>1$, there exist constants $C_{\fq, l}>0$ and 
		$N_{\fq, S, l, \eta, {\bf J}, M}>0$
		satisfying the following {\rm:}
		For any $\fn \in \Ical_{S\cup S(\fq),\eta}^{-}$ with
		$\nr_{F/\QQ}(\fn)>N_{\fq, S, l, \eta, {\bf J}, M}$
		and $\sum_{v\in S(\fn)}\frac{\log q_v}{q_v}\leq M$, there exists $\pi \in \Pi_{\rm{cus}}^*(l,\fn)$ such that
		\begin{itemize}
			\item[(i)] $\epsilon(1/2,\pi \otimes \eta)=-1$,
			\item[(ii)] $L(1/2,\pi)\not=0$ and $L'(1/2,\pi \otimes \eta)\not=0$,
			\item[(iii)] $[\Q(\pi):\Q]\geq C_{\fq,l}\,\sqrt{\log\log \nr_{F/\QQ}(\fn)}$ and 
			\item[(iv)] $\l_v(\pi)\in J_v$ for all $v\in S$.
		\end{itemize}
	\end{cor}
	\begin{proof}[Proof of Corollaries \ref{nonzero L and growth of Hecke} and \ref{nonzero deriv of L and growth of Hecke}]
		
		We sketch a proof of \cite[Theorem 1.3]{SugiyamaTsuzuki}, which is the parallel weight version of Corollary
		\ref{nonzero L and growth of Hecke}, for the convenience of the reader.
		The structure of the proof of \cite[Theorem 1.4]{SugiyamaTsuzuki},
		which is the parallel weight version of Corollary \ref{nonzero deriv of L and growth of Hecke}, is the same as \cite[Theorem 1.3]{SugiyamaTsuzuki}.
		For simplicity, we consider the case $S=\emptyset$.
		Under this restriction, one still sees that the idea of the application of the integrality of Hecke eigenvalues to the growth of Hecke fields.
		
Let $\mu=\mu_{\gq,\eta_\gq}$ be the probability measure on $[-2,2]$ 
		given by
			$$d\mu_{\gq,\eta_{\gq}}(x):=\frac{q_\gq-\eta_{\gq}(\varpi_\gq)}{(q_\gq^{1/2}+q_\gq^{-1/2}-x)(q_\gq^{1/2}+q_\gq^{-1/2}-\eta_\gq(\varpi_\gq)x)}\frac{\sqrt{4-x^2}}{2\pi}dx$$
		as in
		\cite[\S8]{SugiyamaTsuzuki}.
		Put
		$\Lcal(\pi) = L(1/2, \pi)L(1/2, \pi\otimes \eta)$
		for all $\pi \in \Pi_{\rm cus}^*(l,\gn)$.
		Let $I \subset [-2,2]$ be any open subinterval disjoint from the set
		$\{\lambda_\gq(\pi) \mid 
		\pi \in \Pi_{\rm cus}^*(l,\gn),\ \Lcal(\pi)\neq 0 \}.$
		By the relative trace formula for $\PGL_2$ given in \cite{SugiyamaTsuzukiRTF}, for any fixed $\e>0$ we have
		$$\mu(I)\le C \, \nr_{F/\QQ}(\gq)^{\epsilon} (\log \nr_{F/\QQ}(\gn))^{-1+\e}$$
		for any $\gn\in \Ical_{S(\gq),\eta}^{+}$ such that $\nr_{F/\QQ}(\gn)$
		is sufficiently large, where $C>0$ is a constant independent of $\gq, \gn$ and $I$ (see \cite[Lemmas 8.1 and 8.2]{SugiyamaTsuzuki}).
		 Next we fix any $\gn \in \Ical_{S(\gq),\eta}^+$ with a sufficiently large  $\nr_{F/\QQ}(\gn)$.
		Then we can take an open subinterval $I_0 \subset [-2,2]$ such that $\mu(I_0)=2C\, \nr_{F/\QQ}(\gq)^{\epsilon} (\log \nr_{F/\QQ}(\gn))^{-1+\e}$.
		By contraposition, there exists at least one $\pi \in \Pi_{\rm cus}^*(l,\gn)$ such that $\Lcal(\pi)\neq 0$
		and $\lambda_\gq(\pi) \in I_0$.
		By the argument in \cite[Lemme 5.3]{Royer} dividing $[-2,2]$ into a union of copies of $I_0$, we have
		\begin{align}\label{lower bound of number of pi}
\#\{\lambda_\gq(\pi) \mid
				\pi \in \Pi_{\rm cus}^*(l, \gn), 
			\Lcal(\pi)\neq 0\}
			\gg & 
			\,\frac{\mu([-2,2])}{\mu(I_0)} \\
			= & \frac{1}{\mu(I_0)}
			\gg \nr_{F/\QQ}(\gq)^{-\e} (\log \nr_{F/\QQ}(\gn))^{1-\e}.\notag
		\end{align}

		For $M>0$ and $d \in \NN$, let ${\mathcal E}(M,d)$ denote the set of $\l \in \overline{\ZZ}$
		such that
		the absolute values of all the conjugates of $\l$ are at most $M$
		and the absolute degree of $\l$ is at most $d$.
		Now let us set $d(\gn):=\max \{\deg \lambda_\gq(\pi) \mid 
		\pi \in \Pi_{\rm cus}^*(l,\gn),\ \Lcal(\pi)\neq 0 \}$ and assume that
		there exists $k\in \NN$ such that $l_v=k$ for all $v|\infty$.
		Since $l$ is parallel, the Hecke eigenvalue $\nr_{F/\QQ}(\gq)^{(k-1)/2}\lambda_\gq(\pi)$ for
		$\pi \in \Pi_{\rm cus}^*(l,\gn)$ is an algebraic integer by Shimura \cite[Proposition 2.2]{Shimura}.
		This yields $\nr_{F/\QQ}(\gq)^{(k-1)/2}\lambda_\gq(\pi) \in 
		\Ecal(2\nr_{F/\QQ}(\gq)^{(k-1)/2}, d(\gn))$ for all $\pi \in \Pi_{\rm cus}^*(l,\gn)$ such that $\Lcal(\pi)\neq0$.
		Hence we obtain
		\begin{align}\label{upper bound of number of pi}
\#\{ \nr_{F/\QQ}(\gq)^{(k-1)/2}\lambda_\gq(\pi) \mid\pi \in \Pi_{\rm cus}^*(l, \gn), \Lcal(\pi)\neq 0\}
			\le \#\Ecal(2\nr_{F/\QQ}(\gq)^{(k-1)/2}, d(\gn)).
		\end{align}
		By counting monic polynomials with integer coefficients with bounded roots, we obtain
		$\#\Ecal(2\nr_{F/\QQ}(\gq)^{(k-1)/2}, d(\gn)) \le (16\nr_{F/\QQ}(\gq)^{(k-1)/2})^{d(\gn)^2}$
		(see \cite[Lemme 6.2]{Royer}).
		By combining this inequality with \eqref{lower bound of number of pi} and
		\eqref{upper bound of number of pi}, we have the assertion (see \cite[Proposition 8.4]{SugiyamaTsuzuki}).

		In the argument above, the parallel weight condition of $l$ is used only for
		the inequality \eqref{upper bound of number of pi}
		(see \cite[Proposition 8.4]{SugiyamaTsuzuki}).
		In the proof of \cite[Proposition 8.4]{SugiyamaTsuzuki}, it is a key ingredient that ${\mathcal E}(2\nr_{F/\QQ}(\gq)^{(k-1)/2} ,d(\gn))$ contains all Hecke eigenvalues $\nr_{F/\QQ}(\gq)^{(k-1)/2}\l_\gq(\pi)$
		for all $\pi \in \Pi_{\rm cus}^{*}(l, \gn)$.
		In our setting, by Theorem \ref{thm0.1},
		${\mathcal E}(2\nr_{F/\QQ}(\gq)^{(k-1)/2} ,d(\gn))$ contains also all Hecke eigenvalues $\nr_{F/\QQ}(\gq)^{(k-1)/2}\l_\gq(\pi)$ with $k= \max_{v | \infty}l_v$ even when $l$ is non-parallel.
		Therefore we can eliminate the parallel weight condition from \cite[Proposition 8.4]{SugiyamaTsuzuki}
		and generalize \cite[Theorem 1.3]{SugiyamaTsuzuki} to the non-parallel weight case.
		This completes the proof of Corollary \ref{nonzero L and growth of Hecke}, and
		Corollary \ref{nonzero deriv of L and growth of Hecke} is proved in the same way.
	\end{proof}

	\subsection{Proof of Corollaries \ref{Hecke fields for degree 2} and \ref{Hecke fields for degree n} (Siegel modular case)}
	\label{Proof of Corollaries and (Siegel modular case)}
	As an application of the integrality of Hecke eigenvalues for Siegel modular forms (Theorem \ref{integrality for SMF}),
	we explain only the proof of Corollary \ref{Hecke fields for degree n}.
	Corollary \ref{Hecke fields for degree 2} may be proved similarly (see Remark \ref{proof GSp(4)}).
	The following method is the same as
	\cite[\S 6.2 and \S 6.3]{ShinTemplier} with the aid of \cite[\S5, \S6 and \S7]{Royer}. We describe the detail of our proof for the convenience of the reader.

	The space of $W_\rho(\CC)$-valued cusp forms on $\PGSp_{2n}(\AA_\QQ)$
	is isomorphic to that of $W_\rho(\CC)$-valued cusp forms on $\Sp_{2n}(\AA_\QQ)$
	as Hecke modules in the same way as the full level case ($N=1$) \cite[\S4.5.6]{ChenevierLannes},
	since we only consider the Hecke algebras at $p$ not dividing $N$.
	Hence we consider automorphic representations of $\Sp_{2n}(\AA_\QQ)$
	instead of $\GSp_{2n}(\AA_\QQ)$.
	
	Fix a prime number $p$.
	We define by $\Omega_p$
	the set of all isomorphism classes of
	irreducible spherical tempered representations of $\Sp_{2n}(\QQ_p)$,
	and by $\Omega_p^{\rm c}$
	the set of all isomorphism classes of irreducible spherical complementary series representations of $\Sp_{2n}(\QQ_p)$, respectively.
	Then $\Omega_p$
	is identified with
	\[\Omega_p =(\RR/2\pi\ZZ)^n/W\]
	and $\Omega_p^{\rm c}$ with a bounded set in $(\CC/2\pi\ZZ)^n/W$, respectively (cf.\ \cite{MuicTadic}).
	Here $W$ is the Weyl group of $\Sp_{2n}$, which acts on
	$(\CC/2\pi\ZZ)^n$ in the same manner as on the maximal split torus of $\Sp_{2n}$.
	Let $\mu_p=\mu_p^{\rm pl, temp}$ be the spherical Plancherel measure on $\Omega_p$.
	By abuse of notation, $\mu_p$ also denotes the measure on the spherical unitary dual $\Omega_p^+:=\Omega_p\cup \Omega_p^{\rm c}$ which extends $\mu_p$ on $\Omega_p$
	so that the support of the extended measure equals $\Omega_p$.

	Suppose $k_1\ge k_2\ge\cdots \ge k_n> n+1$.
Recall the basis ${\rm HE}_\rho(N)$ of $S_\rho(\Gamma(N))_{\bf1}$ explained in \S \ref{Growth of Hecke fields of Siegel modular forms of general degree}.
For each $f \in {\rm HE}_\rho(N)$, let $\pi_f=\otimes_{v}\pi_{f,v}$ is the irreducible cuspidal automorphic representation of $\Sp_{2n}(\AA_\QQ)$ generated by $f$, and let $(\theta_1(\pi_{f,p}),\ldots, \theta_n(\pi_{f,p})) \in [0,2\pi]^n$ for any prime $p\notdivide N$ be the family such that $(p^{\theta_1(\pi_{f,p})},\ldots, p^{\theta_n(\pi_{f,p})})$ is the Satake parameter of $\pi_{f,p}$.
	Set $S^1=\{z \in \C \mid |z|=1 \}$.
	By the quantitative equidistribution theorem (automorphic Plancherel density) for $\Sp_{2n}$ \cite[Theorem 1.1]{KimWY3},
	there exists $a>0$, $b>0$ and $c_0>0$ such that for any $\varphi$ on $\Omega_p^+$
	satisfying $\varphi(\theta_1\ldots, \theta_n)=P(z_1,\ldots, z_n)$ for a trigonometric polynomial function $P$ on $(S^1)^n$ of degree $\kappa \in \NN$, we have
	\[\frac{1}{\#{\rm HE}_\rho(N)}\sum_{f \in {\rm HE}_{\rho}(N)}\varphi(\theta_{1}(\pi_{f,p}),
	\ldots, \theta_{n}(\pi_{f,p})) = \int_{\Omega_p}\varphi(\theta_1,\cdots, \theta_n) d\mu_p
	+O(p^{a\kappa+b} N^{-n})\]
	as $N \rightarrow \infty$ under $N\ge c_0p^{2n\kappa}$ and $\gcd(N, p)=1$.
	We remark that we do not use the endoscopic classification here in contrast to \cite{KimWY3}.
	Hence the domain $\Omega_p$ of $f$ in \cite[Theorem 1.2]{KimWY3}
	is extended to $\Omega_p^{+}$ in our setting. See Remark \ref{not use endoscopic} for details.

	Let $I$ be any $n$-dimensional open rectangle in $\Omega_p$, i.e., $I$ is a product of $n$ open subintervals of $\RR/2\pi \ZZ$, and let ${\rm ch}_I$ be the characteristic function of $I$.
	Averaging the product of polynomials in \cite[Lemma 6.15]{ShinTemplier} by the Weyl group $W$,
	for any $\kappa\in \NN$ we can take $n$-variable trigonometric polynomials $Q_\kappa^+$ and $Q_\kappa^-$ of degree at most $\kappa$ so that $Q_\kappa^- \le {\rm ch}_I \le Q_\kappa^+$ and
	\[\int_{\Omega_p}(Q_\kappa^+ -Q_\kappa^-)d\mu_p=O\left(\frac{1}{\kappa^n}\right)\]
	(cf.\ \cite[(6.9)]{ShinTemplier}). In particular we have
	\[\mu_p^{\rm aut}(Q_\kappa^-)\le\mu_p^{\rm aut}({\rm ch}_{I})\le\mu_p^{\rm aut}(Q_\kappa^+),\]
	where we set $\mu_p^{\rm aut}(\varphi):= \frac{1}{\#{\rm HE}_\rho(N)}\sum_{f \in \HE_\rho(N)}\varphi((\theta_1(\pi_{f,p}),\ldots,\theta_n(\pi_{f,p})))$ for $\varphi \in C_{\rm c}(\Omega_p^+)$.
	With the aid of \cite[Lemma 6.16]{ShinTemplier}, the quantitative equidistribution theorem yields\footnote{By abuse of notation, we write $\mu_p(\phi)$ for $\int_{\Omega_p} \phi d\mu_p$ for any $\mu_p$-integrable function $\phi$ on $\Omega_p$.}
	\begin{align*}|\mu_p^{\rm aut}({\rm ch}_{I})-\mu_p(I)|\le & 
		\max(\mu_p^{\rm aut}(Q_\kappa^+)-\mu_p(Q_\kappa^-),
		|\mu_p^{\rm aut}(Q_\kappa^-)-\mu_p(Q_\kappa^+)|) \\
		\le & \max_{\delta\in \{+,-\}}(|\mu_p^{\rm aut}(Q_\kappa^{\delta})-\mu_p(Q_\kappa^{\delta})|+|\mu_p(Q_\kappa^{+}-Q_\kappa^-)|,) \\
		=\, & O(p^{a\kappa+b}N^{-n})+O(\kappa^{-n}).
	\end{align*}
	This leads us to
	\[\frac{\#\{ (\theta_1(\pi_{f,p}), \ldots, \theta_n(\pi_{f,p})) \in I \mid f \in {\rm HE}_\rho(N) \}}{\#{\rm HE}_\rho(N)} = \mu_p(I)+O(p^{a\kappa+b}N^{-n})+O(\kappa^{-n}),\]
	where the implied constant is independent of $I$.
	For any fixed $\e>0$ such that $\e<n/a$ and $\e<\frac{1}{2n}$, and for any $N \in \NN$ such that $N^{1-2n\e}>c_0$, set $\kappa=\lfloor\e\log N\rfloor \ge0$.
	Then $\kappa \in \NN$ and $N> c_0p^{2n\kappa}$ hold.
	Thus we have
	\[\frac{\#\{ (\theta_1(\pi_{f,p}), \ldots, \theta_n(\pi_{f,p})) \in I \mid f \in {\rm HE}_\rho(N) \}}{\#{\rm HE}_\rho(N)} = \mu_p(I)+O((\log N)^{-n}).\]
	In particular, if $I$ satisfies
	$I\cap\{(\theta_1(\pi_{f,p}), \ldots, \theta_n(\pi_{f,p}))\mid f \in {\rm HE}_\rho(N)\}=\emptyset$, then the left-hand side equals $0$ and hence $\mu_p(I) \le C (\log_p N)^{-n}$ holds for a constant $C>0$ independent of $I$ and $N$.
	Let $I_0$ be an open $n$-dimensional rectangle in $\Omega_p$ such that $\mu_p(I_0)=2C(\log_p N)^{-n}$, where we add the condition $N> p^{(4C)^{1/n}}$, which implies $2C(\log_p N)^{-n}<1/2$.
	By contraposition, $I_0$ must include at least one $(\theta_1(\pi_{f,p}), \ldots, \theta_n(\pi_{f,p}))$
	for some $f \in {\rm HE}_\rho(N)$.
	Hence, by the same argument as \cite[Lemme 5.3]{Royer}, there exists $\delta \in (0,1/2)$ satisfying $0<\delta<\mu_p(I_0)$ and
		\begin{align}\label{first step ineq}\#\{ (\theta_1(\pi_{f,p}), \ldots, \theta_n(\pi_{f,p})) \in \Omega_p \mid f \in {\rm HE}_\rho(N) \}\ge \frac{\mu_p(\Omega_p)-\delta}{\mu_p(I_0)} > \frac{(\log_p N)^n}{4C},\end{align}
	where we use $\mu_p(\Omega_p)=1$.

Next we construct an injective mapping $\lambda : \Omega_p^{+} \rightarrow \CC^{n}$.
For any $\theta \in \Omega_p^{+}$, let $\pi_\theta$ denote a unique unramified irreducible unitary representation of $\Sp_{2n}(\QQ_p)$ corresponding to $\theta$.
The Hecke operators $T(p)$ and $\{T_{j,n-j}(p^2)\}_{1\le j\le n-1}$
in $\End(S_\rho(\Gamma(N))_{\bf1})$ correspond to
generators, say $T$ and $\{T_j\}_{1\le j\le n-1}$, of the spherical Hecke algebra of $\Sp_{2n}(\QQ_p)$.\footnote{The action of $T_{n,0}(p^2)$ on $S_\rho(\Gamma(N))_{\bf1}$ is $p^{\sum_{j=1}^{n}k_j-n(n+1)} \id$. Hence the operator is negligible.}
Since any unramified irreducible representation of $\Sp_{2n}(\QQ_p)$ is determined by the actions of $T$ and $\{T_j\}_{1 \le j \le n-1}$,
we can injectively map $\theta \in \Omega_p^+$ to the system $\lambda(\theta)=(\lambda_0,\ldots, \lambda_{n-1}) \in \CC^{n}$ such that $\pi_\theta(T)=\lambda_0\id$ and $\pi_\theta (T_{j}) =\lambda_j \id$ for all $j\in\{1,\ldots,n-1\}$.
Therefore, we obtain an injective mapping $\Omega_p^{+} \ni \theta\mapsto \lambda(\theta) \in \CC^{n}$. In particular, we have
\begin{align}\label{lambda injective}&\#\{ (\theta_1(\pi_{f,p}), \ldots, \theta_n(\pi_{f,p})) \in \Omega_p^{+} \mid f \in {\rm HE}_\rho(N) \} \\
= & \#\{ \lambda(\theta_1(\pi_{f,p}), \ldots, \theta_n(\pi_{f,p})) \in \CC^{n} \mid f \in {\rm HE}_\rho(N) \}.\notag \end{align}

Then,
	by combining the boundedness of $\Omega_p^+$ with \cite[(10)]{AsgariSchmidt},
Theorem \ref{integrality for SMF} on the integrality of Hecke eigenvalues
shows that
	every $\lambda(\theta_1(\pi_{f,p}),\ldots, \theta_n(\pi_{f,p}))$ is contained in
	$$R:=\prod_{j=1}^{n}\{\alpha \in \overline{\ZZ} \mid |\s(\alpha)|\le p^A \ (\forall \s \in \Gal(\overline{\QQ}/\QQ)) \},$$
	where $A>0$ is a constant depending only on $n$ and $\rho$ such that the inequality $|\s(\a)|\le p^A$ holds for all eigenvalues $\a$
	of $T(p)$ and of $T_{j,n-j}(p^2)$
	for all $j\in \{1,\ldots, n-1\}$ and all $\s \in \Gal(\overline{\QQ}/\QQ)$.	
		Set $d=\max_{f \in {\rm HE}_\rho(N)}[\QQ(f):\QQ]$.
We can apply \cite[Lemme 6.1 and Lemme 6.2]{Royer} for $M=p^A$ to prove
	the estimate \begin{align}\label{second step ineq}
\#\{ \lambda(\theta_1(\pi_{f,p}), \ldots, \theta_n(\pi_{f,p})) \in \CC^{n} \mid f \in {\rm HE}_\rho(N) \}
\le \#R		\le \{(8p^A)^{d^2}\}^{n}.\end{align}
	Here we do not need the generalized Ramanujan-Petersson conjecture.
	For example, $A=k/2+1$ is enough if we consider elliptic cusp forms of weight $k$ (i.e., $n=1$).

	Combining two inequalities \eqref{first step ineq} and \eqref{second step ineq} with \eqref{lambda injective},
	we conclude
	Corollary \ref{Hecke fields for degree n}.
	\qed

	\begin{rem}\label{proof GSp(4)}
		Corollary \ref{Hecke fields for degree 2} may be proved similarly to Corollary \ref{Hecke fields for degree n}.
However, Corollary  \ref{Hecke fields for degree 2} treats the cases (1) $k_2=n+1=3$, and
(2) the central character $\chi_j$ of automorphic representations of $\GSp_4(\AA_\QQ)$ may not be trivial.
Thus we need modifications of the proof of Corollary \ref{Hecke fields for degree n} to apply to the case of $n=2$.

As for the case (1), we can relax the weight condition $k_1\ge k_2 > 3$ to 
$k_1\ge k_2 \ge 3$ since the equidistribution of Satake parameters for $\GSp_4$ is given by fine calculations of the invariant trace formula for $\GSp_4$ even when $k_2=3$ as in \cite[Theorem 1.1]{KimWY}.

As for the case (2), the equality \eqref{lambda injective} should be modified as follows.
In the proof of Corollary \ref{Hecke fields for degree n} for $n=2$,
we have only to replace $S_\rho(\Gamma(N))_{\bf1}$ with $S_\rho(\Gamma(N))_{\chi}$ and
${\rm HE}_\rho(N)$ with ${\rm HE}_\rho(N,\chi)$, respectively.
When $p$ is a prime number such that $p\notdivide N$ and $\omega_p$ is the $p$-component of the Hecke character $\otimes_{v}\omega_v$ of $\QQ^\times\bsl \AA_\QQ^\times$ corresponding to $\chi$,
the set $\Omega_p^{+}$ for $n=2$ is identified with the set of isomorphism classes of unramified
irreducible unitary representations of $\GSp_4(\QQ_p)$ with central character $\omega_p^{-1}$ since the action of $T_{2,0}(p^2)$ is a scalar multiple by
$p^{k_1+k_2-6} \chi(p)=p^{k_1+k_2-6} \omega_p(p)^{-1}$.
Then we obtain an injective mapping $\lambda:\Omega_p^{+} \rightarrow \CC^2$ similarly.
As a consequence, the proof of Corollary \ref{Hecke fields for degree n} for $n=2$ carries through after the equality \eqref{lambda injective} is replaced with
\begin{align*}&\#\{ (\theta_1(\pi_{f,p}), \theta_2(\pi_{f,p})) \in \Omega_p^{+} \mid f \in {\rm HE}_\rho(N,\chi) \}\\
	= & \#\{ \lambda(\theta_1(\pi_{f,p}), \theta_2(\pi_{f,p})) \in \CC^{2} \mid f \in {\rm HE}_\rho(N,\chi) \}.\notag \end{align*}
Moreover, the Hecke basis ${\rm HE}_\rho(N,\chi)$
is restricted to the set ${\rm HE}_{\rho}^{\rm tm}(N,\chi)$ of non-endoscopic lifts, since the contribution of the endoscopic lifts called Yoshida lifts is negligible as in \cite[\S7.3]{KimWY}
and the equidistribution result for ${\rm HE}_{\rho}^{\rm tm}(N,\chi)$ is given as in \cite[Theorem 1.3]{KimWY}.
\end{rem}

	\begin{rem}\label{not use endoscopic}
		The endoscopic classification used in \cite{KimWY3} is not needed
		to state their equidistribution result \cite[Theorem 1.2]{KimWY3}
		since their evaluation of the invariant trace formula requires only pseudo coefficients and analysis via Shintani zeta functions
		associated to prehomogeneous vector spaces.
		Since we do not use the endoscopic classification,
		we can relax the assumption ``$k_1>\cdots > k_n>n+1$'' of \cite[(1-4)]{KimWY3} to ``$k_1\ge\cdots \ge k_n>n+1$''.
		If we would permit the use of endoscopic classification,
		then their equidistribution result \cite[Theorem 1.2]{KimWY3} would be stated for genuine forms as in \cite[Theorem 1.3]{KimWY3},
		and thus Corollary \ref{Hecke fields for degree n} would be valid
		even when the set $\HE_\rho(N)$ is restricted to the subset $\HE_\rho(N)^{\rm g}$ consisting of genuine forms in $\HE_\rho(N)$.
		Since we do not use the endoscopic classification, the domain of test functions $f$ in our quantitative equidistribution theorem cannot be restricted to $\Omega_p$ in contrast to \cite[Theorem 1.2]{KimWY3}.
	\end{rem}	
	
	\section{Estimations of cuspidal cohomologies}
	\label{Cuspidal cohomology}
	
	In this section, we prove Corollaries \ref{cor Hecke fields GL2d}
	and \ref{growth for BC}.
	We prove Corollary \ref{lower bound for Hecke field} below, from which Corollary
	\ref{cor Hecke fields GL2d} follows immediately.
	For Corollary \ref{lower bound for Hecke field}, we use automorphic inductions from 
	$\GL_{2,F}$ to $\GL_{2d, \QQ}$ and results in \S\ref{Application of integrality to the growth of Hecke fields} for Hilbert modular forms.
	Our problems are:
	\begin{enumerate}
		
		\item the automorphic induction $\Pi=\AI_{F/\QQ}(\pi)$ of a Hilbert cusp form $\pi$ of ${\GL_2}(\AA_F)$ may not be cuspidal,
		
		\item the field of rationality $\QQ(\Pi)$ may be much smaller than $\QQ(\pi)$,
		
		\item the conductor of $\Pi$ may be quite different from that of $\pi$.
	\end{enumerate}
	As for (1), we have the non-parallel weight condition for proving the cuspidality (\S\ref{Automorphic inductions}).
	As for (2), we use the local Langlands correspondence for $\GL_n$ to justify that $\QQ(\Pi)$ is not so smaller than $\QQ(\pi)$ (Lemma \ref{field of AI}).
	As for (3), we combine the equality of the epsilon factors $\e(s,\pi) = \e(s,\Pi)$ with explicit formulas in terms of the conductors
	(Lemma \ref{field of AI}).
	
	For Corollary \ref{growth for BC}, we solve the same problems as (2) and (3) above (Lemma \ref{field of BC}). However, for cuspidality, we do not need the non-parallel weight condition. Instead, we use the condition that $E/F$ is a ramified extension (\S\ref{Quadratic base change liftings}).
	
	For a number field $F$ and an irreducible cuspidal automorphic representation $\pi$ of $\GL_n(\AA_F)$,
	$L(s,\pi)$ denotes the completed standard $L$-function attached to $\pi$ whose functional equation is
	$L(s, \pi) = \e(s, \pi)L(1-s, \pi^{\vee})$.
	For such a $\pi$, let $\gf_\pi$ denote the conductor of $\pi$,
	which is the maximum of non-zero ideals $\gn$ of $\Ocal_F$ such that the $\bfK_1(\gn)$-invariant subspace of $\pi$ is non-zero.
	Here $\bfK_1(\gn)$ is the open compact subgroup of $\GL_n(\AA_{F,\fin})$ defined as the product of
	$\bfK_{1,v}(\gn) \subset \GL_n(\Ocal_{F,v})$ over all finite places $v$ and $\bfK_{1,v}(\gn)$ is the subgroup consisting of
	all elements of $\GL_n(\Ocal_{F,v})$ whose last rows are congruent to $(0,0,\ldots,0,1)$
	modulo $\gn\Ocal_{F,v}$.
	The field of rationality $\QQ(\pi)$ of $\pi$ is defined to be the subfield of $\CC$ consisting of all complex numbers fixed by 
	the group $\{\sigma \in \Aut(\CC) \mid \pi_\fin^\sigma \cong \pi_\fin\}$ (see  \cite[\S3.1]{ClozelMotif} or \cite[Definition 2.2]{ShinTemplier}). It is a finite extension of $\QQ$ when $\pi$ is cohomological (see \cite[\S3]{ClozelMotif} or \cite[Proposition 2.15]{ShinTemplier}).

	\subsection{Preliminaries on cuspidal cohomologies}
	\label{Preliminaries on cuspidal cohomologies}
	Let us review cuspidal cohomology for $\GL_n$ defined over $\mathbb{Q}$. One can refer to \cite{GanRaghuram} or \cite{Raghuram}
	for cuspidal cohomologies.
	
	Let $T_n$ be the maximal $\QQ$-split torus of $\GL_n$ consisting of diagonal matrices.
	Then, the absolute Weyl group
	$W_n$
	for $(\GL_n,T_n)$
	acts on
	$X(T_n) = \Hom(T, {\GG}_m) \cong \ZZ^n$.
	By highest weight theory, we have a bijection
	$$X(T_n)/W_n \ni \l \mapsto V_\l\in {\rm Irr}_{\rm alg}(\GL_n(\CC)),$$
	where
	${\rm Irr}_{\rm alg}(\GL_n(\CC))$ is the set of the equivalence classes of irreducible algebraic representations of $\GL_n(\CC)$ and
	the symbol $V_\l$ is an irreducible algebraic representation of $\GL_n(\CC)$ of highest weight $\l$.
	Let $X^{+}(T_n)$ be the set of dominant weights, which is identified with the set $X(T_n)/W_n$.
	Let $\rho_n$ denote half the sum of positive roots for the Borel subgroup consisting of upper triangular matrices, i.e.,
	$$\rho_n=\frac{1}{2}\sum_{i<j}(\e_i-\e_{j})=\left(\frac{n-1}{2}, \frac{n-3}{2},\ldots, -\frac{n-1}{2}\right) \in \frac{1}{2}\ZZ^{n},$$
	where $\e_j \in X(T_n)$ is defined by $\e_j(\diag(t_1,\ldots,t_n))=t_j$ for all $\diag(t_1,\ldots,t_n)\in T_n$.
	
	Set
	$\gA = \{t1_n \mid t \in \RR_{>0} \}$ as a subgroup of $\GL_n(\AA_\QQ)$ via the embedding $\RR_{>0}\ni t \mapsto (t,1)\in \RR^{\times}\times \AA_{\QQ, \fin}^\times=\AA_\QQ^\times$.
	Put $K_\infty^+=\gA \,\SO(n)$.
	For any open compact subgroup $\Ucal$ of $\GL_n(\AA_{\QQ,\fin})$, the quotient
	$$S_{\Ucal}^{\GL_n}= \GL_n(\QQ)\bsl \GL_n(\AA_\QQ)/K_\infty^+\,\Ucal$$
	is a locally symmetric space.
	For a fixed $\l \in X^+(T_n)$,
	let $H_{\rm cusp}^{\bullet}(S_\Ucal^{\GL_n}, \tilde V_\l)$ denote
	the cuspidal cohomology with coefficients in the local system $\tilde V_\l$ on $S_{\Ucal}^{\GL_n}$ determined by $V_\lambda$.
	For an irreducible cuspidal automorphic representation $\Pi$ of $\GL_n(\AA_\QQ)$, we call
	$\Pi$ $\lambda$-cohomological if $\Pi$ has a non-zero contribution in $H_{\rm cusp}^{\bullet}(S_\Ucal^{\GL_n}, \tilde V_\l)$
	for some $\Ucal$. 
Also, $\Pi$ is called cohomological if there exists $\lambda \in X^+(T_n)$ such that $\Pi$ is $\lambda$-cohomological.
	Cuspidal cohomologies for $\Res_{E/\QQ}\GL_{n, E}$ with a number field $E$ are defined similarly by replacing $\lambda$ with a family $(\lambda_\iota)_{\iota :E \hookrightarrow \CC}$ of dominant integral weights indexed by the embeddings of $E$ into $\CC$.
	
We recall the notion of C-algebraicity and regularity of automorphic representations (cf.\ \cite[Definition 1.8 and Definition 3.12]{ClozelMotif}, \cite[Definition 2.1]{ShinTemplier}).
	Let $F$ be a number field and let $\Pi=(\otimes_{v | \infty}\Pi_v)\otimes \Pi_\fin$ be an irreducible cuspidal automorphic representation of $\GL_n(\AA_F)$.
For each $v|\infty$, there exists an $L$-parameter\footnote{We do not take $\GL_n(\CC) \rtimes W_{F_v}$ as the codomain of
$\phi_{\Pi_v}$ since $\GL_{2, F_v}$ is split over $F_v$.} $\phi_{\Pi_v}: W_{F_v}\rightarrow \GL_{n}(\CC)$ corresponding to $\Pi_v$ via the local Langlands correspondence for $\GL_n(\RR)$ and $\GL_n(\CC)$,
where $W_{F_v}$ is the Weil group of $F_v$.
We consider the restriction $\phi_{\Pi_v}|_{\CC^\times}$ of
$\phi_{\Pi_v}$ to $\CC^\times (\subset W_{F_v})$.
It is regarded as a cocharacter of $T_{n}(\CC)$ and
the cocharacter $\phi_{\Pi_v}|_{\CC^\times}$
is of the form
$$\CC^\times \ni z\mapsto \diag(z^{p_1}\overline{z}^{q_1},\ldots,z^{p_n}\overline{z}^{q_n}) \in T_n(\CC)$$
for some $(p_j)_j, (q_j)_j \in \CC^n$ such that $(p_j-q_j)_j\in \ZZ^n$.

For the case of $F_v \cong \RR$, $\Pi_v$ is called C-algebraic if
$(p_j)_j -(\frac{n-1}{2}, \frac{n-3}{2},\ldots, -\frac{n-1}{2}) \in \ZZ^n$.
And $\Pi_v$ is called regular if $p_j \neq p_{j'}$ hold for all $j\neq j'$.

For the case of $F_v\cong \CC$, $\Pi_v$ is called C-algebraic if
both $(p_j)_j -(\frac{n-1}{2}, \frac{n-3}{2},\ldots, -\frac{n-1}{2}) \in \ZZ^n$
and $(q_j)_j -(\frac{n-1}{2}, \frac{n-3}{2},\ldots, -\frac{n-1}{2}) \in \ZZ^n$ are satisfied.
And $\Pi_v$ is called regular if both $p_{j}\neq p_{j'}$ and $q_{j}\neq q_{j'}$ hold 
for all $j\neq j'$.

We go back to the global situation. The automorphic representation $\Pi$ is called C-algebraic (resp.\ regular) if $\Pi_v$ is C-algebraic (resp.\ regular) for every $v|\infty$.
We remark that
$\pi$ is cohomological if and only if $\pi$ is regular and C-algebraic
(\cite[Lemme 3.14]{ClozelMotif} and \cite[Lemma 2.12 and Remark 2.13]{ShinTemplier}).

	For $\s\in\Aut(\CC)$ and a representation $(\pi, V)$ of a group $G$,
	let $\pi^\s$ denote the representation of $G$ such that
	$\pi^\s(g)(v \otimes a) = \pi(g)v\otimes a$ for $v \otimes a \in V\otimes_{\CC,\s^{-1}} \CC$.
	We remark that the scalar multiple on $V\otimes_{\CC, \s^{-1}}\CC$ is given by $b(v\otimes a)=v\otimes ab$ for any $v\in V$, $a\in \CC$ and any $b \in \CC$.\footnote{The field $\CC$ is a left $\CC$-algebra by $\sigma^{-1}:\CC\rightarrow \CC$. We write $\CC_{\s^{-1}}$ for this left $\CC$-algebra.
		The extension of scalars $V\otimes_{\CC, \s^{-1}}\CC:= (\s^{-1})_{!}V = V\otimes_{\CC}\CC_{\s^{-1}} $ is a right $\CC$-algebra by the usual product of $\CC_{\s^{-1}}=\CC$. Then we have $v\otimes ab=(\s(b)v)\otimes a$ for all $v\in V$, $a \in \CC$ and $b\in\CC$.}

	Let $\Pi=\Pi_\infty\otimes\Pi_\fin$ be a $\l$-cohomological irreducible cuspidal automorphic representation of $\GL_n(\AA_\QQ)$ such that $\Pi^{\Ucal} \neq 0$
	for some open compact subgroup $\Ucal$ of $\GL_n(\AA_{\QQ,\fin})$.
	Since there is the only one archimedean place of $\QQ$, the representation $\Pi^\s:=\Pi_\infty\otimes\Pi_{\fin}^{\s}$ is also $\l$-cohomological
	for all $\s \in \Aut(\CC)$ (\cite[Th\'eor\`eme 3.13]{ClozelMotif}).
	For any $q \in \ZZ$ such that $\lfloor n^2/4\rfloor \le q\le \lfloor n^2/4\rfloor +\lfloor (n-1)/2\rfloor$
	(cf.\ \cite[Proposition 2.15]{Raghuram}),
	the dimension of cuspidal cohomology is estimated as
\begin{align}\label{dim cusp coho}
		\dim_\CC H_{\rm cusp}^{q}(S_\Ucal^{\GL_n}, \tilde V_\l) = & \sum_{\pi \ : \ {\rm cuspidal}}\dim_{\CC}\{H^q(\mathfrak{gl}_n(\RR)/ \Lie(K_\infty^+) ;\pi_\infty\otimes V_\l)\otimes \pi_{\fin}^{\Ucal}\} \\
		\ge & \sum_{\s \in \Aut(\CC/\QQ(\Pi))\bsl{\rm Aut}(\CC)}\dim_{\CC} \{H^q(\mathfrak{gl}_n(\RR)/ \Lie(K_\infty^+); \Pi_\infty \otimes V_{\l}) \otimes (\Pi_{\fin}^{\s})^{\Ucal}\} \notag\\
		\ge & \#(\Aut(\CC/\QQ(\Pi))\bsl \Aut(\CC))=[\QQ(\Pi) : \QQ].\notag
	\end{align}

	\subsection{Automorphic inductions}
	\label{Automorphic inductions}
	
	Let $d$ be a prime number.
	\begin{prop}\label{existence of totally real cyclic}
		There exists a totally real cyclic extension $F/\QQ$ of degree $d$.
	\end{prop}
	\begin{proof}
		Let $q\ge 3$ be a prime number such that $q\equiv 1 \pmod{2d}$.
		The existence of $q$ is guaranteed by Dirichlet's theorem on arithmetic progression.
		Let $\zeta_q$ denote a primitive $q$th root of unity.
		Then, $\QQ(\zeta_q+\zeta_{q}^{-1})/\QQ$ is a totally real cyclic extension. Since $d$ divides $[\QQ(\zeta_q+\zeta_q^{-1}):\QQ]=(q-1)/2$,
		there exists a cyclic subgroup $H$ of $\Gal(\QQ(\zeta_q+\zeta_q^{-1})/\QQ)$ of index $d$.
		Thus it suffices to take the field $F$ corresponding to $H$ via Galois theory.
	\end{proof}

	Let $F/\QQ$ be a totally real cyclic extension with prime degree $d$.
	Fix a generator $\tau$ of $\Gal(F/\QQ)$.
	For any family $l=(l_v)_{v|\infty}$ of positive even integers and any non-zero ideal of $\Ocal_F$, take $\pi\in \Pi_{\rm cus}^*(l,\gn)$ and write $\Pi= \AI_{F/\QQ}(\pi)$ for
	the automorphic induction of $\pi$ (\cite{ArthurClozel}).
	Then
	$\Pi$ is an irreducible automorphic
	representation of $\GL_{2d}(\AA_\QQ)$ such that $L(s,\Pi)=L(s,\pi)$.

	Fix an enumeration of the set of all archimedean places of $F$
	as $\{1,\ldots, d\}$.
	Assume that $l=(l_v)_{v|\infty} =(l_j)_{j=1,\ldots, d} \in 2\NN^{d}$ is regular, i.e., $l_j\neq l_{j'}$ holds for all $j\neq j'$.
	For each $v | \infty$,
	$\pi_v$ is the discrete
	series representation of
	$\PGL_2(F_v)\cong \PGL_2(\RR)$ with minimal ${\rm O}(2)$-type $l_v$.
	Hence
	the $L$-parameter $\phi_{\pi_v} : W_{F_v}=W_\RR \rightarrow \GL_2(\CC)$ attached to $\pi_v$ is given as
	\begin{align}\label{archi L-para}\phi_{\pi_v}=\Ind_{W_\CC}^{W_\RR}\bigg(\frac{z}{\overline{z}}\bigg)^{(l_v-1)/2},
	\end{align}
	where $(\frac{z}{\overline{z}})^{1/2}$ for $z \in W_\CC= \CC^\times$ is defined as $z\sqrt{z\overline{z}}^{-1}$
	and a square root $\sqrt{z\overline{z}}$ is taken to be positive (see \cite[\S3.1.4]{Raghuram-Tanabe}). Thus
	$\pi_v$ is regular.
	Furthermore we have $\pi \circ \tau^j \not \cong \pi$ for all $j\in \{1,\ldots,d-1\}$ by the assumption on the regularity of $l$,
	and hence $\Pi=\AI_{F/\QQ}(\pi)$ is cuspidal (cf.\ \cite[Theorems 4.2 (e) and 5.1]{ArthurClozel}).
	
	As $\pi$ is a $\l$-cohomological cuspidal automorphic representation of $\GL_2(\AA_F)$
	with $\l=((\frac{l_v-1}{2}, -\frac{l_v-1}{2}) -\rho_2)_{v|\infty} \in \prod_{v|\infty}\ZZ^2$ by \cite[Theorem 1.4 (2), (3)]{Raghuram-Tanabe},
	we have the following.
	
	\begin{lem}\label{AI is coho}
		The $L$-parameter $\phi_{\Pi_\infty} : W_\RR \rightarrow \GL_{2d}(\CC)$ corresponding to the archimedean component $\Pi_\infty$ of $\Pi=\AI_{F/\QQ}(\pi)$ satisfies
		\begin{align}\label{L-para GL2d}\phi_{\Pi_\infty}|_{W_\CC} = \bigoplus_{j=1}^{d}\left[\begin{smallmatrix} (\frac{z}{\bar z})^{\frac{l_j-1}{2}} & 0\\ 0 & (\frac{z}{\bar z})^{-\frac{l_j-1}{2}}
			\end{smallmatrix}\right].\end{align}

		If a family $l=(l_j)_{j=1}^{d}$ of even integers is taken such that
		\begin{align}\label{l is regular}l_1> l_2>\cdots > l_d\ge 4,
		\end{align}
		then
		$\Pi$ is $\lambda'$-cohomological with
		$\l'= (\l_1',\ldots, \l_d', -\l_d', \ldots, -\l_1')$, where
		$\l_j'=\frac{l_j}{2}-d+j-1$ and $\l_j'\ge \l_{j+1}'$ for all $j\in \{1,\ldots, d-1\}$. In particular $\l'\neq 0$ holds.
	\end{lem}
	\begin{proof}Since $\Pi_\infty$ is a component of $\AI_{F/\QQ}(\pi)$,
		we have
		$\phi_{\Pi_\infty}=\bigoplus_{v|\infty}\Ind_{W_{F_v}}^{W_\RR}\phi_{\pi_v}=\bigoplus_{v|\infty}\phi_{\pi_v}$.					
		By	the equation $\phi_{\pi_v}|_{W_\CC}=(\frac{z}{\overline{z}})^{(l_v-1)/2}\oplus (\frac{z}{\overline{z}})^{-(l_v-1)/2}$ (\cite[\S3.1.5]{Raghuram-Tanabe}),
		we obtain \eqref{L-para GL2d}.
		
		Assume the condition \eqref{l is regular}. 
		By \eqref{L-para GL2d}, $\Pi_\infty$ is C-algebraic and regular. Hence  $\Pi$ is $\lambda'$-cohomological for some dominant $\lambda'=(\lambda_1'\ldots,\lambda_{2d}')$.
		Since the infinitesimal character of $V_{\lambda'}^{\vee}$ is the same as that of $\Pi_\infty$, we have the equality
		$$(-\lambda_{2d}',\ldots,-\lambda_1')+\rho_{2d}=\bigg(\frac{l_1-1}{2},\ldots, \frac{l_d-1}{2}, -\frac{l_d-1}{2},\ldots, -\frac{l_1-1}{2}\bigg).$$
		Thus $\lambda'=(\lambda_j')_{j=1}^{2d}$ is explicitly given as in the assertion. Furthermore, an easy calculation shows $\lambda_j'\ge \lambda_{j+1}'$ for all $j$ by virtue of $l_{j}\ge l_{j+1}+2$.
		If we would have $\l'=0$, then we show $l_j=2(d-j+1)$ for each $j$, and hence
		$l_d=2$. This contradicts $l_d\ge 4$.
	\end{proof}

	Next we consider relations between automorphic inductions and the action of $\Aut(\CC)$. 
	In what follows, Weil-Deligne representations are considered
	on $\CC$-vector spaces.
	
	Let $L/K$ be a field extension of non-archimedean local fields with degree $[L:K]=d$.
	Let  ${\rm rec}_K$ be the local Langlands bijection for $\GL_n(K)$.
	This means that, for any irreducible admissible representation $\pi$ of $\GL_n(K)$, 
	${\rm rec}_K(\pi)$ is an $n$-dimensional Frobenius semisimple Weil-Deligne representation of the Weil group $W_K$ of $K$ which corresponds to $\pi$ via the local Langlands correspondence.
	We note that
	${\rm rec}_K$ does not always satisfy ${\rm rec}_K(\pi)^\s={\rm rec}_K(\pi^\s)$ for any irreducible admissible representation
	$\pi$ of $\GL_n(K)$ and any $\s \in \Aut(\CC)$ in general.
	Indeed, we actually have
	\begin{align}\label{twisted Langlands}(|\cdot|_{W_K}^{-(n-1)/2}\otimes \rec_{K}(\pi))^\s = |\cdot|_{W_K}^{-(n-1)/2}\otimes{\rec}_K(\pi^\s),
		\end{align}
	where $|\cdot|_{W_K}$ is the $\QQ^{\times}$-valued character of $W_K$ induced by the normalized valuation $|\cdot|_K$ of $K$ via local class field theory
	(cf.\ \cite[\S3.2]{ShinTemplier}).
	In particular, we have $|{\rm Frob}_K|_{W_K}=|\varpi_K|_K$, where ${\rm Frob}_K \in W_K$
	is a geometric Frobenius element for $K$ and $\varpi_K$ is a uniformizer of $K$.
	In the category of equivalence classes of finite dimensional Frobenius semisimple Weil-Deligne representations $(\rho, V, N)$ of $W_K$,
	both the restriction functor to $W_L$ and the induction functor from $W_L$ to $W_K$ are compatible with the action of $\Aut(\CC)$.
Namely, we have $(\tilde \rho\big|_{W_L})^\s = (\tilde\rho^\s)\big|_{W_L}$ and ${\rm Ind}_{W_L}^{W_K}(\tilde\rho')^\s = {\rm Ind}_{W_L}^{W_K}(\tilde\rho'^\s)$, for $\tilde\rho=(\rho, V, N)$ of $W_F$ and $\tilde \rho'=(\rho', V', N')$ of $W_L$,
	respectively.

	For $\s \in \Aut(\CC)$, we define the character $\sgn_{K,\s} : K^\times \rightarrow \{\pm1\}$ by
	$\sgn_{K,\s}(x)=(|x|^{1/2}_K)^\s/|x|_K^{1/2}$. The corresponding character on $W_K$
	is also denoted by $\sgn_{K,\s}$.
	Then $\rec_{K}(\pi)^\s=(\sgn_{K,\s})^{n-1}\otimes \rec_K(\pi^\s)$ holds for any irreducible admissible representation $\pi$ of $\GL_{n}(K)$.
	
	\begin{lem}\label{local AI preserve}
		For an irreducible admissible representation $\pi$ of $\GL_n(L)$,
		the local automorphic induction $\AI_{L/K}(\pi)$ is a representation of $\GL_{nd}(K)$.
		Then, we have
		$$\AI_{L/K}(\pi)^\s = (\sgn_{K,\s})^{n(d-1)}\otimes\AI_{L/K}(\pi^\s)$$
		for any $\s \in \Aut(\CC)$.
		In particular, $\AI_{L/K}(\pi)^\s=\AI_{L/K}(\pi^\s)$ holds if $n=2$.
	\end{lem}
	\begin{proof}Take any $\s \in \Aut(\CC)$.
		Briefly we write $\AI$ for $\AI_{L/K}$.
		By $\rec_K(\AI(\pi)) = \Ind_{W_L}^{W_K}(\rec_L(\pi))$, we obtain
		\begin{align*}
			\rec_K(\AI(\pi)^\s) = \,& (\sgn_{K,\s})^{nd-1}\otimes\rec_K(\AI(\pi))^\s \\
			= \, &(\sgn_{K,\s})^{nd-1}\otimes \Ind_{W_L}^{W_K}(\rec_L(\pi))^\s \\
			= \, & (\sgn_{K,\s})^{nd-1}\otimes\Ind_{W_L}^{W_K}(\rec_L(\pi)^\s)\\
			= \,&
			(\sgn_{K,\s})^{nd-1}\otimes\Ind_{W_L}^{W_K}((\sgn_{L,\s})^{n-1}\otimes\rec_L(\pi^\s)).
		\end{align*}	
		By noting $\Ind_{W_L}^{W_K}((|\cdot|_{W_L}^s)^\s \otimes \tilde \rho)=(|\cdot|_{W_K}^s)^\s \otimes\Ind_{W_L}^{W_K}(\tilde\rho)$ for any $s\in \CC$
		and any Weil-Deligne representation $\tilde \rho$ of $W_L$,
		$\rec_K(\AI(\pi)^\s)$ is described as
		\begin{align*}
			&
			(\sgn_{K,\s})^{nd-1}\otimes(\sgn_{K,\s})^{-(n-1)}
			\otimes \Ind_{W_L}^{W_K}(\rec_L(\pi^\s)) \\
			=&\, (\sgn_{K,\s})^{n(d-1)}\otimes 
			\rec_K(\AI(\pi^\s)).
		\end{align*}
	As $\rec_K$ is bijective, we are done.
	\end{proof}
	
	\begin{lem}\label{field of AI}
		Let $F$ be a cyclic extension of $\QQ$ of prime degree $d$.
		Let $\pi=\otimes_v \pi_v$ be an irreducible cuspidal automorphic representation of $\GL_2(\AA_F)$
		and assume that $\Pi=\AI_{F/\QQ}(\pi)$ is cuspidal and that both $[\QQ(\pi):\QQ]$ and $[\QQ(\Pi):\QQ]$ are finite.
		Then we have $[\QQ(\Pi):\QQ]\le [\QQ(\pi):\QQ] \le d!\,[\QQ(\Pi):\QQ]$ and the conductor $\gf_\Pi\in \NN$ of $\Pi$ satisfies $\gf_\Pi = D_F^2\nr_{F/\QQ}(\gf_\pi)$.
	\end{lem}
	\begin{proof}
		We check the claim $\AI_{F/\QQ}(\pi)^\s=\AI_{F/\QQ}(\pi^\s)$ for any $\s \in \Aut(\CC)$.
By the strong multiplicity one theorem for $\GL_{2d}$ (\cite{PS} and \cite[Corollary 4.10]{JS}), it suffices to check the claim only for local components $\pi_w$ such that $\pi_w$ is unramified and a place $w$ of $F$ divides a prime number $p$ unramified in $F$.
		If $p$ is inert, then $\AI_{F_p/\QQ_p}(\pi_p)^\s \cong \AI_{F_p/\QQ_p}(\pi_p^\s)$ holds by Lemma \ref{local AI preserve}.
		If $p$ splits into $w_1, \ldots, w_d$ in $F$, then $F_p=F\otimes_{\QQ}\QQ_p$ is decomposed into $\oplus_{j=1}^{d}F_{w_j}\cong \QQ_p^d$.
			Let $(\a_{w_{j}}, \beta_{w_{j}})$ for each $j$ be the Satake parameter of $\pi_{w_j}$. Then the Satake parameter of $\pi_{w_j}$ is given as $\sgn_{F_{w_j},\s}(p)\times(\a_{w_j},\beta_{w_j})$ by \eqref{twisted Langlands}.
			Since the Satake parameter of $\Pi_p=\AI_{F_p/\QQ_p}(\boxtimes_{j=1}^{d} \pi_{w_j})$ is $\oplus_{j=1}^{d}(\a_{w_{j}}, \beta_{w_{j}})$, the Satake parameter of  $\Pi_{p}^\s$ is equal to $\sgn_{\QQ_p, \s}(p)^{2d-1}\times \oplus_{j=1}^{d}(\a_{w_{j}}, \beta_{w_{j}})$ by \eqref{twisted Langlands}.
		Since $\sgn_{\QQ_p,\s}=\sgn_{F_{w_j},\s}$ holds for all $j \in \{1,\ldots,d\}$,
		the Satake parameter of $\Pi_{p}^\s$ is rewritten as $\oplus_{j=1}^{d}\sgn_{F_{w_j},\s}(p) (\a_{w_{j}}, \beta_{w_{j}})$,
		and this is nothing but the Satake parameter of $\AI_{F_p/\QQ_p}(\boxtimes_{j=1}^{d}\pi_{w_j}^\s)$.
		Thus $\AI_{F/\QQ}(\pi)_p^\s\cong \AI_{F/\QQ}(\pi^\s)_p$
		holds for any prime numbers $p$ splitting in $F$. Hence the claim holds by the strong multiplicity one theorem.
		
		Let us consider fields of rationality. If $\pi^\s=\pi$, then this implies $\AI_{F/\QQ}(\pi)^\s=\AI_{F/\QQ}(\pi^\s)=\AI_{F/\QQ}(\pi)$,
		and whence $\QQ(\Pi) \subset \QQ(\pi)$.
		
		If $\Pi^\s=\AI_{F/\QQ}(\pi)^\s = \AI_{F/\QQ}(\pi)=\Pi$,
		then we have $\pi^\s \in \AI_{F/\QQ}^{-1}(\Pi)=\{\pi^{\tau^j} \mid 1\le j \le d \}$
		(cf.\ \cite[Theorems 4.2 (e) and 5.1]{ArthurClozel} and \cite[xii]{ArthurClozel}), and hence $\pi^\s=\pi^{\tau^j}$ for some $j$.
		As $\s$ induces a permutation of $\{\id, \tau,\ldots, \tau^{d-1}\}=\Gal(F/\QQ)$,
		we obtain the inequality $[\QQ(\pi): \QQ] \le d![\QQ(\Pi):\QQ]$ by the argument similar to the latter part of the proof of the local result \cite[Proposition 5.2]{ShinTemplier}.
		
		The second assertion is given by
		$\e(s, \Pi)=\e(s, \pi)$, $\e(s,\pi) = \e(1/2,\pi)(D_F^2\nr_{F/\QQ}(\gf_\pi))^{1/2-s}$ and $\e(s,\Pi) = \e(1/2,\Pi)\gf_\Pi^{1/2-s}$
		with the aid of \cite{JPSS}.
	\end{proof}

	\begin{cor}\label{lower bound for Hecke field}
		Let $F/\QQ$ be a totally real cyclic extension of prime degree $d$. 
		We fix a maximal ideal $\gq$ of $\Ocal_F$ and a quadratic character $\eta$ of $F^\times\bsl\AA_F^\times$ such that $\gf_\eta$ is coprime to $\gq$.
		Let $\{\gn_k\}_k$ be a family of non-zero ideals of $\Ocal_F$
		in $\Ical_{S(\gq), \eta}^{+}$
		such that $\lim_{k\rightarrow \infty}\nr_{F/\QQ}(\gn_k)=\infty$.
		Then there exist a dominant weight $\lambda' \in X^+(T_{2d})$ and a family $\{\Pi_k\}_{k\in \NN}$ of irreducible $\lambda'$-cohomological cuspidal automorphic representations of $\GL_{2d}(\AA_\QQ)$ such that
		$$\gf_{\Pi_k}=D_F^2\nr_{F/\QQ}(\gn_k)\ZZ, \qquad
		[\QQ(\Pi_k):\QQ] \gg \sqrt{\log \log \nr_{F/\QQ}(\gn_k)}, \qquad
		L(1/2, \Pi_k)\neq 0,$$
		where the implied constant is independent of $\{\gn_k\}_{k\in \NN}$.
		
		Furthermore, for any $q \in \ZZ$ with $d^2\le q \le d^2+d-1$, we have
		$$\dim H_{\rm cusp}^{q}(S^{\GL_{2d}}_{\bfK_1(D_F^2\nr_{F/\QQ}(\gn_k)\ZZ)}, \tilde V_{\lambda'}) \ge
		C \sqrt{\log\log \nr_{F/\QQ}(\gn_k)}$$
		with a constant $C>0$ independent of $\{\gn_k\}_{k \in \NN}$.
	\end{cor}
	
	\begin{proof}
		Fix a family $l=(l_1,\ldots,l_d)$ of even integers such that $l_1>\cdots> l_{d}\ge 6$. By Corollary \ref{nonzero L and growth of Hecke},
		we can take an infinite family $\{\pi_k\}_{k \in \NN}$ such that $\pi_k \in \Pi_{\rm cus}^*(l, \gn_k)$ and $L(1/2,\pi_k)\neq0$.
		Set $\Pi_k=\AI_{F/\QQ}(\pi_k)$.
		Then we have the assertion
		with the aid of Corollary \ref{nonzero L and growth of Hecke}, \eqref{dim cusp coho}, Lemmas \ref{AI is coho}
		and \ref{field of AI}. Remark that the non-parallel condition of $l$ yields the cuspidality of $\Pi_k$ as we saw.
		As for the range of $q$, one refers to \cite[Proposition 2.15]{Raghuram}.
	\end{proof}

	\subsection{Quadratic base change liftings}
	\label{Quadratic base change liftings}
	
	Let $F$ be an arbitrary totally real number field of finite degree
	and $(l_{v})_{v|\infty}$ a family of even integers such that $l_v\ge6$ for all $v|\infty$. 						In this subsection $(l_v)_{v|\infty}$ may be parallel or non-parallel.
	Let $\eta$ be a quadratic character of $F^{\times}\bsl \AA_{F}^{\times}$ such that $\gf_\eta \neq \Ocal_F$.
	Let $\{\gn_k\}_{k\in \NN}$ be a family of non-zero ideals of $\Ocal_F$ coprime to $\gf_{\eta}$ such that $\lim_{k \rightarrow \infty} \nr_{F/\QQ}(\gn_k)=\infty$.
	Then there exists a family of $\pi_k' \in \Pi_{\rm cus}^*(l,\gn_k)$ such that $L(1/2,\pi_k')L(1/2,\pi_k' \otimes \eta)\neq 0$
	and $[\QQ(\pi_k'):\QQ] \gg \sqrt{\log \log \nr_{F/\QQ}(\gn_k)}$ by Corollary \ref{nonzero L and growth of Hecke}.
	We consider the base change lift $\Pi'_k:=\BC_{E/F}(\pi_k')$ to $\GL_2(\AA_E)$ for any $k$, where $E$ is the quadratic extension of $F$
	associated to $\eta$ by class field theory.
	Then $E/F$ is a ramified extension by $\gf_\eta\neq\Ocal_F$.
	Since $\gf_{\pi_{k}' \otimes \eta}=\gf_{\pi_k'}\gf_\eta^2 \neq \gf_{\pi_{k}'}$ holds,
	we obtain $\pi_k' \not \cong \pi_{k}' \otimes \eta$. Hence, by virtue of \cite[Theorems 4.2 (a) and 5.1]{ArthurClozel},
	$\Pi_k'$ is cuspidal.
	Furthermore, $L(1/2,\Pi_k') = L(1/2,\pi_k')L(1/2,\pi_k' \otimes \eta)\neq 0$.
	
	For any archimedean places $w$ of $E$ and $v$ of $F$ such that $w|v$, $\phi_{(\Pi_k')_w}=\phi_{(\pi'_{k})_v}|_{W_{E_w}}$ can be explicitly given with the aid of \eqref{archi L-para}.
	Therefore $\phi_{(\Pi_k')_w}|_{W_\CC}$ equals
	$(\frac{z}{\overline{z}})^{(l_v-1)/2}\oplus (\frac{z}{\overline{z}})^{-(l_v-1)/2}$ by \cite[\S3.1.5]{Raghuram-Tanabe},
	and hence $\Pi_k'$ is C-algebraic and regular. This implies that $\Pi_k'$ is cohomological.
	
	For proving Corollary \ref{growth for BC}, 
	the following counterpart of Lemma \ref{local AI preserve} is needed.
	\begin{lem}\label{local BC preserve}
		Let $L/K$ be a finite extension of non-archimedean local fields.
		For any irreducible admissible representation $\pi$ of $\GL_n(K)$,
		the local base change lift $\BC_{L/K}(\pi)$ is a representation of $\GL_n(L)$. We have
		$$\BC_{L/K}(\pi)^\s = \BC_{L/K}(\pi^\s)$$
for any $\Aut(\CC)$.
	\end{lem}
	\begin{proof}
Take any $\s \in \Aut(\CC)$. Briefly we write $\BC$ for $\BC_{L/K}$.
By $\rec_{L}\BC(\pi)=\rec_K(\pi)|_{W_L}$, we observe
		{\allowdisplaybreaks\begin{align*}
				\rec_L(\BC(\pi)^\s)=& (\sgn_{L,\s})^{n-1}\otimes \rec_L(\BC(\pi))^\s \\
				=&  (\sgn_{L,\s})^{n-1}\otimes (\rec_K(\pi)\big|_{W_L})^\s\\
				=&  (\sgn_{L,\s})^{n-1}\otimes (\rec_K(\pi)^\s)\big|_{W_L}\\
				=&  (\sgn_{L,\s})^{n-1}\otimes \{(\sgn_{K,\s})^{n-1}\otimes \rec_K(\pi^\s)\}\big|_{W_L}.
			\end{align*}
		}By $\sgn_{K,\s}|_{W_L} = \sgn_{L,\s}$ due to the identity $|\nr_{L/K}(\cdot)|_{K} = |\cdot|_{L}$,
			the last line says
			\begin{align*}
				&\rec_L(\BC(\pi)^\s)=(\sgn_{L,\s})^{2(n-1)}\otimes \rec_K(\pi^\s)\big|_{W_L} = \rec_K(\pi^\s)\big|_{W_L} = \rec_L(\BC(\pi^\s)).
			\end{align*}
		Hence we are done as $\rec_L$ is bijective.
	\end{proof}

	A counterpart of Lemma \ref{field of AI} is also needed for proving Corollary \ref{growth for BC}.
	
	\begin{lem}\label{field of BC}
		Let $E/F$ be a quadratic extension of number fields.
		Let $\pi=\otimes_v\pi_v$ be an irreducible cuspidal automorphic representation of $\GL_2(\AA_F)$.
		Assume that $\Pi=\BC_{E/F}(\pi)$ is cuspidal,
		that $\gf_\pi$ is coprime to the discriminant of $E/F$,
		and that both $[\QQ(\pi):\QQ]$ and $[\QQ(\Pi):\QQ]$ are finite..
		Then we have $[\QQ(\Pi):\QQ]\le [\QQ(\pi):\QQ] \le 2\, [\QQ(\Pi):\QQ]$
		and $\nr_{E/\QQ}(\gf_\Pi) = \nr_{F/\QQ}(\gf_\pi^2)$.
	\end{lem}
	\begin{proof}
		 By the strong multiplicity one theorem for $\Res_{E/F}\GL_{2,E}$ (\cite{PS} and \cite[Corollary 4.10]{JS}), it suffices to check the claim only for local components $\pi_w$ such that $\pi_w$ is unramified and $w$ divides a prime number $p$ unramified in $F$.
		Let $v$ be a finite place of $F$ inert in $E$. Then, $\BC_{E_v/F_v}(\pi_v)^\s=\BC_{E_v/F_v}(\pi_v^\s)$ holds
		by Lemma \ref{local BC preserve}.
		Next we consider the case where $v$ is a split place of $F$
		and $\pi_v$ is unramified with Satake parameter $(\a_v, \beta_v)$.
		Then $\Pi_v = \BC_{E_v/F_v}(\pi_v)$, a representation of $\GL_2(F_v)\times \GL_2(F_v)$, is unramified with Satake parameter $((\a_v, \beta_v), (\a_v, \beta_v))$.
		Thus
		$\BC_{E_v/F_v}(\pi_v)^\s=\BC_{E_v/F_v}(\pi_v^\s)$ holds for all $\s \in \Aut(\CC)$.
		The strong multiplicity one theorem shows $\BC_{E/F}(\pi)^\s=\BC_{E/F}(\pi^\s)$.
		Furthermore,
		if $\Pi^\s=\BC_{E/F}(\pi)^\s=\BC_{E/F}(\pi)=\Pi$, then we have $\pi^\s \in \BC_{E/F}^{-1}(\Pi)=\{\pi, \pi\otimes \eta\}$ by \cite[Theorems 4.2 (d) and 5.1]{ArthurClozel}.
		Consequently
		we obtain $\QQ(\Pi) \subset \QQ(\pi)$ and $[\QQ(\pi):\QQ]\le 2\, [\QQ(\Pi):\QQ]$
		in a similar way to Lemma \ref{field of AI}.
		
		Let $\eta$ be the quadratic character of $F^\times\bsl \AA_F^\times$
		corresponding to $E/F$ by class field theory.
		Then, $\gf_\eta$ equals the discriminant $\Dcal_{E/F}$ of $E/F$.
		By $\e(s, \Pi)= \e(s, \pi)\e(s, \pi\otimes\eta)$ and $\e(s, \pi\otimes\eta) = \e(1/2,\pi)(D_F^2\nr_{F/\QQ}(\gf_{\pi}\gf_\eta^2))^{1/2-s}$
		with the aid of \cite{JPSS},
		we have $$D_E^2\nr_{E/\QQ}(\gf_{\Pi}) = D_F^2\nr_{F/\QQ}(\gf_\pi)\times D_F^2\nr_{F/\QQ}(\gf_\pi \gf_{\eta}^2)
		= \nr_{F/\QQ}(\gf_\pi^2) \times \{D_F^2\nr_{F/\QQ}(\Dcal_{E/F})\}^2.$$
		Consequently we obtain $\nr_{E/\QQ}(\gf_\Pi) = \nr_{F/\QQ}(\gf_\pi^2)$ by using
		$D_E = D_F^2\nr_{F/\QQ}(\Dcal_{E/F})$.
	\end{proof}

	\begin{proof}[Proof of Corollary \ref{growth for BC}]
		Corollary \ref{growth for BC} follows from Corollary \ref{nonzero L and growth of Hecke}, Lemma \ref{field of BC} and \eqref{dim cusp coho} for
		$\Res_{E/\QQ}\GL_{2, E}$. Since $E$ has at least two archimedean places in general, we need to modify  \eqref{dim cusp coho} as follows.
		
		Let	$\lambda=(\lambda_\iota)_{\iota:E\hookrightarrow\CC}$ be a family of dominant integral weights indexed by the embeddings of $E$ into $\CC$.
		Let $\Pi$ be a $\lambda$-cohomological irreducible cuspidal automorphic representation of $\GL_2(\AA_E)$ such that $\Pi_\fin^\Ucal\neq 0$ for an open compact subgroup $\Ucal \subset \GL_2(\AA_{E,\fin})$.
		For $\sigma \in \Aut(\CC)$, set $\lambda^\s=(\lambda_{\sigma \circ \iota})_{\iota : E\hookrightarrow \CC}$.
		Then there exists a unique $\lambda^\sigma$-cohomological irreducible cuspidal automorphic representation $\Pi'=\Pi_\infty'\otimes \Pi_\fin'$ of $\GL_2(\AA_E)$
		such that $\Pi'_{\fin}\cong \Pi_\fin^\s$ (\cite[Th\'eor\`eme 3.13]{ClozelMotif}). Then $\Pi'$ is denoted as $\Pi^\s$ and $\Pi_\infty'$ as $\Pi_\infty^\s$, respectively.

		We write $\QQ(\lambda)$ for the subfield of $\CC$ consisting of the complex numbers fixed by all $\sigma \in \Aut(\CC)$ satisfying $\lambda^\sigma=\lambda$. Then $\QQ(\lambda)$ is a number field of finite degree contained in the Galois closure $E^{\Gal}$ of $E/\QQ$, and $\Aut(\CC/\QQ(\lambda))$ acts on the set of $\lambda$-cohomological irreducible cuspidal automorphic  representations of $\GL_2(\AA_E)$.
		
		For the $\lambda$-cohomological cuspidal representation $\Pi$ and for any $\s \in \Aut(\CC)$, the equality
		$\lambda^\s=\lambda$ holds if and only if $\Pi_\infty^\s\cong \Pi_\infty$.
		Indeed, $\Pi_\infty^\s\cong \Pi_\infty$ implies the invariance
		$(\lambda_\iota+\rho_2)_{\iota:E\hookrightarrow \CC}^\s=(\lambda_\iota+\rho_2)_{\iota:E\hookrightarrow \CC}$ of the infinitesimal character of $\Pi_\infty^\vee$, and hence $\lambda^\s=\lambda$. Conversely,
		$\lambda^\s=\lambda$ implies $\Pi_\infty^\s =\Pi_\infty$ by the explicit formulas \cite[Propositions 2.9 and 2.14]{Raghuram} of the $\lambda$-cohomological representations at the archimedean places.
		
			Let $\mathfrak{g}_\infty$ be the Lie algebra of $\GL_2(E\otimes_\QQ \RR)$, and
$\mathfrak{k}_\infty^+ \subset \mathfrak{g}_\infty$ stands for the Lie algebra of
			$$\prod_{w : \text{real}}\RR_{>0}\SO(2) \times \prod_{w : \text{complex}} \CC^\times\U(2).$$
The estimate \eqref{dim cusp coho} is modified as
		{\allowdisplaybreaks\begin{align*}
				&\dim_\CC H_{\rm cusp}^{q}(S_\Ucal^{\Res_{E/\QQ}\GL_2}, \tilde V_\l) \\
				\ge & \sum_{\s \in \Aut(\CC/\QQ(\lambda)\QQ(\Pi))\bsl\Aut(\CC/\QQ(\lambda))}\dim_{\CC} \{H^q(\mathfrak{g}_\infty/ \mathfrak{k}_\infty^+;  \Pi_\infty\otimes V_{\l}) \otimes (\Pi_{\fin}^{\s})^{\Ucal}\} \notag\\
				\ge & \#(\Aut(\CC/\QQ(\lambda)\QQ(\Pi))\bsl\Aut(\CC/\QQ(\lambda)))\\
				= & [\QQ(\lambda)\QQ(\Pi) : \QQ(\lambda)] =\frac{[\QQ(\lambda)\QQ(\Pi):\QQ]}{[\QQ(\lambda):\QQ]}\ge \frac{[\QQ(\Pi):\QQ]}{[\QQ(\lambda):\QQ]} \ge \frac{[\QQ(\Pi):\QQ]}{[E^{\Gal}:\QQ]}\gg_E[\QQ(\Pi):\QQ],
			\end{align*}
		}where we use $\Aut(\CC/\QQ(\Pi))\cap\Aut(\CC/\QQ(\lambda))=\Aut(\CC/\QQ(\lambda)\QQ(\Pi))$.
	As for the range of $q$, one refers to \cite[Proposition 2.15]{Raghuram}.
	\end{proof}

\begin{rem}\label{rem for Q(lambda)}
The inequality $[\QQ(\lambda)\QQ(\Pi) : \QQ]\ge [\QQ(\Pi):\QQ]$ above can be improved by the equality
$[\QQ(\lambda)\QQ(\Pi) : \QQ] = [\QQ(\Pi):\QQ]$
since $\QQ(\lambda)$ is contained in $\QQ(\Pi)$.
Indeed, take any $\s\in\Aut(\CC/\QQ(\Pi))$. Then $\s$ satisfies $\Pi_{\fin}^\s\cong \Pi_{\fin}$. As $\Pi^\s$ is cohomological and cuspidal by (\cite[Th\'eor\`eme 3.13]{ClozelMotif}), the strong multiplicity one theorem (\cite{PS} and \cite[Corollary 4.10]{JS}) implies $\Pi^\s \cong \Pi$. In particular, $\Pi_\infty^\s\cong\Pi_\infty$ holds.
Consequently $\lambda^\s=\lambda$ holds as we see.
This leads us to $\Aut(\CC/\QQ(\Pi)) \subset \Aut(\CC/\QQ(\lambda))$ and hence $\QQ(\lambda)\subset \QQ(\Pi)$.
\end{rem}

	\appendix
	\section{Integral structures of automorphic forms on groups from PEL-type}
	\label{appendix}
	
	In this appendix, we introduce an integral structure on
	the space of automorphic forms on a PEL-type Shimura variety
	by using $q$-expansion.
	The technical key ingredient of the construction
	is Lan's comparison theorem \cite{Lan16} between algebraic $q$-expansions
	and analytic ones.

	For a commutative ring $R$, a subring $S\subset R$ and an $S$-module $L$,
	we set $L_R=L\otimes_S R$ throughout this appendix.
	
	\subsection{A note on Mumford's construction of degenerating semi-abelian varieties}
	In this subsection, we fix an excellent integrally closed
	domain $A$ complete with respect to an ideal $I$ such that $I=\sqrt{I}$.
	Let $K$ be the fractional field of $A$ and let $\eta:=\mathrm{Spec}(K)$.
	
	Let $X$ be a free abelian group of rank $r$ and let $\widetilde G$
	be the split torus over $A$ with the character group $X$.
	Explicitly, we have
	\[
	\mathcal O(\widetilde{G})=A[X]=A\left[\mathfrak X^{\alpha} \mid \alpha\in X\right]/\left(\mathfrak X^0-1,\mathfrak X^{\alpha+\beta}-\mathfrak X^{\alpha}\mathfrak X^{\beta}\mid \alpha,\beta\in X \right).
	\]
	A \emph{period} means a subgroup $Y$ of $\widetilde G(K)$ which is free of rank $r$ (\cite[Definition 1.1]{Mumford72}).
	Let
	\[
	\phi\colon Y\to X
	\]
	be a polarization of $Y$ in the sense of Mumford (\cite[Definition 1.2]{Mumford72}), that is, this is a group homomorphism
	satisfying the following two conditions:
	\begin{itemize}
		\item[(a)]$\mathfrak X^{\phi(y)}(z)=\mathfrak X^{\phi(z)}(y)$ holds for any $y,z\in Y$,
		\item[(b)] $\mathfrak X^{\phi(y)}(y)$ is contained in $I$ for any $y\in Y-\{0\} $.
	\end{itemize}
	Any polarization is injective by the condition (b).
	In \cite{Mumford72}, Mumford constructed a polarized semi-abelian
	variety $G=\widetilde G/Y$ over $A$ associated to $(Y,\phi)$,
	which enjoys the following conditions:
	\begin{itemize}
		\item[(i)]The generic fiber $G_\eta:=G\times_{\mathrm{Spec}(A)}\eta$ of $G$ is an abelian variety over $K$.
		\item[(ii)]The special fiber $G_0:=G\times_{\mathrm{Spec}(A)}\mathrm{Spec}(A/I)$ of $G$ is naturally isomorphic to the split turus $\widetilde G\times _{\mathrm{Spec}(A)}\mathrm{Spec}(A/I)$.
		\item[(iii)]For any positive integer $N$, we have the following short exact sequence
		\[
		0\to \widetilde G_\eta[N]\to G_\eta[N]\to Y/NY\to 0
		\]
		of finite flat group schemes over $K$.
		
	\end{itemize}
	In this section, we show the following.
	\begin{lem}\label{Mumordconstructionrefinement}
		Let $\{e_1,\dots,e_r\}$ be a free basis of $Y$ and define an element $l$ of $I$ by
		\[
		l=\prod_{i=1}^r\mathfrak X^{\phi(e_i)}(e_i).
		\]
		Then, the properties $(\mathrm{i})$ and $(\mathrm{iii})$ above also hold
		over $A_l=A[1/l]$. Namely, $G_l=G\times_{\mathrm{Spec}(A)}\mathrm{Spec}(A_l)$
		is an abelian variety and there exists a short exact sequence
		\[
		0\to \widetilde G_l[N]\to G_l[N]\to Y/NY\to 0
		\]
		of finite flat group schemes
		over $A_l$.
	\end{lem}
	\begin{rem}\label{onphi}
		The semi-abelian variety $G$ over $A$ does not depend on $\phi \ ($\cite[Corollary (4.7)]{Mumford72}$)$.
	\end{rem}
	To show this, we recall Mumford's construction of $G$ quickly.
	A key notion for the construction is a \emph{relatively complete model} $\widetilde G$
	over $A$.
	This is a $5$-tuple
	\[
	(\widetilde P,\widetilde{\mathcal L},i\colon \widetilde G\hookrightarrow \widetilde P,T=\{T_a\}_{a\in \widetilde G},S=\{S_y\}_{y\in Y}),
	\]
	where $\widetilde P$ is an integral scheme separated and locally of finite type over $A$, $\widetilde {\mathcal L}$ is an ample invertible sheaf on $\widetilde P$, $i$ is an open immersion, and $T$ and $S$ are actions of $\widetilde G$
	and $Y$ on $(\widetilde P,\widetilde{\mathcal L})$, respectively.
	We do not make the condition of relatively complete model
	clear here. See \cite[Definition (2.1)]{Mumford72} for the precise definition. An important condition what we need for the explanation below is
	that there exists a $\widetilde G$-stable open subscheme $U$
	of $\widetilde P$ of finite type over $A$ such that $S_y(U)$ covers $\widetilde P$ when $y$ runs over all elements of $Y$.
	According to \cite[Theorem (2.5)]{Mumford72}, a relatively complete model exists if we replace $\phi$ with $2n\phi$ for a sufficiently large positive integer $n$. This replacement does not matter for our purpose by the independence of $G$ from $\phi$ (Remark \ref{onphi}).
	
	For each positive integer $n$, the subscription $n$ of objects means
	the base change of the objects over $A$ to $A/I^{n+1}$.
	For example, $\widetilde P_n$
	denotes the base change of $\widetilde P$ to $A/I^{n+1}$.
	Then, for a sufficiently large positive integer $k$,
	the union
	\[
	\bigcup_{y\in kY} {S_{y+\alpha}(\overline U_n)}\subset \widetilde P_n
	\]
	is actually a disjoint union for any $\alpha\in Y$ (\cite[Proposition (3.6)]{Mumford72}).
	Therefore, the quotient
	\[
	P_n'=\widetilde P_n/kY
	\]
	is well-defined by the usual patching procedure.
	By construction, $\widetilde{\mathcal  L}_n$ descends to $P_n'$ and $\mathcal L_n'$
	denotes the resulting ample line bundle on $P_n'$.
	Note that the quotient of $(P_n',\mathcal L_n')$ by the \emph{finite} group
	$Y/kY$ exists. We write $(P_n,\mathcal L_n)$ for this quotient.
	Put
	\[
	\widetilde B:=\left(\widetilde P-\bigcup_{y\in Y}S_y( \widetilde G)\right)_{\mathrm{red}}
	\]
	and define a closed subscheme $B_n$ of $P_n$ by
	\[
	B_n=\widetilde B_n/Y\subset P_n.
	\]
	Then, as $P_n$ is projective over $A/I^{n+1}$, the formal schemes
	\[
	P_\infty:=\varprojlim_n P_n\supset B_\infty=\varprojlim_n B_n
	\]
	over $\mathrm{Spf}(A,I)$ and the ample line bundle $\mathcal L_\infty$
	on $P_\infty$ are algebrizable over $A$ by Grothendieck's GAGF (\cite{G}).
	Namely, there exist a projective variety $P$ over $A$,
	a closed subscheme $B$ of $P$, and an ample invertible sheaf $\mathcal L$ on $P$
	such that the completions of $P$, $B$, and $\mathcal L$ along $I$ are isomorphic to $P_\infty$,
	$B_\infty$, and $\mathcal L_\infty$, respectively.
	Then, the quasi-projective variety $G$ over $A$ which we would like to construct is defined by
	\[
	G=P- B.
	\]
	This scheme does not depend on the choice of the relatively
	complete model (\cite[Corollary (4.7)]{Mumford72}).
	Moreover, there exists a natural group low on $G$ induced by that
	on $\widetilde G$ (\cite[Corollary (4.8)]{Mumford72}) and $G$ is a semi-abelian variety
	under this group low (\cite[Corollary (4.12)]{Mumford72}).
	Now we give a proof of the lemma.
	\begin{proof}[Proof of Lemma $\ref{Mumordconstructionrefinement}$]
		First we show that $G_l$ is an abelian variety over $A_l$.
		According to \cite[Proposition (3.1)]{Mumford72}, each $\mathfrak X^{\phi(e_i)}$
		is a unit on $\widetilde P_l$:
		\[
		\mathfrak X^{\phi(e_i)}\in \Gamma(\widetilde P_l,\mathcal O)^\times.
		\]
		Then, for any $\alpha\in X$, $\mathfrak X^{\alpha}$ is a unit on $\widetilde P_l$
		since $\phi(Y)$ is a finite index subgroup of $ X$.
		This implies that any morphism $\widetilde G_l\to \mathbb G_m$ induced by $\mathfrak X^\alpha$ extends to $\widetilde P_l\to \mathbb G_m$.
		Therefore, we have a commutative diagram
		\[
		\xymatrix{
			\widetilde G_l\ar@{^{(}->}[rrr]^j\ar[drrr]_{\mathrm{id}}& & &\widetilde P_l\ar[d]^{\pi}\\
			& & &\widetilde G_l
		}
		\]
		of schemes over $A_l$. Note that the open immersion $j$ is naturally identified with the pull-back of the graph $\Gamma_{j\pi}\colon \widetilde P_l\to \widetilde P_l\times_{\Spec(A_l)}\widetilde P_l$ of the morphism $j\pi\colon \widetilde P_l\to \widetilde P_l$ by the diagonal morphism.
		Hence, as $\Gamma_{j\pi}$ is closed,
		the base change of $\widetilde P$ to $A_l$ coincides with
		$\widetilde G_l$, i.e.,
		\[
		\widetilde P_l=\widetilde G_l.
		\]
		This implies that $l$ annihilates $\mathcal O_{\widetilde B_n}$ and $\mathcal O_{B_n}$,
		and thus it annihilates $\mathcal O_B$. Therefore, the equation
		\[
		G_l=P_l
		\]
		holds, and we conclude that $G_l$ is projective over $A_l$.
		Therefore the condition (i) for $G$ holds over $A_l$.
		
		Next we show that (iii) holds over $A_l$.
		Let us put $ G^*:=\cup_{y\in Y}S_y(\widetilde G)\subset \widetilde P$
		and let $[N]\colon G^*\to G^*$ denote the induced morphism by $[N]\colon \widetilde G\to \widetilde G$ by abuse of notation.
		For each $y\in Y$, $\sigma_y\colon\mathrm{Spec}(A)\to  G^*$ denotes the
		extension of $y\colon \eta\to \widetilde G$, and $Z^{(N)}_y$ denotes
		the pull-back of $[N]\colon  G^*\to G^*$ by $\sigma_y$.
		According to \cite[Theorem (4.10)]{Mumford72},
		we have a natural isomorphism
		\[
		G[N]=\coprod_{y\in Y/NY}Z^{(N)}_y
		\]
		of finite flat group schemes over $A$.
		As $\widetilde G_l=\widetilde G^*_l=\widetilde P_l$, the scheme $Z^{(N)}_y\times_{\mathrm{Spec}(A)}\mathrm{Spec}(A_l)$ is just the inverse
		image of $\sigma_y\in \widetilde G(A_l)$ under the multiplication by $N$.
		Therefore, we obtain the short exact sequence
		\[
		0\to \widetilde G_l[N]\to G_l[N]\to Y/NY\to 0
		\]
		of finite flat group schemes over $A_l$.
	\end{proof}
	\begin{rem}In \cite[Section 2]{Mumford72}, for a sufficiently large $n$,
		Mumford constructed a relatively complete model of $(\widetilde G,Y,2n\phi)$
		explicitly.
		The assertions of Lemma $\ref{Mumordconstructionrefinement}$ can also be proved
		by using that explicit model directly.
	\end{rem}
	\subsection{Algebraic modular forms arising from PEL-Shimura data}
	Let $B$ be a finite dimensional semisimple algebra over $\Q$
	and let $\mathcal O$ be an order of $B$.
	Let $\star\colon B\to B^{\mathrm{op}}$ be a fixed $\Q$-linear anti-involution of $B$
	satisfying the following positivity condition:
	\begin{equation}
		\begin{split}
			\mathrm{Tr}_{B/\Q}(xx^\star)&\geq 0\text{ for all }x\in B,\\
			\mathrm{Tr}_{B/\Q}(xx^\star)&=0\text{ if and only if }x=0.
		\end{split}
	\end{equation}
	\begin{defn}[{\cite[Definition 1.2]{Lan13}}]
		A PEL-type $\mathcal O$-lattice is a triple $(L,\langle \ ,\ \rangle ,h_0)$, where
		\begin{itemize}
			\item[-]$L$ is a left $\mathcal O$-module which is a free $\Z$-module of finite rank.
			\item[-]$\langle \ ,\ \rangle\colon L\times L\to \Z(1)$ is an alternating non-degenerate pairing satisfying
			$\langle bx,y\rangle=\langle x,b^\star y\rangle$ for all $x,y\in L$ and $b\in \mathcal O$.
			\item[-] $h_0\colon \C\to \End_{B_{\R}}(L_{\R})$ is an $\R$-algebra homomorphism such that
			$\langle h_0(z)x, y\rangle=\langle x, h_0(\bar z)y\rangle$ for all $x,y\in L_{\R}$ and $z\in \C$
			and such that $\frac{1}{2\pi\sqrt{-1}}\langle \ ,\ h_0(\sqrt{-1})\rangle\colon L_{\R}\times L_{\R}\to \R$
			is symmetric and positive definite. 
		\end{itemize}
		\label{PELdatum}
	\end{defn}
	\begin{ex}We see three examples of $\mathcal O$-lattices of PEL-type.
		\begin{enumerate}
			\item {\rm(}Siegel case{\rm)}
			Set $F=\mathbb Q$ and let $L$ be the free $\mathbb Z$-module $\mathbb Z^{2n}$ of rank $2n$ equipped with the alternating pairing $\langle\ ,\ \rangle$ defined by
			\[
			\langle v,w\rangle=2\pi\sqrt{-1}\left(-\sum_{i=1}^nv_iw_{i+n}+\sum_{i=1}^nv_{i+n}w_i\right).
			\]
			Let $h_0\colon \mathbb C\to \mathrm{End}_{\mathbb R}(L_{\mathbb R})=\mathrm{M}_{2n}(\mathbb R)$ be the homomorphism of $\mathbb R$-algerba defined by
			\[
			h_0(\sqrt{-1})=\left[\begin{smallmatrix}
				0_n&-1_n\\
				1_n&0_n
			\end{smallmatrix}\right].
			\]
			Then the triple $(L,\langle\ ,\ \rangle ,h_0)$ is a $PEL$-type $\mathbb Z$-lattice.
			Here the involution $\star$ is the identity.
			\item {\rm(}\!``Hilbert case''{\rm)}
			Let $F$ be a totally real number field of degree $d$ and let $\mathcal O$
			denote the ring of integers of $F$.
			Let $\mathfrak a$ and $\mathfrak b$ be non-zero fractional ideals of $F$
			and set $L:=\mathfrak a\oplus\mathfrak b$.
			We define an alternating pairing
			\[
			\langle\ ,\ \rangle\colon L\times L\to \mathbb Q(1)
			\]
			by
			\[
			\left\langle \begin{pmatrix}
				v_1\\
				v_2
			\end{pmatrix},\begin{pmatrix}
				w_1\\
				w_2
			\end{pmatrix}\right\rangle=2\pi\sqrt{-1}\mathrm{Tr}_{F/\mathbb Q}(v_1w_2-v_2w_1).
			\]
			Here we identify $L_{\mathbb R}$ with $F_{\mathbb R}^2$ by the  canonical isomorphisms $\mathfrak a\otimes_{\mathbb Z}\mathbb R=F\otimes_{\mathbb Q}\mathbb R$ and $ \mathfrak b\otimes_{\mathbb Z}\mathbb R=F\otimes_{\mathbb Q}\mathbb R$.
			Let $h_0$ be the homomorphism $\mathbb C\to \mathrm{End}_{F_{\mathbb R}}(F_{\mathbb R^2})$
			defined by
			\[
			h_0(\sqrt{-1})=\begin{pmatrix}
				0&\mathrm{id}_{F_{\mathbb R}}\\
				-\mathrm{id}_{F_{\mathbb R}}&0
			\end{pmatrix}\in \mathrm{End}_{F_{\mathbb R}}(F_{\mathbb R^2}).
			\]
			Then, if $\mathfrak b$ is contained in $\mathfrak a^{-1}\mathcal D_{F/\mathbb Q}^{-1}$, where $\mathcal D_{F/\mathbb Q}$
			is the global different of $F$ over $\mathbb Q$, then
			the triple $(L,\langle\ ,\ \rangle,h_0)$ is a PEL-type
			$\mathcal O$-lattice.
			The involution $\star$ on $F$ is also taken to be the identity in this case
			{\rm (}see Remark $\ref{remHMF}${\rm )}.
			
			\item  The similitude unitary group $\GU_{p, q}$ also arises from a PEL-datum. See \cite[Subsection 2.2, 2.3]{EischenHarrisLiSkinner} for $\GU_{p,q}$ and related groups.
			
		\end{enumerate}
		\label{ex1}
		
	\end{ex}
	For the rest of this section, we fix $(B,\star,\mathcal O)$
	and a PEL-type $\mathcal O$-lattice $(L,\langle \ ,\ \rangle ,h_0)$.
	We define the group scheme $\bf G$ of finite type over $\Z$ by the equation
	\[
	\mathbf G(R):=\left\{(g,r)\in \GL_{\mathcal O_R}(L_R)\times R^\times \mid \langle gx,gy\rangle=r\langle x,y\rangle\text{ for all }x,y\in L_R \right\},
	\]
	where $R$ is an arbitrary $\Z$-algebra.
	For example, in the Siegel case above, $\mathbf G$ is naturally isomorphic
	to the general symplectic group $\mathrm{GSp}_{2n,\mathbb Z}$ over $\mathbb Z$
	of rank $n$.
	
	\begin{rem}\label{remHMF}
		In ``Hilbert case'', the associated algebraic group $\mathbf G\times_{\mathrm{Spec}(\mathbb Z)}\mathrm{Spec}(\mathbb Q)$ is \emph{not} isomorphic
		to $\mathrm{Res}_{F/\mathbb Q}\mathrm{GL}_{2,F}$ but isomorphic to the
		subgroup of $\mathrm{Res}_{F/\mathbb Q}\mathrm{GL}_{2,F}$ consisting of $2\times 2$ matrices with \emph{rational determinants}.
		The holomorphic automorphic forms on $\mathbf G$ are Katz's Hilbert modular forms {\rm (}\cite{Katz}{\rm )}
		but not the usual Hilbert modular forms {\rm (}\cite{Shimura}{\rm )}.
	\end{rem}

		Let $Z(\mathbf G)$ be the center of $\mathbf G$. By definition, $Z(\mathbf G)(R)$ is contained in the center of the abstract group $\mathbf G(R)$ for any commutative $\mathbb{Q}$-algebra $R$ (\cite[Proposition 1.92]{Milne}). Let $Z_\infty$ be the center of $\mathbf G(\R)$, which contains $Z(\mathbf G)(\R)$.
	Let $\mathcal U_\infty\subset \mathbf G(\R)$ be the stabilizer of $h_0$. The following lemma is well-known.
	\begin{lem}[{\cite[Section 2.3]{Lan12}}]
		The quotient manifold $\mathsf X:=\mathbf G(\R)/\mathcal U_\infty$
		has a natural complex structure.
		\label{lemX}
	\end{lem}
	For example, in the Siegel case, this $\mathsf X$ is naturally isomorphic to the Siegel upper-lower half-space, which is isomorphic to the disjoint union of two Siegel upper half-spaces.
			\begin{lem}\label{commensurable}
			There exists a maximal compact subgroup $K_\infty$  of $\mathbf G(\R)$ such that the groups $\mathcal U_\infty$ and $Z_{\infty}K_\infty $ are commensurable.
		\end{lem}
		\begin{proof}
			Let $\mathbf G^{\mathrm{der}}$ be the derived subgroup of $\mathbf G\times_{\mathrm{Spec}(\Z)}\mathrm{Spec}(\Q)$ and let $\mathbf G^{\mathrm{der}}(\R)^+$ be the connected component of $\mathbf G^{\mathrm{der}}(\R)$ containing the unit element. According to \cite[Proposition 19.21, Proposition 19.5]{Milne}, $\mathbf G^{\mathrm{der}}\times_{\mathrm{Spec}(\Q)}\mathrm{Spec}(\R)$ is a semisimple algebraic group. Therefore, its Lie algebra is a semisimple Lie algebra (\cite[p.476]{Milne}). Thus $\mathbf G^{\mathrm{der}}(\R)^+$ is a semisimple Lie group.
			Put $K'_0:=\mathcal U_\infty\cap \mathbf G^{\mathrm{der}}(\R)^+$.  Since $\mathbf G^{\mathrm{der}}(\R)^+/K'_0$ is isomorphic to a connected component of a Hermitian symmetric domain, $K'_0$ is a maximal compact subgroup of $\mathbf G^{\mathrm{der}}(\R)^+$.
			This implies that $K':=\mathcal U_\infty\cap \mathbf G^{\mathrm{der}}(\R)$  is a compact subgroup of $ \mathbf G^{\mathrm{der}}(\R)$. We will show that any maximal compact subgroup $K_\infty $ of $\mathbf G(\R)$ containing $K'$ satisfies the condition in the lemma.
			
			The group $Z_\infty \mathbf G^{\mathrm{der}}(\R)$ is a finite index subgroup of  $\mathbf G(\R)$ since $Z(\mathbf G)(\R)\mathbf G^{\mathrm{der}}(\R)$ is a finite index subgroup of $\mathbf G(\R)$ by \cite[Proposition 12.46]{Milne}. Therefore, it suffices to show that  $\mathcal U_\infty\cap Z_\infty \mathbf G^{\mathrm{der}}(\R) $ and $Z_\infty K_\infty$ are commensurable.
			Since $\mathcal U_\infty$ contains $Z_\infty$ by definition, we have 
			\[
			\mathcal U_\infty\cap Z_\infty \mathbf G^{\mathrm{der}}(\mathbb R)=Z_\infty K'.
			\]
			Since $K'$ is contained in $K_\infty$, it suffices to show that $K'$ is a finite index subgroup of $K_\infty\cap \mathbf G^{\mathrm{der}}(\R)$.
			Since $K'_0$ is maximal and contained in $K_\infty \cap\mathbf G^{\mathrm{der}}(\R)^+$, $K_\infty\cap \mathbf G^{\mathrm{der}}(\R)^+$ coincides with $K'_0$. Thus, we have a natural inclusion
			\[
			K_\infty\cap \mathbf G^{\mathrm{der}}(\R)/K'_0\hookrightarrow \mathbf G^{\mathrm{der}}(\R)/\mathbf G^{\mathrm{der}}(\R)^+
			\]
			of quotient sets. Since $\mathbf G^{\mathrm{der}}(\R)^+$ is a finite index subgroup of $\mathbf G^{\mathrm{der}}(\R)$, $K_0'$ is a finite index subgroup of $K_\infty\cap \mathbf G^{\mathrm{der}}(\R)$. In particular, $K'$ is also a finite index subgroup of $K_\infty\cap \mathbf G^{\mathrm{der}}(\R)$.
			This completes the proof of the lemma.
		\end{proof}
	
		From now, we pick a maximal compact subgroup $K_\infty$ of $\mathbf G(\mathbb R)$ satisfying the condition in the lemma above.
	
	Let $V$ be the maximal quotient of the $\C$-vector space 
	$L\otimes_{\Z}\C$ on which $h_0(\sqrt{-1})$ acts on $V$ as the multiplication
	by $1\otimes\sqrt{-1}$.
	Namely,
	\[
	V:=(L\otimes_{\Z}\C)/\left(h_0(\sqrt{-1})x-(1\otimes\sqrt{-1})x \mid x\in L\otimes_{\Z}\R \right).
	\]
	Let $F$ denote a number field of finite degree containing the field of definition of $V$,
	that is, $V_0$ is a $B_{F}$-submodule of $V$
	such that $V_0\otimes_{{F}}\C=V$.
	Later, we will take $F$ containing the reflex field of
	the fixed PEL-type $\mathcal O$-lattice.
	We define the algebraic group $M$ over $F$ by
	\[
	M:=\GL_{B_{F}}(V_0^\vee(1))\times \mathbb G_m.
	\]
	The rank of $M$ is $\frac{1}{2}\dim_{\Q}B \times {\rm rank}_{\Z}L+1$.
	We have an isomorphism
	\[
	L\otimes\mathbb C\cong V\oplus V^{\vee}(1)
	\]
	as symplectic spaces, which is compatible with actions of $B_{\mathbb C}$ on both-hand sides. Therefore, we have
	\[
	\mathbf G\times_{\mathrm{Spec}(\mathbb Q)} \mathrm{Spec}(\mathbb C)\cong \mathrm{GSp}_{B_{\mathbb C}}(V\oplus V^{\vee}(1)).
	\]
	Since $M$ is isomorphic to the Levi quotient of the Borel subgroup associated to $V\oplus V^{\vee}(1)$,
	we have $\mathrm{rank}(M)=\mathrm{rank}(\mathbf G)$.
	
	We remark that any algebraic representation of $M$ is regarded as a representation of $K_\infty \times Z_\infty$ via
	\[
	K_{\infty , \C}\cong \GL_{B_F}(V_0^{\vee}(1))(\C)
	\]
	and via
	\[
	K_\infty \times Z_\infty \rightarrow M(\C)=\GL_{B_{F}}(V_0^\vee(1))(\mathbb C)\times \mathbb G_m(\mathbb C) ;\quad (k,z) \mapsto (zk, z).
	\]
	Note that $\Ucal_\infty=M(\CC)\cap \mathbf G(\RR)$.
	
	\begin{defn}\label{automform}
		Let $K$ be an open compact subgroup of $\mathbf G(\A_{\QQ,\fin})$
		and let $(\rho,W_\rho)$ be an irreducible algebraic representation of $M$ on a finite dimensional $\C$-vector space $W_\rho$.
		A smooth function
		\[
		f\colon \mathbf G(\A_\QQ)\rightarrow W_\rho
		\]
		is called an algebraic modular form of weight $\rho$ and level $K$ if $f$ satisfies the following conditions:
		\begin{itemize}
			\item[(1)]For each $\gamma\in \mathbf G(\Q)$, $g\in \mathbf G(\A_\QQ)$, $u_\infty\in \mathcal U_\infty$ and $k\in K$,
			we have $f(\gamma g u_\infty k)=\rho(u_\infty)^{-1}f(g)$.
			\item[(2)]Let $f_\infty$ denote the restriction of $f$ to $\mathbf G(\R)$. Then the map
			\[
			\mathsf X:=\mathbf G(\R)/ \mathcal U_{\infty}\rightarrow (\mathbf G(\R)\times W_\rho)/ \mathcal U_\infty; \quad g\mapsto (g,f_\infty(g))
			\] is a holomorphic section
			with respect to the canonical complex structure on $\mathsf X$ (cf.\ Lemma \ref{lemX}).
			
			\item[(3)] $f$ is slowly increasing.
		\end{itemize}
		We write $\mathscr A^{\mathrm{alg}}_\rho(K)$ for the $\C$-vector space of all
		algebraic modular forms of weight $\rho$ and level $K$.
	\end{defn}

	\begin{lem}[{\cite[Section 2.5]{Lan13}}]
		The topological space
		\[\mathrm{Sh}_{\mathbf G,K}:={\mathbf G}(\Q)\backslash {\mathbf G}(\A_\QQ)/\mathcal U_\infty K
		\]
		has a canonical structure of a smooth quasi-projective algebraic variety over a reflex field of the fixed Shimura datum.
		\label{compact}
	\end{lem}
	
	Let ${\mathrm{Sh}}^{\mathrm{min}}_{\mathbf G,K}$ denote
	the minimal compactification
	of ${\mathrm{Sh}}_{\mathbf G,K}$ constructed by Lan in \cite[Section 7.2]{Lan12}.
	Lan also defined an invariant $c_{\mathbf G}$ by
	\[
	c_{\mathbf G}:=\begin{cases}
		\mathrm{codim}(\partial{\mathrm{Sh}}^{\mathrm{min}}_{\mathbf G,K},{\mathrm{Sh}}^{\mathrm{min}}_{\mathbf G,K})\ &\text{if }\partial{\mathrm{Sh}}^{\mathrm{min}}_{\mathbf G,K}\neq \emptyset,\\
		\infty &\text {if } \partial{\mathrm{Sh}}^{\mathrm{min}}_{\mathbf G,K}= \emptyset.
	\end{cases}
	\]	
	and	considered the condition (LB) given by	
	\begin{align}c_{\mathbf G}>1 \tag{LB}
	\end{align}
	(cf.\ \cite[(2.1)]{Lan16}).
	Let $\mathcal E_\rho$ denote the holomorphic
	vector bundle
	\[
	\mathbf G(\Q)\backslash(\mathbf G(\A_\QQ)\times W_\rho)/\mathcal U_\infty K
	\]
	over ${\mathrm{Sh}}_{\mathbf G,K}$.
	\begin{prop}[Higher K\"ocher's principle {\cite{Lan16}}] \
		\begin{itemize}
			\item[(1)]
			The vector bundle $\mathcal E_\rho$ extends to
			any toroidal compactification ${\mathrm{Sh}}^{\mathrm{tor}}_{\mathbf G,K}$
			canonically. That extension is denoted by $\mathcal E_\rho^{\mathrm{tor}}$.
			\item[(2)]If $\mathbf G$ satisfies $\rm (LB)$, then any holomorphic section $f$
			can be extended to a holomorphic section of  $\mathcal E_\rho^{\mathrm{tor}}$
			over ${\mathrm{Sh}}^{\mathrm{tor}}_{\mathbf G,K}$ uniquely.
		\end{itemize}
		\label{Koecher}
	\end{prop}

As a corollary of the higher K\"ocher's principle,
any smooth function $f : \mathbf{G}(\AA_\QQ)\rightarrow W_\rho$ satisfying (1) and (2) in Definition \ref{automform} is slowly increasing and hence $f$ is an algebraic modular form, if $\mathbf G$ satisfies the condition $\rm (LB)$.

Recall that a smooth function
	\[
	f\colon \mathbf G(\A_\QQ)\to W_\rho
	\]
	satisfying the conditions (1) and (3)
	in Definition \ref{automform} is an automorphic form on
	$\mathbf G(\AA_\QQ)$ in the classical sense
	if $f$ is $K_\infty$-finite and $Z(\mathfrak g_{\C})$-finite,
	where $Z(\mathfrak{g}_\CC)$ is the center of the universal enveloping algebra of the complexification $\mathfrak{g}_\CC$ of the Lie algebra of $\mathbf G(\RR)$.
	We have the following.
	\begin{prop}
		Any algebraic modular form on $\mathbf G(\A_{\Q})$ is an automorphic form in the classical sense.
		\label{calssical} 
	\end{prop}
	\begin{proof}
		Let $f$ be an algebraic modular form on $\mathbf G(\A_\Q)$ of weight $\rho$ and level $K$ as in Definition \ref{automform}.
		By Definition \ref{automform} (1) and (3), it is enough to show that $f$ is $K_\infty$-finite and $Z(\mathfrak{g}_\C)$-finite.
		
		Let $K_\infty$ be the maximal compact subgroup of $\bfG(\RR)$ which we took as above, that is, $K_\infty$ is a maximal compact subgroup of $\mathbf G(\R)$ such that $Z_{\infty}K_\infty$ and $\mathcal U_\infty$ are commensurable (Lemma \ref{commensurable}).
			 Since $K_\infty\cap\mathcal U_\infty$ is a finite index subgroup of $K_\infty$, there exists
			 a normal subgroup $N$ of $K_\infty$ which is contained in $K_\infty\cap\mathcal U_\infty$. Let $k_1,\dots,k_m$ be elements of $K_\infty$ such that
			 \[
			 K_\infty=\coprod_{i=1}^mN k_i.
			 \]
			 For each $i \in \{1,\ldots,m \}$, define $f_i\colon \mathbf G(\A_{\Q})\to W_\rho$ by $f_i(g):=f(gk_i)$, ($g\in \mathbf G(\A_{\Q})$).
			 Let $V_i$ be the vector space spanned by $N f_i$. Then the vector space $V$ spanned by $K_\infty f$ coincides with $\sum_{i=1}^m V_i$. Moreover, $V_i$ coincides with the vector space spanned by $\rho(n)^{-1}f_i, (n\in N)$.
			 Let $e_1,\dots,e_d$ be a basis of $W_\rho$ with $d=\dim W_\rho$. For each $i \in \{1,\ldots,m \}$ and $j\in \{1,\ldots,d\}$, define $F_j^{(i)}\colon \mathbf G(\A_{\Q})\to \C$ by
			 \[
			 f_i(g)=\sum_{j=1}^{d} F_j^{(i)}(g)e_j, \qquad g \in \mathbf G(\A_\Q).
			 \]
			 For an element $n$ of $N$, write $\rho(n)^{-1}e_j=\sum_{k=1}^{d}a_{kj}e_k$. Then we have
			 \[
			 \rho(n)^{-1}f_i(g)=\sum_{j,k=1}^{d}a_{kj} F_j^{(i)}(g)e_k.
			 \]
			 This implies that $V_i$ is contained in the vector space spanned by $F_j^{(i)}(g)e_k, (j,k\in\{1,\dots,d\})$. Therefore, we have $\dim V_i\leq (\dim W_\rho)^2<+\infty$. Thus $V$ is finite dimensional.

		Hence we have only to prove the $Z(\mathfrak{g}_\C)$-finiteness of $f$.
		Since $Z_\infty$ acts on $f$ as scalars by Schur's lemma, we may consider
		the complexification $\mathfrak{g}_{0,\C}$ of the Lie algebra of $\mathbf G(\R)/Z_\infty$.
		Note the decomposition $\mathfrak{g}_{0,\C} = \mathfrak{p}^+ \oplus \mathfrak{k}_\C\oplus \mathfrak{p}^-$, where $\mathfrak{k}_\C$ is the complexification of the Lie algebra of $K_\infty$, and $\mathfrak{p}^+$ and $\mathfrak{p}^-$
		are the Lie subalgebra of $\mathfrak{g}_{0,\C}$ corresponding to the holomorphic tangent space and to the anti-holomorphic tangent space
		at the point $1_{{\mathbf G}(\R)}\mathcal U_\infty \in \mathsf X$, respectively.
		Then the holomorphy of $f$ implies $\mathfrak{p}^{-}f=0$.
		Let $\pi$ be the $\mathfrak{g}_{0,\C}$-module generated by $f$ via the differentials of right translations.
		A Cartan subalgebra of $\mathfrak{k}_\C$ is a Cartan subalgebra of $\mathfrak{g}_{0, \C}$
		by ${\rm rank}(\mathbf G(\R)/Z_\infty) = {\rm rank}(K_\infty)$,
		and its action on $\pi$ is semisimple.
		Thus $\pi$ is a lowest weight module with a lowest weight vector $f$.
		In particular, $\pi$ is a weight module such that each weight space of $\pi$ is finite dimensional by \cite[\S20.2, Theorem (c)]{Humphreys0}.
		Here we note that \cite[\S20.2, Theorem (c)]{Humphreys0} is stated for highest weight modules, and that lowest weight modules are treated similarly. Since each weight space of $\pi$ is preserved by $Z(\mathfrak{g}_{0,\C})$-action,
		any element of $\pi$ is $Z(\mathfrak{g}_{0,\C})$-finite (cf.\ \cite[\S1.1, Theorem (e)]{Humphreys}).
	\end{proof}
	
	\subsection{Integral structure on the space of algebraic modular forms arising from Fourier expansions}
	In this subsection, we suppose that any automorphic form $f$
	has a Fourier expansion. Namely, we suppose the following.
	\begin{assumption}There exists a maximal totally isotropic $\ZZ$-submodule $L_{\mathrm{MI}}$ of
		$L$ \emph{stable under the action of $\mathcal O$}.
		\label{totdeg}
	\end{assumption}
	We fix such an $L_{\mathrm{MI}}$ through this section. Let $\mathbf P$ be the parabolic subgroup of $\mathbf G$ corresponding to the flag $\mathscr F:=(0\subset L_{\mathrm{MI}}\subset L)$.
	Let $\mathbf U$ be the unipotent radical of $\mathbf P$
	and let $\mathbf G'$ be the Levi quotient of $\mathbf P$.
	We use the notation of \cite{Lan12} as follows.
	Let $X$ and $Y$ be free abelian groups defined by
	\[
	X:=\Hom_{\Z}(L_{\mathrm{MI}},\Z(1)),\quad Y:=L/L_{\mathrm{MI}}.
	\]  
	Let $\phi\colon Y\hookrightarrow X$ be an injection defined by the non-degenerate
	pairing
	\[
	\langle\ ,\ \rangle\colon X^\vee(1)\times Y=L_{\mathrm{MI}}\times L/L_{\mathrm{MI}}\rightarrow \Z(1).
	\]
	For a positive integer $N$, Lan defined a free abelian group $\mathbf S_N$ to be the maximal torsion free
	quotient of $X\otimes_{\Z}\frac{1}{N}Y$ on which
	both $\phi(y)\otimes y'=\phi(y')\otimes y$ and $bx\otimes \frac{1}{N}y=x\otimes b^\star \frac{1}{N}y$
	hold for all $x\in X,y,y'\in Y,b\in \mathcal O$ (\cite[Subsection 6.2.3]{Lan12}). Namely, set
	\[
	\mathbf S_N:=\left(\frac{X\otimes_{\Z}\frac{1}{N}Y}{(\phi(y)\otimes y'-\phi(y')\otimes y,bx\otimes \frac{1}{N}y-x\otimes b^\star \frac{1}{N}y|x\in X,y,y'\in Y,b\in \mathcal O)} \right)_{\mathrm{free}}.
	\]
	Here we use a notation similar to \cite[Subsection 3.6]{Lan13}.
	Then there exist natural isomorphisms
	\begin{equation}
		\label{isom}
		\mathbf U(\R)\cong\Hom_{\mathcal O}(X^\vee (1),Y_{\R})\cong \Hom_{\Z}(\mathbf S_N,\R(1)).
	\end{equation}
	In particular, $\mathbf U$ is abelian.
	More strongly, we have
	\begin{equation}
		\label{eq2}
		\mathbf U(\Q)\cap K(N)\cong\Hom_{\Z}(\mathbf S_N,\Z(1)),
	\end{equation}
	where \[
	K(N):=\mathrm{Ker}(\mathbf G(\widehat{\mathbb Z})\to \mathbf G(\mathbb Z/N\mathbb Z)).
	\]
	\begin{ex}\label{ex2}
		We see two examples which satisfy our assumption.
		For simplicity, we only consider the case where $N=1$.
		\begin{itemize}
			\item[(1)] $($Siegel case$)$
			We use the same notation as in Example $\ref{ex1}$ $($1$)$.
			Then, $L$ satisfies Assumption \ref{totdeg}.
			Indeed, any totally maximal isotropic subspace of $L$
			is stable under the action of $\mathbb Z=\mathcal O$.
			For example, we can take $L_{\mathrm{MI}}$ as
			the subspace consisting of vectors $^t(v_1,\cdots,v_g,0,\cdots,0)$.
			In this case, the homomorphism $\phi\colon Y\to X$ is an isomorphism
			and the abelian group $\mathbf S_1$ can be identified with the symmetric
			quotient of $X$ by $\phi$.
			Therefore, we have a natural identification
			\[
			\mathbf U(\mathbb R)=\{\text{quadratic forms on }X_{\mathbb R}\}=\mathrm{Sym}^2(\mathbb R^g).
			\]
			\item[(2)] $($``Hilbert case''$)$
			We use the same notation as in Example $\ref{ex1}$ $($2$)$.
			Then, the space
			\[
			L_{\mathrm{MI}}:=\mathfrak a\oplus\{0\}\subset L=\mathfrak a\oplus\mathfrak b
			\]
			is an $\mathcal O$-stable maximal totally isotropic subspace.
			Thus, this example also satisfies Assumption \ref{totdeg}.
			In this case, we have a natural isomorphism
			\[
			\mathbf S_1\cong \mathfrak a^\vee(1)\otimes_{\mathcal O}\mathfrak b\cong \mathfrak a^{-1}\mathfrak b\mathcal D_{F/\mathbb Q}^{-1}.
			\]
			Here, we use a fixed $2\pi\sqrt{-1}$ for the last isomorphism.
			Therefore, the unipotent group $\mathbf U$ is naturally isomorphic to 
			the affine group scheme associated to $\mathfrak a\mathfrak b^{-1}\mathcal D_{F/\mathbb Q}^{-1}$.
		\end{itemize}
	\end{ex}
	For a general open compact subgroup $K$, Lan constructed a free abelian group
	$\mathbf S_{K}$ satisfying
	\begin{equation}
		\label{eq3}
		\mathbf U(\Q)\cap K\cong\Hom_{\Z}(\mathbf S_{K},\Z(1))
	\end{equation}
	as a subquotient of $\Q\otimes_{\Z} X\otimes_{\Z} Y$.
	See \cite[Section 6.2]{Lan12} for the precise definition of this abelian group.
	
	For each $g\in \mathbf G(\A_{\QQ,\fin})$, new lattices $L_{\mathrm{MI}}^g$ and $L^g$ are defined by
	\[
	L_{\mathrm{MI}}^g:=g(\widehat \Z\otimes_{\Z} L_{\mathrm{MI}})\cap \Q\otimes_{\Z}L_{\mathrm{MI}},\quad
	L^g:=g(\widehat \Z\otimes_{\Z} L)\cap \Q\otimes_{\Z}L
	\]
	(\cite[Subsection 3.1]{Lan13}).
	This twisting by $g$ defines a flag $\mathscr F_g$ of $L^g$.
	It is easily checked that $L^g$ also defines a PEL-type $\mathcal O$-lattice
	and we can define an affine group scheme $\mathbf G^g$ over $\mathbb Z$
	as we defined $\mathbf G$ from $L$.
	Let $\mathbf P^g$ be the parabolic subgroup of $\mathbf G^g$ corresponding to
	$\mathscr F_g$ and let $\mathbf U^g$ be its unipotent radical.
	Then, we can define the free abelian group $\mathbf S_K^g$
	so that
	\[
	\mathbf U^g(\Q)\cap K\cong\Hom_{\Z}(\mathbf S^g_K,\Z(1)).
	\]
	In this paper, we call the flag $\mathscr F_g$ for some $g\in \mathbf G(\A_{\QQ,\fin})$
	a chosen \emph{standard $0$-dimensional rational boundary component},
	which is a representative of a rational boundary component in the sense of
	\cite[Definition 3.1.2]{Lan16}.
	Note that the flag $\mathscr F_g$ defines a rational boundary component of $X_0$ in the sense
	of \cite[III, Definitions 3.2 and 3.12]{AMRT}.
	\begin{ex}\label{ex3}
		We use the same notation as in Example \ref{ex2} and put $\mathscr F=\mathscr F_1$.
		\begin{itemize}
			\item[(1)] $\rm ($Siegel case$ \rm )$
			The symbols $\mathfrak H_n$ and $ \mathfrak D_n$ denote the Siegel upper half-space and the Siegel disk, respectively. Recall that $\mathfrak D_n$ is defined by
			\[
			\mathfrak D_n=\left\{X\in \mathrm{Sym}^2(\mathbb C^n)\middle |\ 1_n-X\ ^{\! t}\overline X>0\right\}.
			\]
			In  the Siegel case, $\mathsf X_0$ is naturally isomorphic to $\mathfrak H_n\cong \mathfrak D_n$ by the Cayley transformation.
			Then, the rational boundary component, which is defined by $\mathscr F$, of $\mathfrak D_n$ in $\overline{\mathfrak D_n}\subset
			\mathrm{Sym}^2(\mathbb C^n)$ is the singleton $\{1_n\}$.
			\item[(2)] $\rm ($\!``Hilbert case''$\rm )$ In this case,
			we have
			\[
			\mathsf X_0\cong \mathfrak H_1^d\cong \mathfrak D_1^d.
			\]
			Then, the rational boundary component, which is defined by $\mathscr F$, of $\mathfrak D_1^d$ in $\mathbb C^d$ is the singleton $\{(1,1,\ldots,1)\}$.
		\end{itemize}
	\end{ex}
	The Shimura variety $\mathrm{Sh}_{\mathbf G,K}$ has a description
	\[
	\mathrm{Sh}_{\mathbf G,K}=\amalg_{i\in I}\ \Gamma_i\backslash \mathsf X_0,
	\]
	where $\mathsf X_0$ is the connected component of $\mathsf X$ containing the identity $1$
	and $I$ is a certain finite subset of $\mathbf G(\mathbb A_{\QQ, \mathrm{fin}})$
	containing $1$
	(cf.\ \cite[(2.5.2)]{Lan12}).
	\begin{defn}Let $f$ be an element of $\mathscr A^{\mathrm{alg}}_\rho(K)$.
		The Fourier expansion of $f$ at the standard $0$-dimensional rational boundary component
		$\mathscr F_{g_i}$ is defined as an infinite sum
		\[
		f=\sum_{l\in\mathbf S_K^{g_i}}a_l(g_i,f)q_l,
		\]
		where $q_l\colon \mathbf U^{g_i}(\Q)\backslash \mathbf U^{g_i}(\A_\QQ)\rightarrow \C^\times$ is a smooth function corresponding to $l\in \mathbf S_{K}^{g_i}$
		and
		\[
		a_l(g_i,f):=\int_{\mathbf U^{g_i}(\Q)\backslash \mathbf U^{g_i}(\A_\QQ)}f(g_iu)q_l(u)^{-1}du\in W_\rho.
		\]
		\label{fourierexp}
	\end{defn}
	\begin{ex}
		\label{ex4}We use the same notation as in Example \ref{ex3}.
		Let $f$ be an element of $\mathscr A^{\mathrm{alg}}_\rho(\mathbf G(\widehat{\mathbb Z}))$.
		\begin{itemize}
			\item[(1)]{\rm(}Siegel case{\rm)} In this case, $f$ defines a holomorphic function
			\[
			F\colon \mathfrak H_n\to W_\rho
			\]
			which is invariant under the translations by the elements of \[
			\mathbf U(\widehat{\mathbb Z})\cap\mathbf U(\mathbb Q)=\Hom_{\mathbb Z}(\mathbf S_1,\mathbb Z).
			\]
			The archimedean part of the smooth function $q_l$ for $ l\in \Sym^2(\mathbb Z^n)$,
			is given by
			\[
			q_l(X)=\exp\left(2\pi\sqrt{-1}\mathrm{tr}(Xl)\right),\quad X\in \Sym^2(\mathbb R^n).
			\]
			Thus the Fourier expansion defined above coincides with the classical Fourier expansion of Siegel modular forms. 
			Here, we identify $\Sym^2(\mathbb Z^n)$ with its dual by the pairing
			\[
			(X,Y)\mapsto \mathrm{tr}(XY).
			\]
			\item[(2)]{\rm(}``Hilbert case''{\rm)}
			In this case, $f$ defines a holomorphic function
			\[
			F\colon \mathsf X_0\cong \mathfrak H_1^d\to W_\rho\cong \mathbb C
			\]
			which is invariant under the translations by the elements of $\mathbf S_1^\vee$.
			The archimedean part of the smooth function $q_l$ for $l\in \mathbf S_1\cong \mathfrak a^{-1}\mathfrak b\mathcal D_{F/\mathbb Q}^{-1}$ is given by
			\[
			q_l(\tau)=\exp\left(2\pi\sqrt{-1}\mathrm{Tr}_{F/\mathbb Q}(l\tau)\right),\quad \tau\in F_{\mathbb C}.
			\]
			Therefore, the Fourier expansion of $f$ also coincides with the classical one.
		\end{itemize}
	\end{ex}
	We can regard $f$ as a tuple of holomorphic functions $\{f_i\colon \mathsf X_0\rightarrow W_\rho\}_{i\in I}$ satisfying a certain modularity with respect to each $\Gamma_i$.
	Moreover, as explained in \cite[Section 2.5]{Lan16},
	$\mathsf X_0$ can be embedded into a Siegel upper half-space.
	Then, as the example above, the Fourier expansion of $f$ at $\mathscr F_{g_i}$ coincides with the Fourier expansion of $f_i$
	at the chosen standard point at infinity of the Siegel upper half-space.
	\begin{defn}
		Let $R$ be a subring of $\mathbb C$ containing $\mathcal O_F$
		and let $W_{\rho,0}$ be an $R$-lattice of $W_\rho$ such that $W_{\rho,0}\otimes_{\mathcal O_F}F$ is stable under the action of $M$.
		Here, this lattice is not necessarily stable under the action of $M$.
		We choose a finite set $\{g_i\}_i\subset \mathbf G(\mathbb A_{\mathbb Q,\mathrm{fin}})$ so that the corresponding connected components of Shimura varieties cover the whole space, and define an $R$-submodule $\mathscr A^{\mathrm{alg}}_\rho(K;R)$
		of $\mathscr A^{\mathrm{alg}}_{\rho}(K)$ by
		\[
		\mathscr A^{\mathrm{alg}}_\rho(K;R):=\{f\in\mathscr A^{\mathrm{alg}}_\rho(K) \mid a_{l}(g_i,f)\in W_{\rho,0} \text{ for all $l$ and $i$} \}.
		\]
		\label{Rautom}
	\end{defn}
	The main result of this section is the following.
	\begin{thm}\label{integral model}
		Under Assumption \ref{totdeg}, there exists a number field $F$ of finite degree
		such that
		\[
		\mathscr A^{\mathrm{alg}}_\rho(K;\mathcal O_{F})\otimes_{\mathcal O_{F}}\C=\mathscr A^{\mathrm{alg}}_\rho(K)
		\]
		for an arbitrary $\mathcal O_F$-lattice $W_{\rho,0}$ of $W_\rho$.
	\end{thm}

	\begin{lem}Let $K'$ be a subgroup of $K$ of finite index.
		Then, if the assertion of Theorem $\ref{integral model}$
		holds for $K'$, then the assertion holds for $K$.
		\label{lem0}
	\end{lem}
	\begin{proof}Consider the following inclusion relations:
		\[
		\xymatrix{
			\mathscr A^{\mathrm{alg}}_\rho(K;\mathcal O_{F})\ar@{^{(}->}[d]\ar@{^{(}->}[rr]& &\mathscr A^{\mathrm{alg}}_\rho(K)\ar@{^{(}->}[d]\\
			\mathscr A^{\mathrm{alg}}_\rho(K';\mathcal O_{F})\ar@{^{(}->}[rr]& &\mathscr A^{\mathrm{alg}}_\rho(K')	}
		\]
		Here we retake $\{g_i\}_i$ for the smaller open compact subgroups of $K'$ if we need.
		Then, one can check that this diagram is Cartesian. Hence, we have the conclusion
		by a diagram chasing.
	\end{proof}
	We recall the moduli problem attached to the given PEL-type $\mathcal O$-lattice
	$(L,\langle \ ,\ \rangle,h_0)$ and a neat open compact subgroup $K$
	of $\mathbf G(\A_{\QQ,\fin})$.
	According to Lemma \ref{lem0},
	we may suppose that $K=K(N)$ where $N$ is greater than or equal to three
	(cf.\ \cite[Remark 1.1.4]{Lan12}).
	Recall that the open subgroup $K(N)$ of $\mathbf G(\widehat{\Z})$
	is defined by
	\[
	K(N):=\mathrm{Ker}\left(\mathbf G(\widehat \Z)\rightarrow \mathbf G(\Z/N\Z)\right).
	\]
	Let $F$ be a number field of finite degree contained in $\mathbb C$ which is ``sufficiently large'',
	for example, $F$ contains the reflex field of the given PEL-type $\mathcal O$-lattice.
	We consider the following moduli problem:
	\begin{defn}
		We define $\mathsf M_{K(N)}=\mathsf M_N$
		as the moduli stack of $(A,\lambda,\iota,\alpha)$, where
		\begin{itemize}
			\item[-]$A$ is an abelian scheme over a scheme $S$,
			\item[-]$\lambda\colon A\rightarrow A^\vee$ is a polarization,
			\item[-]$\iota\colon \mathcal O\rightarrow \End_S(A)$ defines an $\mathcal O$-structure
			of $(A,\iota)$ such that $\underline{\mathrm{Lie}}(A/S)$ satisfies the determinantal condition $($\cite[Section 1.1.2]{Lan12}$)$ given by $(L_{\R},\langle \ ,\ \rangle,h_0)$,
			\item[-]$\alpha_N\colon (L/NL)_S\cong A[N]$ is an integral level $K(N)$-structure of $(A,\lambda,\iota)$
			in the sense of \cite[Definition 1.3.6.2]{Lan13}.
		\end{itemize}
		\label{moduli}
	\end{defn}
	\begin{prop}[{\cite[Theorem 1.4.1.12 and Corollary 7.2.3.10]{Lan13}}]
		The restriction of the stack $\mathsf M_N$ to the category of $F$-schemes
		is represented by a quasi-projective smooth $F$-variety.
		\label{representativity}
	\end{prop}
	By abuse of notation, we write $\mathsf M_N$ for that scheme.
	Furthermore, we take $F$ sufficiently large so that every geometrically connected component of $\mathsf M_N$ is
	defined over $F$.
	\begin{prop}[{\cite[Lemma 2.3.2 and Lemma 2.4.4]{Lan13}}]
		The Shimura variety $\mathrm{Sh}_{\mathbf G,K(N)}$
		is canonically isomorphic to an open and closed subscheme of $\mathsf M_N\times_{\mathrm{Spec}(F)}\mathrm{Spec}(\C)$.
		\label{component}
	\end{prop}
	Let $\mathsf M_N^0$ denote the corresponding open and closed subscheme of $\mathsf M_N$.
	\begin{lem}[{cf.\ \cite[Definition 5.2.5 and Lemma 5.3.2]{Lan13}}]The holomorphic vector bundle $\mathcal E_\rho$ over $\mathrm{Sh}_{\mathbf G,K(N)}$
		is algebraizable.
		Moreover, if $F$ contains a field of definition of the $B_{\C}$-module $V$,
		then the vector bundle $\mathcal E_\rho$ is defined over $\mathsf M_N^0$.
	\end{lem}
	We use the same symbol $\mathcal E_\rho$ for that vector bundle over $\mathsf M_N^0$.
	\begin{thm}
		[Higher K\"ocher's principle {\cite{Lan16}}]
		Suppose the condition $\rm (LB)$. Then, we have
		\begin{equation}
			\begin{split}
				\mathscr A^{\mathrm{alg}}_\rho(K)&=H^0(\mathsf M_N^0\times_{\mathrm{Spec}(F)}\mathrm{Spec}(\C),\mathcal E_\rho)\\
				&=H^0(\mathsf M_N^{0,\mathrm{tor}}\times_{\mathrm{Spec}(F)}\mathrm{Spec}(\C),\mathcal E^{\mathrm{tor}}_\rho).
			\end{split}
		\end{equation}
		Here $\mathsf M_N^{0,\mathrm{tor}}$ is any toroidal compactification of $\mathsf M_N^0$
		constructed by Lan \cite{Lan12}.
		\label{Koe}
	\end{thm}
	\begin{cor}	\label{fin dim}
		The $\C$-vector space $\mathscr A^{\mathrm{alg}}_\rho(K)$
		is finite dimensional.
	\end{cor}
	\begin{proof}
		The $\CC$-vector space
		$H^0(\mathsf M_N^{0,\mathrm{tor}}\times_{\mathrm{Spec}(F)}\mathrm{Spec}(\C),\mathcal E^{\mathrm{tor}}_\rho)$
		is finite dimensional by the higher direct image theorem.
		Thus we have the conclusion by Theorem \ref{Koe}.
	\end{proof}
	\begin{cor}
		We have
		\[
		\mathscr A^{\mathrm{alg}}_\rho(K)=H^0(\mathsf M_N^{0,\mathrm{tor}},\mathcal E^{\mathrm{tor}}_\rho)\otimes_{F}\C.
		\]
	\end{cor}
	\begin{proof}This is an elementary consequence
		of Theorem \ref{Koe} and the flat base change theorem of coherent sheaves over schemes.
	\end{proof}
	Let us recall an algebraic interpretation of the Fourier expansion.
	For each $i$, put $X_i:=\Hom_{\Z}(L^{g_i}_{MI},\Z(1))$ and $Y_i:=L^{g_i}/L_{\mathrm{MI}}^{g_i}$.
	Let $\phi_i\colon Y_i\hookrightarrow  X_i$ be an injective homomorphism induced by the
	given pairing $\langle \ ,\ \rangle$.
	We define the torus $E_i$ by
	\[
	E_i=\mathrm{Spec}(\mathbb Z[X_i])=X_i^*\otimes\mathbb G_m.
	\]
	Let $\mathbf P_i\subset \mathbf S_{N,\mathbf R}^{g_i,\vee}$
	be the cone defined by the ``positivity condition'' (\cite[Section 6.2.5]{Lan12}).
	For example, in the Siegel case (Example \ref{ex1} (1) and Example \ref{ex2} (1)),
	$\mathbf P_i$ is the subset consisting of positive semi-definite
	quadratic forms on $X_{i,\mathbf R}$, which is written as $C(X)$ in \cite[Section 2.1]{FaltingsChai}.
	Let $\sigma\subset \mathbf P_i^0$ be any non-degenerate rational polyhedral cone
	and let $E_i\hookrightarrow \Xi_i(\sigma)$ be the corresponding
	torus embedding, where
	\[
	\Xi_i(\sigma)=\mathrm{Spec}(\mathbb Z[\mathbf S_{N}^{g_i}\cap \sigma^\vee]).
	\]
	Here $\mathbf P_i^0$ is the interior of $\mathbf P_i$.
	We define $\overline{\Xi}_{i,\mathbb C}(\sigma)$ by
	\[
	\overline{\Xi}_{i,\mathbb C}(\sigma)=\mathrm{Spec}(\mathbb C[\![\mathbf S_{N}^{g_i}\cap \sigma^\vee]\!]).
	\]
	Then, the abelian group $Y_i$ can be regarded as a period of
	the torus $X_i^\vee\otimes\mathbb G_m=\mathrm{Spec}(\mathbb Z[\mathfrak X^x\mid x\in X])$ over $K_i:=\mathrm{Frac}(\mathbb C[\![\mathbf S_{N}^{g_i}\cap \sigma^\vee]\!])$
	by the equation
	\[
	\mathfrak X^x(y)=[y\otimes x]\in \mathbb C[\![\mathbf S_{N}^{g_i}\cap \sigma^\vee]\!].
	\]
	Then, as any element of $\sigma$ is ``positive-definite'',
	$\phi_i\colon Y_i\to X_i$ induced by $\langle\ ,\ \rangle$ is a polarization of
	$Y_i$ in the sense of Mumford.
	Therefore, we can apply Mumford's construction for
	$(Y_i\subset X_i^\vee\otimes\mathbb G_m,\phi_i)$ over $\mathbb C[\![\mathbf S_{N}^{g_i}\cap \sigma^\vee]\!]$
	and let
	\[
	(G_i\to \overline{\Xi}_{i,\mathbb C}(\sigma),\lambda)
	\] 
	denote the resulting polarized semi-abelian variety.
	As Mumford's construction is functorial, $\mathcal O$ acts on this polarized semi-abelian variety
	and $\iota_i$ denotes this action.
	Now, we fix a splitting
	\[
	\epsilon_i\colon L^{g_i}\cong L_{\mathrm{MI}}^{g_i}\oplus L^{g_i}/L_{\mathrm{MI}}^{g_i}
	\]
	as $\mathcal O$-modules.
	Then, by using this, we have a natural level structure
	\[
	\alpha_{i,N}\colon L^{g_i}/NL^{g_i}\cong G_{i,\eta_i}[N]
	\]
	over the generic point $\eta_i$ of $\overline{\Xi}_i(\sigma)$.
	Let
	\begin{equation}
		\label{classifying}
		\mathrm{cl}_i\colon \eta_i=\Spec(K_i)\to \mathsf M_N^0
	\end{equation}
	denote the classifying morphism defined by the tuple
	\[
	(G_{i,\eta_i},\lambda,\iota_i,\alpha_N).
	\]
	One of important applications of the main comparison theorem \cite[theorem 4.1]{Lan13}
	is the following.
	\begin{thm}[{\cite[Theorem 5.3.5]{Lan16}}]
		The homomorphism
		\[
		\mathrm{cl}_i^*\colon \mathcal A^{\rm alg}_\rho(K(N))\to K_i\otimes_{\mathbb C}W_\rho
		\] induced by the pull-back via the classifying morphism {\rm(}\ref{classifying}{\rm)}
		coincides with the Fourier expansion in Definition \ref{fourierexp}.
	\end{thm}
	\begin{rem}
		According to the construction of Lan in \cite{Lan16}, the image of the homomorphism $\mathrm{cl}_i^*$ is contained in
		\[
		\mathbb C[\![\mathbf P^{\vee}_i\cap \mathbf S_{K(N)}]\!]\otimes_{\mathbb C}W_\rho.
		\]
		Moreover, this homomorphism	does not depend on the choice of the rational polyhedral cone $\sigma$.
		This can be proved by the very same manner as in \cite[Chapter V, Section 1]{FaltingsChai}.
	\end{rem}
	The following proposition is a key point of the proof of Theorem \ref{integral model}.
	\begin{prop}Let $e_1,\dots,e_r$ be a free basis of $Y$ and let
		$l$ be an element of the ring $\mathbb Z[\![\sigma^\vee\cap \mathbf S_{K(N)}^{g_i}]\!]$ defined by
		\[
		l=\prod_{i=1}^r[e_i\otimes\phi_i(e_i)].
		\]
		Then, the point $\mathrm{cl}_i\colon \eta_i\to \mathsf M_N^0$ is over a point
		\[
		\mathrm{Spec}\left(\mathbb Z[\![\sigma^\vee\cap \mathbf S_{K(N)}^{g_i}]\!]\left[\frac{1}{l}\right]\otimes_{\mathbb Z}\mathcal O_F\left[\mu_N,\frac{1}{N}\right]\right)\to \mathsf M_N^0,
		\]
		where $\mu_N$ denotes the set of all $N$th roots of unity in $\CC$.
		\label{keyprop}
	\end{prop}
	\begin{proof}
		It is sufficient to show that the object $(G_{i,\eta_i},\lambda,\iota_i,\alpha_N)$
		is defined over the ring $\mathbb Z[\![\sigma^\vee\cap \mathbf S_{K(N)}^{g_i}]\!]\left[\frac{1}{l}\right]\otimes_{\mathbb Z}\mathcal O_F\left[\mu_N,\frac{1}{N}\right]$. This is already proved in Lemma \ref{Mumordconstructionrefinement}.
	\end{proof}

	\begin{cor}\label{cor1}
		The image $\mathrm{cl}^*_i$ is contained in
		$(\mathbb Z[\![\sigma^\vee\cap \mathbf S^{g_i}_{K}]\!])\otimes_{\mathbb Z}\mathcal O_F[\mu_N,1/N])\otimes_{\mathcal O_F}W_\rho$.
	\end{cor}
	\begin{proof}
		This is a direct consequence of Proposition \ref{keyprop}.
	\end{proof}
	
	\begin{proof}[Proof of Theorem $\ref{integral model}$]For each $f\in \mathscr A^{\mathrm{alg}}_\rho(K(N))$, let $f_i$ be the restriction of $f$ to
		$\mathsf X_i$. By applying Corollary \ref{cor1}, we have the Fourier expansion
		$f_i=\sum_{1\leq j\leq \dim(\rho)} a_j f_{j}\otimes w_{j}$
		for some $a_j\in F(\mu_N)$, $f_{j}\in \Z[\![\sigma^\vee\cap \mathbf S^{g_i}_{K(N)}]\!]$, and $w_{j}\in W_{\rho}$.
		This implies that the $\ZZ$-submodule of $W_\rho$ generated by the Fourier coefficients of $f_i$
		coincides with $\sum_{1\leq j\leq \dim(\rho)} \Z a_jw_{j}$, and hence it is finitely generated. This completes the proof.
	\end{proof}

	\begin{rem}
		In some special cases of Siegel modular forms,
		the existence of an integral structure as in Theorem \ref{integral model} can be proved by using Ibukiyama's differential operators {\rm (}\cite{Ibukiyama}{\rm)} and linearized pullback formulas. Such a method was exposited by B\"ocherer \cite{Boecherer}.
		His method is valid for vector-valued Siegel cusp forms of general degree and level $1$ with trivial nebentypus,
		under the assumption that the weight is large enough for using Ibukiyama's differential operators.
		In \cite{Boecherer}, B\"ocherer also announced that his method should work
		for non-cusp forms of general level.
		Contrary to this, Theorem \ref{integral model} includes all the cases of general holomorphic Siegel modular forms.
	\end{rem}

	\section*{Acknowledgements}
	The authors would like to thank Masao Tsuzuki for a lot of discussions and for giving them a suggestion on applications to the growth of cuspidal cohomologies. 
	They also would like to thank Hidenori Katsurada for a lot of fruitful comments on the integrality of Hecke eigenvalues for Siegel modular forms,
	Takashi Hara for his suggestion on a generalization of the early manuscript,
	and Satoshi Wakatsuki for discussion on equidistribution results for Siegel modular forms, respectively.
	Thanks are due to Masanobu Kaneko for giving them opportunities to discuss this work at Kyushu University. The first author was supported by
	Grant-in-Aid for JSPS Fellows (JP17J01827) and  JSPS KAKENHI Grant Number (JP23K03069).
	The second author was supported by Grant-in-Aid for Research Activity Start-up (JP19K21025)
	and Grant-in-Aid for Early-Career Scientists (JP20K14298).
	

\end{document}